\documentclass[11pt,a4paper]{article}
\usepackage[utf8]{inputenc}
\usepackage[english]{babel}%remembertwochangestoactivate
\usepackage{amsfonts} 
\usepackage{amsmath}
\usepackage{amssymb}
\usepackage{amsthm}%%%Proof environment
\usepackage{mathrsfs}%%%Mathscr riccioli
\usepackage{microtype}%%%secretpackagemarcello
\usepackage{marvosym}%%%always BEFORE wasysym or arxiv crashes
\usepackage{wasysym} %%%diameter=varnothing=emptyset
\usepackage{stmaryrd}
\usepackage{dsfont}%%% alternative to previous 1, 0
\usepackage{enumerate}
\usepackage{multicol,multirow}%%%column(s)
\usepackage{graphicx}%%%pics
\usepackage{xcolor}
\usepackage{authblk}%%%multiauthors
\usepackage{url}%%%size of urls. default was {\scriptsize}

\usepackage{hyperref}%%%linklink
\hypersetup{
	colorlinks=false,
	linkcolor=blue,
	urlcolor=red,
	linktoc=all}

\usepackage{tikz} %%%TikZ
\usetikzlibrary{decorations.pathreplacing}
\usetikzlibrary{decorations.pathmorphing}
\usetikzlibrary{calc}
\usetikzlibrary{decorations.pathreplacing}
\usetikzlibrary{decorations.pathmorphing}
\usetikzlibrary{arrows}
\usetikzlibrary{patterns,fadings}
\usepackage[all,cmtip]{xy}%%%xy alternative to tikz-cd
\newcommand{\Hilb}{\mathcal{H}}%myHilb
\newcommand{\eps}{\text{$\varepsilon$}}%braiding
\newcommand{\oneop}{\mathds{1}}
\newcommand{\Rot}{\mathop{\mathsf{Rot}}}

\newcommand{\Mob}{\mathsf{M\ddot ob}}

\newcommand{\Sone}[1][]{\mathbb{S}^{1#1}}

\newcommand{\C}{\mathcal{C}}
\newcommand{\Z}{\mathcal{Z}}
\renewcommand{\O}{\mathcal{O}}

\newcommand{\cU}{\mathcal{U}}
\newcommand{\cV}{\mathcal{V}}
\newcommand{\cW}{\mathcal{W}}

\newcommand{\K}{\mathcal{K}}

\newcommand{\cD}{\mathcal{D}}

\newcommand{\cP}{\mathcal{P}}

\newcommand{\cK}{\mathcal{K}}
\newcommand{\A}{\mathcal{A}}
\newcommand{\B}{\mathcal{B}}

\newcommand{\M}{\mathcal{M}}

\newcommand{\N}{\mathcal{N}}

\newcommand{\RR}{\mathbb{R}}

\newcommand{\RRbar}{\overline{\mathbb{R}}}
\newcommand{\CC}{\mathbb{C}}

\newcommand{\ZZ}{\mathbb{Z}}

\DeclareMathOperator{\End}{End}
\DeclareMathOperator{\Ind}{Ind}

\DeclareMathOperator{\Hom}{Hom}

\DeclareMathOperator{\DHR}{DHR}

\DeclareMathOperator{\Vir}{Vir}

\DeclareMathOperator{\Ad}{Ad}

\DeclareMathOperator{\id}{id}

\def\III{{I\!I\!I}}%typeIII khr
\newcommand{\op}{\mathrm{op}}

\newcommand{\lqq}{\lq\lq}

\newcommand{\e}{\mathrm{e}}

\usepackage{xspace}
\DeclareRobustCommand{\eg}{e.g.\@\xspace}
\DeclareRobustCommand{\cf}{cf.\@\xspace}
\DeclareRobustCommand{\ie}{i.e.\@\xspace}

\DeclareRobustCommand{\Sec}{Sec.\@\xspace}
\DeclareRobustCommand{\Prop}{Prop.\@\xspace}
\DeclareRobustCommand{\Lem}{Lem.\@\xspace}
\DeclareRobustCommand{\Cor}{Cor.\@\xspace}
\DeclareRobustCommand{\Thm}{Thm.\@\xspace}
\DeclareRobustCommand{\Ch}{Ch.\@\xspace}
\DeclareRobustCommand{\App}{App.\@\xspace}

\DeclareRobustCommand{\Def}{Def.\@\xspace}
\DeclareRobustCommand{\Rmk}{Rmk.\@\xspace}
\DeclareRobustCommand{\Eq}{Eq.\@\xspace}

\DeclareRobustCommand{\rhs}{r.h.s.\@\xspace}
\DeclareRobustCommand{\lhs}{l.h.s.\@\xspace}
\DeclareRobustCommand{\resp}{resp.\@\xspace}
\makeatletter
\DeclareRobustCommand{\etc}{%
    \@ifnextchar{.}%
        {etc}%
        {etc.\@\xspace}%
}
\makeatother

\newcommand{\Cstar}{$C^\ast$\@\xspace}

\def\u1net{{\A_\RR}}

\theoremstyle{plain}
\newtheorem{theorem}{Theorem}[section]
\newtheorem{corollary}[theorem]{Corollary}
\newtheorem{lemma}[theorem]{Lemma}
\newtheorem{proposition}[theorem]{Proposition}

\theoremstyle{definition}
\newtheorem{definition}[theorem]{Definition}

\theoremstyle{remark}
\newtheorem{example}[theorem]{Example}
\newtheorem{remark}[theorem]{Remark}

 \newcommand{\tikzmath}[2][\squarescale]
	{\vcenter{\hbox{
	\begin{tikzpicture}
	[scale=#1]#2
	\end{tikzpicture}}}}

\begin{document}

\title{\huge Infinite index extensions of local nets and defects}
%\vspace{-6mm}

\author[1]{\Large Simone Del Vecchio}
\author[1]{\Large Luca Giorgetti}
%%%\author[1]{Author \thanks{E-mail}}
\affil[1]{\normalsize Dipartimento di Matematica, Universit\`a di Roma Tor Vergata\\

Via della Ricerca Scientifica, 1, I-00133 Roma, Italy\\

{\tt delvecch@mat.uniroma2.it}\\ 
{\tt giorgett@mat.uniroma2.it}}
%%%nolastcomma\renewcommand\Authands{ and }

\date{}

\maketitle

\begin{abstract}
Subfactor theory provides a tool to analyze and construct extensions of Quantum Field Theories, once the latter are formulated as local nets of von Neumann algebras. We generalize some of the results of \cite{LoRe95} to the case of extensions with infinite Jones index. This case naturally arises in physics, the canonical examples are given by global gauge theories with respect to a compact (non-finite) group of internal symmetries. Building on the works of Izumi, Longo, Popa \cite{ILP98} and Fidaleo, Isola \cite{FiIs99}, we consider generalized Q-systems (of intertwiners) for a semidiscrete inclusion of properly infinite von Neumann algebras, which generalize ordinary Q-systems introduced by Longo \cite{Lon94} to the infinite index case.
We characterize inclusions which admit generalized Q-systems of intertwiners and define a braided product among the latter, hence we construct examples of QFTs with defects (phase boundaries) of infinite index, extending the family of boundaries in the grasp of \cite{BKLR16}.
\end{abstract}

%%%
\section{Introduction}
%%%

The study of \emph{extensions} in relativistic Quantum Field Theory (QFT) is well-motivated in several respects. Gauge theory, for instance, provides examples of extensions where a theory of (anti)commuting fields obeying Bose/Fermi statistics, equipped with a gauge group symmetry, contains a subtheory generated by gauge invariant (observable) fields. The former can be viewed as an extension of the latter, and similarly any intermediate theory gives rise to a smaller extension of the observable theory. Defects and boundaries can also be described by extensions, where different types of bulk fields (depending on their relative spacetime localization with respect to a certain \lqq defect" line or hypersurface) generate extensions of a common subtheory which contains, for example, the components of the stress-energy tensor that are conserved across the boundary. Extensions also appear in classification instances of QFTs, where all the theories belonging to a certain family share a common subtheory (dictated, \eg, by spacetime symmetry), hence the classification problem can be turned into a classification of extensions. This is the case, for example, in chiral Conformal Field Theory (CFT) where the Fourier modes of the conformal stress-energy tensor necessarily obey the commutation relations of the Virasoro algebra at a fixed value of the central charge parameter. Lastly, the analysis of extensions can be used to construct new examples of QFTs. Starting from some theory, if one can write it as a non-trivial subtheory extended by a certain family of generators, then new theories can be constructed by suitably manipulating the generators and their commutation relations, compatibly with locality, and leaving the subtheory untouched.

All of these different situations and problematics have a model-independent and mathematically rigorous formulation in the Algebraic approach to QFT (\emph{AQFT}) due to Haag and Kastler, see \cite{HaagBook}. 

Global gauge theories have been tackled since the early works of \cite{DHR69I}, \cite{DHR69II}, \cite{DoRo72}, culminating in \cite{DHR71}, \cite{DHR74} and \cite{DoRo90} in particular, where it is shown that every theory of local observables arises as gauge group fixed points of a bigger field theory, obeying (anti)commutation relations and equipped with a (global) gauge group symmetry. Both the gauge group and the field extension are intrinsically determined by the local observables (hence dictated by locality, \ie, Einstein's causality). Intermediate extensions of gauge group fixed points (in 3+1 spacetime dimensions) have been studied in \cite{CDR01}. In the chiral CFT setting (1 spacetime dimension) gauge group fixed points (also called orbifold CFTs) appear in \cite{Xu00orb}, \cite{Xu05}, \cite{Mue05}, and have been generalized to finite hypergroup fixed points by \cite{Bis17} (generalized orbifold CFTs), where the information about gauge invariance contained in the conditional expectation is expressed by an hypergroup action via completely positive (CP) maps. Intermediate chiral extensions have been analysed by \cite{Lon03} and \cite{Xu14}.
Defects and boundaries have been studied with AQFT methods in recent works by \cite{BKL15}, \cite{BKLR16}, \cite{BiRe16}, see also \cite[\Ch 5]{BKLR15} where the main mathematical tools to construct and classify boundary conditions are developed. This analysis of defects and boundaries in QFT has been our initial motivation for the work presented in this article. Lastly, again using extensions, the classification of all chiral CFTs with central charge $c<1$ has been achieved in \cite{KaLo04}.

In the Haag-Kastler formalism, a (local) quantum field theory is described by a net of local algebras $\{\O\mapsto\A(\O)\}$, see \cite{Reh15}, \cite{HaMue06} for self-contained introductions. Local algebras $\A(\O)$ are assumed to be von Neumann algebras on the vacuum Hilbert space, and they typically turn out to be factors (hyperfinite and of type $\III_1$ in the classification of Connes \cite{Con73}, \cite{Haa87}). Hence an extension of QFTs $\{\A\subset\B\}$ is naturally described by a family of \emph{subfactors} indexed by spacetime regions $\O$ (\eg, double cones in Minkowski space or bounded intervals on the line). It was in the work of Longo and Rehren \cite{LoRe95} (indeed titled \lqq Nets of subfactors") that it became clear how to use subfactor theory as a tool to classify and construct extensions in QFT. Their main idea is to exploit the notion of Q-system, due to \cite{Lon94}, for nets of subfactors, in order to relate coherent families of conditional expectations (which respect the net structure) to coherent families of (dual) canonical endomorphisms \cite{Lon87} (which turn out to be the restrictions to different spacetime regions $\O$ of a unique global DHR endomorphism $\theta$ of $\{\A\}$). The family of conditional expectations generalizes the notion of global gauge symmetry, while the DHR endomorphism $\theta$ represents the (reducible) vacuum representation of the bigger theory $\{\B\}$ once restricted to $\{\A\}$. 

Mathematically speaking, the theory of subfactors plays a prominent role in the panorama of Operator Algebras since the work of Jones \cite{Jon83}. He established a notion of index for subfactors, which is an invariant (hence opened the way to classification questions) and surprisingly quantized for values between 1 and 4 (Jones' rigidity theorem). Since then, the major efforts have been devoted to the study of finite index (finite depth) subfactors and a complete classification has been achieved for subfactors with index at most $5+\frac{1}{4}$ \cite{JMS14}, \cite{AMP15}, using techniques of \cite{Pop95-1} and \cite{Jon99}. At the same time, the analyses of QFT extensions \cite{LoRe95} and of theories with defects and boundaries \cite{BKLR16} cover the \emph{finite} index case only, both being based on the notion of Q-system (which is tightly connected to the existence of conjugate morphisms $\bar\iota$ of the inclusion morphism $\iota:\N\hookrightarrow\M$ for a subfactor $\N\subset\M$, hence to the finiteness of the dimension of $\iota$ in the sense of \cite{LoRo97}).

In this article, building on the notion of Pimsner-Popa basis for an inclusion of von Neumann algebras, see \cite{PiPo86}, \cite{Pop95}, and on the characterization of the canonical endomorphisms given by Fidaleo and Isola in \cite{FiIs99}, we reformulate the results on QFT extensions of \cite[\Sec 4]{LoRe95} in the finite index case, and generalize them to \emph{infinite} index extensions, see Section \ref{sec:gennetQsysofinterts}. This case naturally appears in physical situations, \eg, if we consider global gauge theories with respect to a compact \emph{non-finite} group of internal symmetries.

In order to do so, we first adopt the notion of \emph{generalized Q-system}, due to \cite{FiIs99}, see Definition \ref{def:genQsys}, and then consider more special \emph{generalized Q-systems of intertwiners}, see Definition \ref{def:genQsysofinterts} and \ref{def:unitalqsys}. The latter can be thought, roughly speaking, as \Cstar Frobenius algebra-like objects with possibly infinitely many comultiplications, see Remark \ref{rmk:cfqsys}, \ref{rmk:discretespecialpipo} and \cf \cite[\Sec 3.1]{BKLR16}.

Any \emph{semidiscrete} inclusion of (properly infinite, with separable predual) von Neumann algebras $\N\subset\M$, \ie, an inclusion endowed with a faithful normal conditional expectation $E:\M\rightarrow\N$, admits a generalized Q-system. Vice versa, from any generalized Q-system one can (re)construct the bigger algebra $\M$ and the conditional expectation $E$, see \cite[\Thm 4.1]{FiIs99}. An advantage of using generalized Q-systems (in the finite index case as well) is that no factoriality or irreducibility assumption on the inclusion is needed along the way. This enhanced flexibility is particularly desirable in the study of boundary conditions, see comments after \cite[\Thm 4.4]{BKLR16}, where non-irreducible, non-factorial extensions necessarily appear.
On the other hand, generalized Q-systems (in the infinite index case) dwell a bit further away from the purely categorical setting of their finite index counterpart.

Given a semidiscrete inclusion of von Neumann algebras $\N\subset\M$ (where $\N$ is an infinite factor), we show that the existence of a generalized Q-system with the additional intertwining property is actually equivalent to the \emph{discreteness} of the inclusion in the sense of \cite{ILP98} (but admitting non-irreducible extensions), see Section \ref{sec:discretecase}. This characterization relies on strong results of \cite{ILP98} and \cite{FiIs99}, and can be physically interpreted by saying that a semidiscrete extension is discrete if and only if it is generated by \emph{charged fields}, in the sense of \cite{DoRo72}. These are elements $\psi\in\M$ which generate from the vacuum a non-trivial (irreducible) subsector $\rho\prec\theta$ of the dual canonical endomorphism $\theta$, \ie, $\psi n = \rho(n) \psi$ for every $n\in\N$.

In Section \ref{sec:gennetQsysofinterts} it is shown that generalized $Q$-systems of intertwiners indeed induce discrete (finite or infinite index) extensions of QFTs in the sense of \cite{LoRe95}. Two different ways to obtain the construction are provided: one is a direct generalization of \cite[\Thm 4.9]{LoRe95}, the other one exploits an inductive procedure which is somewhat more suitable to be used for the analysis of braided products and boundary conditions in the subsequent sections.

In Section \ref{sec:covarext}, we give a general proof of \emph{covariance} of QFT extensions constructed from covariant nets of local observables. This fact is apparently well known to experts, and clear in many examples, but we could not find a general statement in the literature (on finite index extensions). The key ingredient in our proof is the \emph{equivariance} of the action of the spacetime symmetry group on the DHR category. More precisely, the mere existence of covariance cocycles, see Definition \ref{def:equivaraction}, is not sufficient to guarantee covariance. One needs in addition naturality and tensoriality properties of the cocycles.

Given two generalized Q-systems of intertwiners (in a \Cstar \emph{braided} tensor category), one can easily define their \emph{braided product} in analogy with the case of ordinary Q-systems, see Definition \ref{def:brprodgenQsysofinterts} and \cf \cite[\Sec 3]{EvPi03}, \cite[\Sec 4.9]{BKLR16}. In Section \ref{sec:brproducts}, we prove the non-trivial statement that the braided product of two generalized Q-systems of intertwiners is again a generalized Q-system of intertwiners, \ie, that the analytical properties defining a Pimsner-Popa basis (as a part of the definition of a generalized Q-system) behave well with respect to the categorical notions of naturality and tensoriality of a braiding in a \Cstar tensor category. Thus, in the QFT setting, we can define the braided product of nets of local observables and construct new examples of irreducible phase boundary conditions with infinite index (infinitely many bulk fields) by taking the direct integral decomposition of the braided product net with respect to its center, see Section \ref{sec:brproductnets} and \ref{sec:apptodefectsinQFT}. On the other hand, we leave open the questions about universality of the braided product construction and the classification of boundary conditions in the infinite index case, \cf \cite[\Sec 5]{BKLR16}.

In Section \ref{sec:exampleU(1)}, we work out examples of infinite index (discrete) extensions of the chiral $U(1)$-current algebra \cite{BMT88} and  explicitly compute their braided products. These examples show an important difference with the analysis of boundary conditions in the finite index case, namely the center of the braided product may be a continuous algebra, \ie, with no non-trivial minimal projections, hence the irreducible boundary conditions constructed by direct integral decomposition need not be representations of the braided product itself. 

Notation-wise we work with nets of local algebras $\{\O\mapsto\A(\O)\}$ indexed by partially ordered and directed sets of spacetime regions $\K$, in order to formulate our results, when possible, for arbitrary spacetime dimensions, \eg, in 1D theories on the line, 1+1D or 3+1D theories in Minkowski space.

%%%
\section{Pimsner-Popa bases}\label{sec:pipo}
%%%

Let $\N\stackrel{E}{\subset}\M$ be a unital inclusion of von Neumann algebras with a normal faithful conditional expectation $E:\M\rightarrow\N$. Assume that $\M$ acts standardly on a separable Hilbert space $\Hilb$ and let $\N \subset \M \subset \M_1 := \left\langle\M,e_\N \right\rangle$ \footnote{Here $\left\langle S \right\rangle$ denotes the von Neumann algebra generated by a subset $S\subset\B(\Hilb)$. For a pair of subsets $S_1, S_2\subset \B(\Hilb)$ we also denote $\left\langle S_1, S_2\right\rangle$ by $S_1 \vee S_2$.} be the Jones basic construction \cite{Jon83}, see also \cite[\Sec 1.1.3]{Pop95}, \cite[\Sec 2.2]{LoRe95}. 
Up to spatial isomorphism it can be characterized as follows. Let $\Omega\in\Hilb$ be a cyclic and separating vector for $\M$ such that the induced (normal faithful) state $\omega$ of $\M$ is invariant under $E$, \ie, $\omega \circ E = \omega$, and set $e_\N := [\N\Omega]$, the orthogonal projection on the subspace $\Hilb_0 := \overline{\N\Omega} \subset \Hilb$. The projection $e_\N\in\N' \cap \M_1$ is the Jones projection of $\N\subset\M$ with respect to $E$, and implements $E$ in the sense that $E(m) e_\N = e_\N m e_\N$, $m\in \M$. Moreover, it is uniquely determined up to conjugation with unitaries in $\M'$ \cite[\App I]{Kos89}.

\begin{definition}\label{def:pi-po}\cite{PiPo86}, \cite{Pop95}.
A \textbf{Pimsner-Popa basis} for $\N\stackrel{E}{\subset}\M$ is a family of elements $\{M_i\} \subset \M$, where $i$ runs in some set of indices $I$, such that
\begin{itemize}
\item[$(i)$] $P_i := M_i^* e_\N M_i$ are projections in $\M_1$ which are mutually orthogonal, \ie, $P_i P_i^* = P_i$ and $P_i P_j = \delta_{i,j} P_i$ for every $i,j\in I$.
\item[$(ii)$] $\sum_i P_i = \oneop$, where the sum converges (unconditionally) in the strong operator topology.
\end{itemize}
\end{definition}

For future reference, we mention the following equivalent characterization of the algebraic properties of Pimsner-Popa bases, see \cite[\Sec 1.1.4]{Pop95}.

\begin{lemma}\label{lem:pi-poequiv}
In the notation of Definition \ref{def:pi-po}, the conditions $(i)$ and $(ii)$ are respectively equivalent to
\begin{itemize}
\item[$(i)'$] $q_i := E(M_iM_i^*)$ are projections in $\N$ (not necessarily mutually orthogonal) and $E(M_iM_j^*) = 0$ for every $i \neq j$, $i,j\in I$.
\item[$(ii)'$] $\overline{\sum_i M_i^* e_\N \Hilb} = \Hilb$ in the Hilbert space topology.
\end{itemize}
\end{lemma}

\begin{proposition}\label{prop:pipoexpansion}\emph{\cite{Pop95}}.
If $\{M_{i}\}$ is a Pimsner-Popa basis for $\N\stackrel{E}{\subset} \M$ then every $m\in M$ has the following expansion 
$$m=\sum_i M_i^* E(M_i m)$$
unconditionally convergent in the topology generated by the family of seminorms $\{\|\cdot\|_{\varphi}:\varphi\in (\M_{*})_{+},\, \varphi = \varphi \circ E\}$, with $\|m\|_{\varphi}:=\varphi(m^{*}m)^{1/2}$. 

The expansion is unique if and only if $E(M_i M_i^*)=\oneop$ for every $i\in I$. 
\end{proposition}

\begin{remark}
In view of the proposition above, Pimsner-Popa bases $\{M_i\}$, or better their adjoints $\{M_i^*\}$ can be seen as bases for $\M$ as a right pre-Hilbert $\N$-module with the $\N$-valued inner product $(M_i^* | M_j^*) := E(M_iM_j^*)$.
\end{remark}

The cardinality of a Pimsner-Popa basis $\{M_{i}\}$ is not a invariant for $\N\stackrel{E}{\subset}\M$.
Indeed, by the following \emph{cutting} and \emph{gluing} procedures \cite[\Sec 1.1.4]{Pop95} we obtain other Pimsner-Popa bases:
\begin{itemize}
\item[$(1)$] If, for each $i$, we take a set of partial isometries $a^{j}_{i}\in\N$ such that $\sum_{j}a_{i}^{j}a^{j *}_{i}=E(M_{i}M_{i}^{*})$, then $\{a_{i}^{j *}M_{i}\}$ is also a basis.
\item[$(2)$] If $E(M_{j}M_{j}^{*})$ and $E(M_{k}M_{k}^{*})$ are orthogonal, then we can replace the pair $M_{j}, M_{k}$ in $\{M_i\}$ by $M_{j}+M_{k}$ and we still get a basis.
\end{itemize}

The good notion of dimension of $\M$ as an $\N$-module is given by the Jones index of the inclusion $\N\subset\M$ with respect to $E$, \cite{Jon83}, \cite{Kos86}. This guiding idea is supported by the following theorem due to \cite[\Prop 1.3]{PiPo86}, \cite[\Thm 3.5]{BDH88}, \cite[\Thm 1.1.5, 1.1.6]{Pop95}, which characterizes the finiteness of the index (and computes its value) by means of Pimsner-Popa bases.

\begin{theorem}\label{thm:havet-popa}\emph{\cite{Pop95}.}
$\N\stackrel{E}{\subset} \M$ has finite Jones index
if and only if it has a Pimsner-Popa basis $\{M_i\}$ such that $\sum_i M_i^* M_i$ is ultraweakly convergent in $\M$. In this case, $\sum_{i}M^{*}_{i}M_{i}$ belongs to the center of $\M$, it holds
$$\sum_{i}M^{*}_{i}M_{i}=\Ind(\N\stackrel{E}{\subset}\M)$$
where $\Ind(\N\stackrel{E}{\subset}\M)$ denotes the Jones index of $E$, and the same is true for any other Pimsner-Popa basis. 

If in addition $\N$ is properly infinite, then $\N\stackrel{E}{\subset}\M$ has finite Jones index if and only if it has a Pimsner-Popa basis made of one element $\{M\}$. Moreover, $M$ can be chosen such that $E(MM^*) = \oneop$.
\end{theorem}

We are mainly interested in inclusions of properly infinite von Neumann algebras (with separable predual), due to their appearance in QFT, see, \eg, \cite{Kad63}, \cite{Lon79}. In this setting, with no finite index or factoriality assumptions, it was shown by Fidaleo and Isola \cite[\Thm 3.5]{FiIs99} that Pimsner-Popa bases made of elements of $\M$ always exist.  

\begin{proposition}\emph{\cite{FiIs99}}.
Every inclusion $\N\stackrel{E}{\subset} \M$ of properly infinite von Neumann algebras with a normal faithful conditional expectation $E: \M \rightarrow \N$ admits a Pimsner-Popa basis $\{M_i\}\subset\M$ in the sense of Definition \ref{def:pi-po}.
\end{proposition}

%%%
\section{Infinite index and generalized Q-systems (of intertwiners)}\label{sec:infindex}
%%%

\emph{Q-systems} were introduced by R. Longo in \cite[\Sec 6]{Lon94}. They provide a way to algebraically characterize infinite subfactors $\N\subset\M$ with \emph{finite index} together with a normal faithful conditional expectation $E:\M\rightarrow\N$ by means of data pertaining to the smaller factor $\N$. 
The main technical tool to achieve this characterization is the notion of \emph{canonical endomorphism} \cite{Lon87} for the inclusion $\N\subset\M$, 
namely the homomorphism $\gamma:\M\rightarrow\N$ defined by $\gamma := (j_\N j_\M)_{\restriction \M}$, where $j_\N := \Ad_{J_{\N,\Phi}}$, $j_\M := \Ad_{J_{\M,\Phi}}$ and $J_{\N,\Phi}$, $J_{\M,\Phi}$ are respectively the modular conjugations of $\N$, $\M$ with respect to a cyclic and separating vector $\Phi$ for $\N$ and $\M$. From a categorical perspective, a Q-system is a special \Cstar Frobenius algebra in a strict \Cstar tensor category $\C$ with simple unit, \cf \cite[\Def 3.8]{BKLR15}. 
In the more concrete case of subfactors, the category is $\C = \End_0(\N)$, whose objects are the endomorphisms of the factor $\N$ with finite dimension in the sense of \cite{LoRo97}.

Here we recall and analyze the more general notion of \emph{generalized Q-system}, introduced by F. Fidaleo and T. Isola in \cite[\Sec 5]{FiIs99} for a possibly infinite index (semidiscrete or semicompact) inclusion of properly infinite von Neumann algebras. We then introduce the more special notion of generalized Q-system \emph{of intertwiners} that will be the fundamental object in the subsequent sections, in particular for the applications to QFT.

Let $\N\subset\M$ be a unital inclusion of properly infinite von Neumann algebras on a separable Hilbert space $\Hilb$. Denote by $C(\M,\N)$ and $E(\M,\N)$ respectively the set of all normal and normal faithful conditional expectations of $\M$ onto $\N$. We call the inclusion $\N\subset \M$ \textbf{semidiscrete} if $E(\M,\N)\neq\emptyset$, and \textbf{semicompact} if $E(\N',\M')\neq\emptyset$, or equivalently if $E(\M_1,\M)\neq\emptyset$ or $E(\N,\N_1)\neq\emptyset$, where $\N_1\subset\N\subset\M\subset\M_1$ denotes the tower of von Neumann algebras obtained by canonical extension and restriction of the original inclusion \cite[\Sec 2.5 A]{LoRe95}. The terminology is adopted from \cite{FiIs99}, \cite{ILP98}, \cite{FiIs95}, \cite{HeOc89}. Recall that a finite index inclusion is both semidiscrete and semicompact, see \eg \cite[\Prop 4.4]{Lon90}. 

Let $\End(\N)$ be the collection of normal faithful unital *-endomorphisms of $\N$. The following notion is tailored to describe \emph{semidiscrete} inclusions of von Neumann algebras $\N\subset\M$ with $E\in E(\M,\N)$, possibly of \emph{infinite index}.

\begin{definition}\label{def:genQsys}\cite{FiIs99}.
Let $\N$ be a properly infinite von Neumann algebra. A \textbf{generalized Q-system} in $\C = \End(\N)$ is a triple $(\theta, w, \{m_i\})$ consisting of an endomorphism $\theta\in\End(\N)$, an isometry $w\in\Hom_{\End(\N)}(\id,\theta)$ (\ie, $wn = \theta(n)w$, $n\in\N$), and a family $\{m_i\}\subset\N$ indexed by $i$ in some set $I$, such that
\begin{itemize}
\item[$(i)$] $p_i := m_i^* ww^* m_i$ are mutually orthogonal projections in $\N$, \ie, $p_i p_j^* = \delta_{i,j} p_i$, such that $\sum_i p_i = \oneop$. (\lqq Pimsner-Popa condition") 
\item[$(ii)$] $n w = 0 \Rightarrow n = 0$ if $n\in \N_1 := \left\langle\theta(\N), \{m_i\}\right\rangle$. (\lqq faithfulness condition") 
\end{itemize}
\end{definition}

\begin{remark}
An analogous definition of generalized Q-system in $\End(\N)$, involving an isometry $x\in\Hom_{\End(\N)}(\theta,\theta^2)$ instead of $w\in\Hom_{\End(\N)}(\id,\theta)$, can be given in the \emph{semicompact} case, see \cite[\Sec 5]{FiIs99}. We shall however be interested in extensions $\N\subset\M$ with a (normal faithful) conditional expectation $E\in E(\M,\N)$ as they arise in QFT when $\N=\A(\O)$, $\M=\B(\O)$ are local algebras (relative to some spacetime region $\O$) and $\{\A\subset \B\}$ is an extension of a net of local observables $\{\A\}$ by means of a \lqq field net" $\{\B\}$. Here $E$ generalizes the notion of an average over a global gauge group action on fields, giving the observables as the gauge invariant part.
\end{remark}

\begin{theorem}\label{thm:FiIs}\emph{\cite{FiIs99}}.
Let $\N$ be a properly infinite von Neumann algebra with separable predual and $\theta\in \End(\N)$. Then the following are equivalent
\begin{itemize}
\item[$(1)$] There is a von Neumann algebra $\N_1$ such that $\N_1\subset\N$ with $E' \in E(\N_1, \N_2)\neq \emptyset$, where $\N_2 := \theta(\N)\subset\N_1$, and $\theta$ is a canonical endomorphism for $\N_1\subset\N$.
\item[$(2)$] There is a von Neumann algebra $\M$
such that $\N\subset\M$ with $E\in E(\M, \N)\neq \emptyset$, and $\theta$ is a dual canonical endomorphism for $\N\subset\M$, \ie, $\theta = \gamma_{\restriction\N}$ where $\gamma\in\End(\M)$ is a canonical endomorphism for $\N\subset\M$.
\item[$(3)$] The endomorphism $\theta$ is part of a generalized Q-system in $\End(\N)$, $(\theta, w, \{m_i\})$, see Definition \ref{def:genQsys}.
\end{itemize}
\end{theorem}

\begin{proof}
We may assume that $\N$ is in its standard representation on $\Hilb$. The equivalence of $(1)$ and $(2)$ is then obtained by canonical extension and restriction \cite[\Sec 2.5 A]{LoRe95}. The tower of von Neumann algebras reads
\begin{equation}\label{eq:vNtower}\ldots\subset\N_2 = \theta(\N) \stackrel{E'}{\subset} \N_1 = \left\langle\theta(\N), \{m_i\}\right\rangle \stackrel{\theta}{\subset} \N \stackrel{E, \gamma}{\subset} \M\subset\ldots\end{equation}
where $\N_2 \subset \N_1 = \gamma (\N \subset \M)$ is a spatial isomorphism of inclusions and the relation $ E' \circ \gamma = \gamma \circ E$ on $\M$ gives a bijection between $E(\N_1,\N_2)$ and $E(\M,\N)$. 

The equivalence of $(1)$ and $(3)$ is due to \cite[\Thm 4.1]{FiIs99}. In particular, they show that $e_{\N_2} := ww^*$ is a Jones projection for the inclusion $\N_2\subset\N_1$ with respect to $E' := \theta(w^* \cdot w)$ and that $\N = \left\langle\N_1, e_{\N_2}\right\rangle$ is the associated Jones extension. Hence the condition $(i)$ in Definition \ref{def:genQsys} says that $\{m_i\}\subset\N_1$ is a Pimsner-Popa basis for $\N_2\subset\N_1$ with respect to $E'$.
The condition $(ii)$ in Definition \ref{def:genQsys} is nothing but faithfulness of $E'$.
\end{proof}

\begin{remark}
The condition that the $p_i$ in Definition \ref{def:genQsys} are (mutually orthogonal) projections in $\N$, \ie, $p_ip_j^* = \delta_{i,j}p_i$, does not enter in the proof of $(3) \Rightarrow (1)$ of Theorem \ref{thm:FiIs}, only $\sum_i p_i = \oneop$ is relevant there. We can however always assume it because 
$m_i^* ww^* m_i$ is a projection if and only if $w^*m_i m_i^* w$ is a projection, which is equivalent to $ww^*m_i m_i^* ww^* = E'(m_i m_i^*) ww^*$ is a projection, \ie, $E'(m_i m_i^*)$ is a projection, because $ww^* = e_{\N_2}$ and $n\mapsto n e_{\N_2}$ is an isomorphism of $\N_2$ onto $\N_2 e_{\N_2}$. Hence we can apply a Gram-Schmidt orthogonalization procedure to the $\{m_i\}$ with respect to the operator-valued inner product $(m_i | m_j) := E'(m_j m_i^*)$ and choose another basis $\{\tilde m_i\}$ such that $E'(\tilde m_j \tilde m_i^*) = \delta_{i,j}\oneop$. 
\end{remark}

\begin{remark}\label{rmk:genQsysmoreflex}
Notice that no \emph{factoriality} $\Z(\N) = \CC\oneop$, $\Z(\M) = \CC\oneop$, nor \emph{irreducibility} $\N'\cap\M = \CC\oneop$ assumptions enter in the proof of Theorem \ref{thm:FiIs}, \cite{FiIs99}. In the case of non-irreducible finite index subfactors, $E$ is not necessarily the \emph{minimal} conditional expectation, see \cite{Hia88}, \cite[\Sec 5]{Lon89}. 
\end{remark}

\begin{proposition}\emph{\cite{FiIs99}}.
Let $\N\subset\M$ a semidiscrete inclusion of properly infinite von Neumann algebras and let $\gamma$ be a canonical endomorphism. The following are equivalent
\begin{itemize}
\item[$(1)$] $\N\subset\M$ is irreducible in the sense that $\N'\cap\M = \Z(\N)$.
\item[$(2)$] $E(\M,\N)$ contains only one element.
\item[$(3)$] $\Hom_{\End(\N)}(\id,\theta)$ is cyclic as a $\Z(\N)$-module, where $\theta = \gamma_{\restriction \N}$.
\end{itemize}
\end{proposition}

We now specialize the notion of generalized Q-system (Definition \ref{def:genQsys}) by requiring an additional intertwining property of the Pimsner-Popa elements.

\begin{definition}\label{def:genQsysofinterts}
Let $\N$ be a properly infinite von Neumann algebra. We call $(\theta, w, \{m_i\})$ a \textbf{generalized Q-system of intertwiners} in $\C = \End(\N)$ 
if, in addition to the properties of Definition \ref{def:genQsys}, it satisfies $m_i \in\Hom_{\End(\N)}(\theta,\theta^2)$ (\ie, $m_i\theta(n)=\theta^2(n)m_i$, $n\in\N$) for every $i\in I$. 

In this case we can use string diagrams to denote $w$ and $m_i$ as follows 
$$
w \,= 
	\hspace{-2mm}
	\tikzmath{
	\draw[dashed]
	(5,3) node [above] {$\id$} -- (5,-5);
	\draw[ultra thick]
	(5,-5) node [] {\textbullet} -- (5,-17) node [below] {$\theta$};
	}
	\,=
	\hspace{-0.5mm}
	\tikzmath{
	\draw (5,7) node [above] {};
	\draw[ultra thick]
	(5,-5) node [] {\textbullet} -- (5,-17) node [below] {$\theta$};
	},
\qquad m_i \,=
	\hspace{-2mm}
	\tikzmath{
	\draw[ultra thick]
	(0,-14) node [below] {$\theta$}--(0,-10)..
	controls (0,0) and (10,0)
	..(10,-10)--(10,-14) node [below] {$\theta$};
	\draw[ultra thick]
	(5,7) node [above] {$\theta$}--(5,-2.5) node [] {\textbullet};
	\draw (5,0.5) node [right] {$i$};
	}, \; i\in I.
$$
\end{definition}

At this point, a comparison between the notions of generalized Q-system and \lqq ordinary" Q-system in the finite index setting is due.

\begin{remark}\label{rmk:cfqsys}(\textbf{The finite index case}).
An infinite subfactor $\N\subset\M$ with $E\in E(\M,\N)$ can be characterized by an \lqq ordinary" Q-system $(\theta,w,x)$ if and only if the Jones index of $E$ is finite, see \cite{Lon94}, \cite[\Sec 2.7]{LoRe95}. The algebraic relations defining a Q-system in $\End_0(\N)$ read as follows: $\theta\in\End_0(\N)$, $w\in\Hom_{\End_0(\N)}(\id,\theta)$, $x \in\Hom_{\End_0(\N)}(\theta,\theta^2)$ and
$$w^* x = \theta(w^*) x = \oneop,\quad x^2 = \theta(x) x, \quad xx^* = x^* \theta(x) = \theta(x^*) x, \quad x^*x \in \CC\oneop.$$
The conditions in the line above are called respectively unit property, associativity, Frobenius property and specialness, see \cite[\Def 3.8]{BKLR15}. It is known that the Frobenius property is a consequence of the other properties \cite{LoRo97}, \cite[\Lem 3.7]{BKLR15} and that specialness is not needed to construct the extension $\N_2 = \theta(\N)\subset\N_1$, \ie, $\N\subset\M$ \cite[\Rmk 3.18]{BKLR15}.

Moreover, it is an easy exercise to check that ordinary Q-systems are also generalized Q-system of intertwiners with $\{m_i\} = \{x\}$ (up to a normalization of $w$ and $x$), in the sense of Definition \ref{def:genQsysofinterts}. Indeed the Pimsner-Popa condition $x^*ww^*x = \oneop$ follows by $w^*x = \oneop$, and the faithfulness condition $n w = 0 \Rightarrow n = 0$, $n\in \N_1 = \left\langle\theta(\N), x\right\rangle$ follows because $\left\langle\theta(\N), x\right\rangle = \theta(\N)x = x^*\theta(\N)$ hold, due to $\theta(w^*)x = \oneop$, associativity and Frobenius property.
 
On the other hand, a finite index inclusion of infinite factors $\N_2\subset \N_1$ with normal faithful conditional expectation $E^{\prime}(\cdot)=\theta(w^{*}\cdot w)$, always has a Pimsner-Popa basis of one element, $m\in\N_1$, such that $E'(mm^*) = \oneop$ by Theorem \ref{thm:havet-popa}. The triple $(\theta,w,m)$ is a generalized Q-system in the sense of Definition \ref{def:genQsys}. The characterizing properties 
$$m^{*}ww^{*}m=\oneop,\quad w^{*}mm^{*}w=\oneop$$ 
are a weaker form of the unit property for ordinary Q-systems, and the Pimsner-Popa expansion of Proposition \ref{prop:pipoexpansion} gives in particular
$$m^2=E'(m^2m^*)m,\quad mm^*=E'(m(m^*)^2)m.$$
If we assume the unit property $w^*m=\oneop$ to hold, we get back the associativity $m^2=\theta(m)m$ and the Frobenius property $mm^* = \theta(m^*)m$.
\end{remark}

If $(\theta, w, \{m_i\})$ is a generalized Q-system (of intertwiners) in $\C = \End(\N)$, consider the tower of von Neumann algebras 
$$\ldots\subset\N_2 \stackrel{E'}{\subset} \N_1 \stackrel{\theta}{\subset} \N \stackrel{E, \gamma}{\subset} \M \stackrel{\gamma_1}{\subset} \M_1\subset\ldots$$
as in equation (\ref{eq:vNtower}), where the Jones extension $\M_1 = \left\langle \M,e_\N \right\rangle$ of $\N\subset\M$ with respect to $E$ coincides with the canonical extensions, namely $\left\langle \M,e_\N\right\rangle = j_{\M}(\N')$, see \cite[\Sec 2.5 D]{LoRe95}, \cite[\Sec 3]{Lon89}. Here $\Omega$ is a cyclic and separating vector for $\M$ as in Section \ref{sec:pipo} and $j_{\M} = \Ad_{J_{\M,\Omega}}$ is the associated modular conjugation. Moreover, $\theta$ and $\gamma_1$ are canonical endomorphisms dual to $\gamma$, hence $\theta = \gamma_{\restriction \N}$, $\gamma = {\gamma_1}_{\restriction \M}$. Then ${\gamma_1}^{-1}(ww^*) = {\gamma_1}^{-1}(e_{\N_2}) = e_{\N}$ and $\{M_i := {\gamma}^{-1}(m_i)\} \subset \M$ clearly forms a Pimsner-Popa basis for $\N\subset\M$ with respect to $E$. 

\begin{definition}\label{def:dualgenQsysofinterts}
We call $(\gamma, w, \{M_i\})$ a generalized Q-system (of intertwiners) \textbf{dual} to $(\theta, w, \{m_i\})$. The intertwining relation $m_i \in \Hom_{\End(\N)}(\theta,\theta^2)$ is equivalent to $M_i n = \theta(n) M_i$, $n\in\N$.
\end{definition}

%%%
\section{Braided products}\label{sec:brproducts}
%%%

Suppose additionally that two generalized Q-systems of intertwiners are composed of data belonging to a certain \emph{braided} tensor subcategory of $\End(\N)$, we can consider their \emph{braided product} as follows

\begin{definition}\label{def:brprodgenQsysofinterts} 
Let $\N$ be a properly infinite von Neumann algebra and $\C\subset\End(\N)$ a \Cstar \emph{braided} tensor subcategory of $\End(\N)$. Let $(\theta^A, w^A, \{m^A_{i_1}\})$ and $(\theta^B, w^B, \{m^B_j\})$ two generalized Q-systems of intertwiners in $\C$ (Definition \ref{def:genQsysofinterts}), indexed respectively by $i\in I$ and $j\in J$. We call
$$(\theta^A\theta^B, w^Aw^B, \{m^A_i \times_\eps^{\pm} m^B_j\})$$
the \textbf{braided product} of $(\theta^A, w^A, \{m^A_i\})$ and $(\theta^B, w^B, \{m^B_j\})$, indexed by $(i,j)\in I\times J$, where
$$m^A_i \times_\eps^{\pm} m^B_j := \theta^A(\eps^{\pm}_{\theta^A,\theta^B}) m^A_i \theta^A(m^B_j)$$
depending on the $\pm$ choice. Here $\eps^+ = \eps$ and $\eps^- = \eps^\op$ denote respectively the braiding of $\C$ and its opposite. Equivalently
$$
w^Aw^B \,= 
	\hspace{-2mm}
	\tikzmath{
	\draw (5,7) node [above] {};
	\draw[ultra thick]
	(5,-5) node [] {\textbullet} -- (5,-17) node [below] {$\theta^A$};
	\draw (10,7) node [above] {};
	\draw[ultra thick]
	(10,-5) node [] {\textbullet} -- (10,-17) node [below] {$\;\;\theta^B$};
	},
\qquad m^A_i \times_\eps^{+} m^B_j \,=
	\hspace{-2mm}
	\tikzmath{
	\draw[ultra thick]
	(15,7) node [above] {$\theta^B$}--(15,-3.2) node [] {\textbullet};
	\draw (15,0.5) node [right] {$j$};
	\draw[ultra thick]
	(7,-17) node [below] {$\theta^B$}--(7,-13)..
	controls (7,0) and (23,0)
	..(23,-13)--(23,-17) node [below] {$\theta^B$};
	\fill[color=white] (10.2,-5.5) circle (2);%braiding
	\draw[ultra thick]
	(5,7) node [above] {$\theta^A$}--(5,-3.2) node [] {\textbullet};
	\draw (5,0.5) node [right] {$i$};
	\draw[ultra thick]
	(-3,-17) node [below] {$\theta^A$}--(-3,-13)..
	controls (-3,0) and (13,0)
	..(13,-13)--(13,-17) node [below] {$\,\theta^A$};
	}, \; (i,j) \in I \times J
$$
and similarly for $m^A_i \times_\eps^{-} m^B_j$.
\end{definition}

Surprisingly, the analytic conditions dictated on generalized Q-systems by subfactor theory (\eg the property characterizing a Pimsner-Popa basis) turn out to be naturally compatible with the categorical notion of braiding in a tensor category of endomorphisms. Indeed we have the following proposition which extends the braided product construction, see \cite[\Sec 4.9]{BKLR15}, to the infinite index case.

\begin{proposition}\label{prop:brprodisQsys}
The braided product of two generalized Q-systems of intertwiners is again a generalized Q-system of intertwiners.
\end{proposition}

\begin{proof}
The intertwining properties appearing in Definition \ref{def:genQsysofinterts} are easily checked once we write the operators $w^Aw^B$ and $m^A_i \times_\eps^{\pm} m^B_j$, $(i,j)\in I\times J$ as tensor products and compositions of arrows in the braided tensor category of endomorphisms $\C\subset\End(\N)$, as in the case of ordinary Q-systems \cite[\Def 4.30]{BKLR15}.

The Pimsner-Popa condition $(i)$ in Definition \ref{def:genQsys} is more lengthy to check. For each $i\in I$ and $j\in J$, let 
$$p^{AB,\pm}_{i,j} := (m^A_i \times_\eps^{\pm} m^B_j)^* w^Aw^B w^{B*}w^{A*}(m^A_i \times_\eps^{\pm} m^B_j)$$
$$= \theta^A(m^{B*}_j)m^{A*}_i\theta^A((\eps^{\pm}_{\theta^A,\theta^B})^*) \theta^A(w^B)w^Aw^{A*}\theta^A(w^{B*})\theta^A(\eps^{\pm}_{\theta^A,\theta^B}) m^A_i \theta^A(m^B_j)$$
$$= \theta^A(m^{B*}_j)m^{A*}_i\theta^A(\theta^A(w^B))w^Aw^{A*}\theta^A( \theta^A(w^{B*})) m^A_i \theta^A(m^B_j)$$
because $(\eps^{\pm}_{\theta^A,\theta^B})^* w^B = \theta^A(w^B) (\eps^{\pm}_{\theta^A,\id})^*$ by \emph{naturality} of the braiding $\eps^+ = \eps$ in the braided tensor category $\C$, or of its opposite $\eps^- = \eps^\op$, and $\eps^{\pm}_{\theta^A,\id} = \oneop$ by convention. Moreover
$$= \theta^A(m^{B*}_j)\theta^A(w^B)m^{A*}_iw^Aw^{A*}m^A_i\theta^A(w^{B*}) \theta^A(m^B_j)$$
$$= m^{A*}_iw^Aw^{A*}m^A_i \theta^A(m^{B*}_jw^Bw^{B*} m^B_j)$$
hence we have shown 
\begin{equation}\label{eq:brprodprojs}p^{AB,\pm}_{i,j} = p^{A}_{i} \theta^A(p^{B}_{j}) = \theta^A(p^{B}_{j}) p^{A}_{i}\end{equation} 
where $p^{A}_{i}$, $i\in I$ and $p^{B}_{j}$, $j\in J$ are the projections appearing in Definition \ref{def:genQsys} respectively for the two generalized Q-systems. Equation (\ref{eq:brprodprojs}) is much more effectively expressed using graphical calculus
$$
\tikzmath{
	\draw[ultra thick]
	(15,7) node [above] {$\theta^B$}--(15,-3.2) node [] {\textbullet};
	\draw (15,0.5) node [right] {$j$};
	\draw[ultra thick]
	(7,-13) node [] {\textbullet}--(7,-13)..
	controls (7,0) and (23,0)
	..(23,-13)--(23,-17) node [below] {};
	\fill[color=white] (10.2,-5.3) circle (2);%braiding
	\draw[ultra thick]
	(5,7) node [above] {$\theta^A$}--(5,-3.2) node [] {\textbullet};
	\draw (5,0.5) node [right] {$i$};
	\draw[ultra thick]
	(-3,-13) node [] {\textbullet}--(-3,-13)..
	controls (-3,0) and (13,0)
	..(13,-13)--(13,-17) node [below] {};

	\draw[ultra thick]
	(7,-21) node [] {\textbullet}--(7,-21)..
	controls (7,-34) and (23,-34)
	..(23,-21)--(23,-17) node [below] {};
	\draw[ultra thick]
	(15,-41) node [below] {$\theta^B$}--(15,-30.8) node [] {\textbullet};
	\draw (15,-34.5) node [right] {$j$};
	\fill[color=white] (10.2,-28.8) circle (2);%braiding
	\draw[ultra thick]
	(-3,-21) node [] {\textbullet}--(-3,-21)..
	controls (-3,-34) and (13,-34)
	..(13,-21)--(13,-17) node [below] {};
	\draw[ultra thick]
	(5,-41) node [below] {$\theta^A$}--(5,-30.8) node [] {\textbullet};
	\draw (5,-34.5) node [right] {$i$};
	}\,= 
	\hspace{-1mm}
\tikzmath{
	\draw[ultra thick]
	(25,7) node [above] {$\theta^B$}--(25,-3.2) node [] {\textbullet};
	\draw (25,0.5) node [right] {$j$};
	\draw[ultra thick]
	(19,-13) node [] {\textbullet}--(19,-13)..
	controls (19,0) and (31,0)
	..(31,-13)--(31,-17) node [below] {};
	\draw[ultra thick]
	(5,7) node [above] {$\theta^A$}--(5,-3.2) node [] {\textbullet};
	\draw (5,0.5) node [right] {$i$};
	\draw[ultra thick]
	(-1,-13) node [] {\textbullet}--(-1,-13)..
	controls (-1,0) and (11,0)
	..(11,-13)--(11,-17) node [below] {};

	\draw[ultra thick]
	(19,-21) node [] {\textbullet}--(19,-21)..
	controls (19,-34) and (31,-34)
	..(31,-21)--(31,-17) node [below] {};
	\draw[ultra thick]
	(25,-41) node [below] {$\theta^B$}--(25,-30.8) node [] {\textbullet};
	\draw (25,-34.5) node [right] {$j$};
	\draw[ultra thick]
	(-1,-21) node [] {\textbullet}--(-1,-21)..
	controls (-1,-34) and (11,-34)
	..(11,-21)--(11,-17) node [below] {};
	\draw[ultra thick]
	(5,-41) node [below] {$\theta^A$}--(5,-30.8) node [] {\textbullet};
	\draw (5,-34.5) node [right] {$i$};
	}.
$$
Now one can easily check that $p^{AB,\pm}_{i,j}$ are mutually orthogonal projections which sum up to $\oneop$.

The faithfulness condition $(ii)$ in Definition \ref{def:genQsys} follows from Lemma \ref{lem:embedbrprod} below. Indeed, let $n\in\N_1^A\times_\eps^\pm\N_1^B$ (see below), then $n w^Aw^B = n \theta^A(w^B)w^A = 0$ if and only if $n \theta^A(w^B) = 0$ since $\N_1^A\times_\eps^\pm\N_1^B \subset \N_1^A$. Now $n \theta^A(w^B) = 0$ if and only if $\eps^{\pm}_{\theta^A,\theta^B} n \theta^A(w^B) = \eps^{\pm}_{\theta^A,\theta^B} n (\eps^{\pm}_{\theta^A,\theta^B})^* w^B = 0$ by naturality of the braiding. Since $\Ad(\eps^{\pm}_{\theta^A,\theta^B})(\N_1^A \times_\eps^\pm \N_1^B)\subset \N_1^B$ we have that $n=0$ and the proof is complete.
\end{proof}

\begin{lemma}\label{lem:embedbrprod}
In the notation of Definition \ref{def:brprodgenQsysofinterts}, consider the two towers of von Neumann algebras $\N_2^A\subset\N_1^A\subset\N\subset\M^A$ and $\N_2^B\subset\N_1^B\subset\N\subset\M^B$ respectively associated to the two generalized Q-systems (of intertwiners) as in Theorem \ref{thm:FiIs}. Let 
$$\N_1^A \times_\eps^\pm \N_1^B := \left\langle\theta^A\theta^B(\N), \{m^A_i \times_\eps^{\pm} m^B_j\}\right\rangle$$
then 
$$\N_1^A \times_\eps^\pm \N_1^B \subset \N_1^A, \quad \Ad(\eps^{\pm}_{\theta^A,\theta^B})(\N_1^A \times_\eps^\pm \N_1^B)\subset \N_1^B.$$
\end{lemma}

\begin{proof}
The first inclusion follows from the very definitions. To show the second observe that $\Ad(\eps^{\pm}_{\theta^A,\theta^B})(\theta^A\theta^B(\N))= \theta^B\theta^A(\N) \subset \theta^B(\N)$. Hence it is enough to check that
$$\eps^{\pm}_{\theta^A,\theta^B}\theta^A\theta^B(\eps^{\pm}_{\theta^A,\theta^B})\theta^A(\eps^{\pm}_{\theta^A,\theta^B})m^A_i\theta^A(m^B_j)(\eps^{\pm}_{\theta^A,\theta^B})^{*} = \theta^B(\eps^{\mp}_{\theta^B,\theta^A})m^B_j\theta^B(m^A_i),$$
but this follows from repeated application of naturality and tensoriality of the braiding
$$m^A_i\theta^A(m^B_j)(\eps^{\pm}_{\theta^A,\theta^B})^{*}=\theta^A\theta^A(m^B_j)m^A_i\eps^{\mp}_{\theta^B,\theta^A}$$
$$= \theta^A\theta^A(m^B_j)\eps^{\mp}_{\theta^B,\theta^A\theta^A}\theta^B(m^A_i) = \eps^{\mp}_{\theta^B\theta^B,\theta^A\theta^A}m^B_j\theta^B(m^A_i)$$
where $\eps^{\mp}_{\theta^B\theta^B,\theta^A\theta^A} = \theta^A(\eps^{\mp}_{\theta^B,\theta^A})\theta^A\theta^B(\eps^{\mp}_{\theta^B,\theta^A})\eps^{\mp}_{\theta^B,\theta^A}\theta^B(\eps^{\mp}_{\theta^B,\theta^A})$. 
\end{proof}

\begin{corollary}\label{cor:brexp}\emph{(of Proposition \ref{prop:brprodisQsys})}.
$\theta^A\theta^B\in\End(\N)$ is a canonical endomorphism for the inclusion $\N_1^A \times_\eps^\pm \N_1^B \subset \N$. Moreover, the inclusion \begin{equation}\label{eq:brprunder}
\theta^A\theta^B(\N) \subset \N_1^A \times_\eps^\pm \N_1^B
\end{equation} 
is semidiscrete \footnote{By the results of the next section, the inclusion (\ref{eq:brprunder}) is also discrete in the sense of Definition \ref{def:discreteinclusion}.} with (normal faithful) conditional expectation given by
$${E^{AB}}' := \theta^A\theta^B(w^{B*}w^{A*}\cdot w^Aw^B).$$
\end{corollary}

Denote by
$$\M^A\times_\eps^\pm\M^B$$
the von Neumann algebra appearing in the tower
$$\ldots \subset \theta^A\theta^B(\N) \subset \N_1^A \times_\eps^\pm \N_1^B \stackrel{\theta^A\theta^B}{\subset} \N \stackrel{\gamma^{AB}}{\subset} \M^A\times_\eps^\pm\M^B \subset \ldots$$
obtained as in Theorem \ref{thm:FiIs} from the braided product Q-system. We call it the \textbf{braided product of $\M^A$ and $\M^B$}.
Here $\gamma^{AB}$ denotes a canonical endomorphism for the inclusion $\N \subset \M^A\times_\eps^\pm\M^B$ dual to $\theta^A\theta^B$. By definition, ${\gamma^{AB}}_{\restriction\N} = \theta^A\theta^B$ and $\gamma^{AB}(\M^A\times_\eps^\pm\M^B) = \N_1^A \times_\eps^\pm \N_1^B$. Similarly, $\gamma^A$ and $\gamma^B$ are respectively canonical endomorphisms dual to $\theta^A$ and $\theta^B$.

In order to show that the braided product $\M^A\times_\eps^\pm\M^B$ actually contains $\M^A$ and $\M^B$ as subalgebras (see Proposition \ref{prop:brprodembeddings} below) we need to consider generalized Q-systems of intertwiners with an additional property, which is a weaker version of the unit property in ordinary Q-systems, namely $\theta(w^*)x = \oneop$, \cf \cite[\Prop 4.12]{BKLR16}. We shall come back to this property in the next section, see Proposition \ref{prop:discreteminimal} and Definition \ref{def:unitalqsys}.

\begin{proposition}\label{prop:brprodembeddings}
In the notation of Definition \ref{def:brprodgenQsysofinterts}, let $(\theta^A, w^A, \{m^A_{i}\})$ and $(\theta^B, w^B, \{m^B_j\})$ fulfill in addition $\theta^A(w^{A*})m_0^A = \oneop$ and $\theta^B(w^{B*})m_0^B = \oneop$ for two distinguished labels $0\in I$ and $0 \in J$. Then the maps 
$$\jmath^A : \M^A \rightarrow \M^A\times_\eps^\pm\M^B ,\quad \jmath^A := (\gamma^{AB})^{-1}\circ \Ad_{(\eps_{\theta^A,\theta^B}^{\pm})^*} \circ\, \theta^B \circ \gamma^A$$
$$\jmath^B : \M^B \rightarrow \M^A\times_\eps^\pm\M^B ,\quad \jmath^B := (\gamma^{AB})^{-1} \circ \theta^A \circ \gamma^B$$

are embeddings respectively of $\M^A$ and $\M^B$ into $\M^A\times_\eps^\pm\M^B$. Call $\iota^A : \N \rightarrow \M^A$ and $\iota^B : \N \rightarrow \M^B$ the embeddings of $\N$ into $\M^A$ and $\M^B$ respectively. Then the two embeddings of $\N$ into the braided product coincide, \ie
\begin{equation}\label{eq:embABonN}\jmath^A\circ\iota^A = \jmath^B\circ\iota^B,\end{equation}
the commutation relations among $M_i^A$ and $M_j^B$, as in Definition \ref{def:dualgenQsysofinterts}, in the braided product $\M^A\times_\eps^\pm\M^B$ are given by
\begin{equation}\label{eq:commrelbrprod}\jmath^B(M_j^B) \jmath^A(M_i^A) = \eps_{\theta^A,\theta^B}^{\pm} \jmath^A(M_i^A) \jmath^B(M_j^B),\quad i\in I,j\in J.\end{equation}
Moreover, $\M^A$ and $\M^B$ generate the braided product, \ie
\begin{equation}\label{eq:genbrprod}\M^A\times_\eps^\pm\M^B = \left\langle\jmath^A(\M^A), \jmath^B(\M^B)\right\rangle.\end{equation}
\end{proposition}

\begin{proof}
We show first that
\begin{equation}\label{eq:embA}\Ad_{(\eps_{\theta^A,\theta^B}^{\pm})^*} \circ\, \theta^B (\N_1^A) \subset \N_1^A \times_\eps^\pm \N_1^B\end{equation}
\begin{equation}\label{eq:embB}\theta^A (\N_1^B) \subset \N_1^A \times_\eps^\pm \N_1^B\end{equation}
from which it is clear that $\jmath^A$ and $\jmath^B$ are embeddings into $\M^A\times_\eps^\pm\M^B$. For the inclusion (\ref{eq:embB}) it is enough to show that $\theta^A(m_j^B) \in \N_1^A \times_\eps^\pm \N_1^B$. By assumption $\theta^A(w^{A*}) m_0^A = \oneop$, hence
$$\theta^A(m_j^B) = \theta^A(w^{A*}) m_0^A \theta^A(m_j^B) = \theta^A\theta^B(w^{A*})\theta^A(\eps^\pm_{\theta^A,\theta^B}) m_0^A \theta^A(m_j^B)$$
$$=\theta^A\theta^B(w^{A*}) m_0^A \times_\eps^\pm m_j^B\in\N_1^A \times_\eps^\pm \N_1^B$$
using naturality of the braiding and $\eps^\pm_{\id,\theta^B}=\oneop$.
For the inclusion (\ref{eq:embA}), it is enough to observe that $\Ad_{\eps_{\theta^A,\theta^B}^\mp}$ is an isomorphism between $\N_1^A \times_\eps^\pm \N_1^B$ and $\N_1^B \times_\eps^\mp \N_1^A$, \cf \cite[\Sec 4.9]{BKLR15}, and consider the previous case interchanging $A$ with $B$ and $\pm$ with $\mp$. 
Now, (\ref{eq:embABonN}) is clear.
To show the commutation relations among $M_i^A$ and $M_j^B$ apply first $\gamma^{AB}$ to equation (\ref{eq:commrelbrprod}). The \rhs then reads 
$$\theta^A\theta^B(\eps_{\theta^A,\theta^B}^\pm)(\eps_{\theta^A,\theta^B}^{\pm})^*\theta^B(m_i^A) \eps_{\theta^A,\theta^B}^{\pm}\theta^A(m_j^B) $$
$$= \theta^A\theta^B(\eps_{\theta^A,\theta^B}^\pm)(\eps_{\theta^A,\theta^B}^{\pm})^*\eps_{\theta^A\theta^A,\theta^B}^{\pm}m_i^A\theta^A(m_j^B) $$
$$=  \theta^A\theta^B(\eps_{\theta^A,\theta^B}^\pm)\theta^A(\eps_{\theta^A,\theta^B}^{\pm})m_i^A\theta^A(m_j^B).$$
Similarly, one can compute the \lhs, namely 
$$\theta^A(m_j^B) (\eps^{\pm}_{\theta^A,\theta^B})^*\theta^B(m_i^A)(\eps^{\pm}_{\theta^A,\theta^B}) = \theta^A(\eps_{\theta^A,\theta^B\theta^B}^\pm) \theta^A\theta^A(m_j^B) m_i^A.$$
By the intertwining property $\theta^A\theta^A(m_j^B) m_i^A = m_i^A\theta^A(m_j^B)$ and by tensoriality of the braiding we have equation (\ref{eq:commrelbrprod}).
In the previous computations we have shown in particular that
$$\jmath^A(M_i^A) \jmath^B(M_j^B) = (\gamma^{AB})^{-1}(m_i^A\times_\eps^\pm m_j^B)$$
from which equation (\ref{eq:genbrprod}) follows.
\end{proof}

%%%
\section{The case of discrete inclusions}\label{sec:discretecase}
%%%

Generalized Q-systems with the additional intertwining property $m_i \in\Hom_{\End(\N)}(\theta,\theta^2)$ as in Definition \ref{def:genQsysofinterts} can be constructed whenever the inclusion $\N\subset\M$ is \emph{discrete} (see Definition \ref{def:discreteinclusion} below, \cf \cite[\Def 3.7]{ILP98}). The main idea is to look first at elements $\psi_\rho\in\M$ which generate on $\N$ subendomorphisms $\rho\prec\theta$ of the dual canonical endomorphism $\theta\in\End(\N)$ of $\N\subset\M$ from the vacuum (identity representation), namely such that
$$\psi_\rho n = \rho(n) \psi_\rho, \quad n\in\N.$$ 
Such elements are called \textbf{charged fields} after the work of \cite{DoRo72} in QFT. In the subfactor setting they can be constructed as in \cite[\Prop 3.2]{ILP98}. We generalize the latter construction to the case of non-irreducible, non-factorial extensions (as one needs in the study of defects in QFT, see \cite[\Thm 4.4]{BKLR16}), and we show how charged fields can be used, in the discrete case, to define generalized Q-systems of intertwiners. Moreover, we show that a semidiscrete inclusion admitting a generalized Q-system of intertwiners is necessarily discrete.

Consider an inclusion $\N\subset\M$, where $\N$ is an infinite factor and $\M$ is a properly infinite von Neumann algebra on a separable Hilbert space $\Hilb$. If $E(\M,\N)\neq \emptyset$ denote by $\hat E \in P(\M_1,\M)$ the normal semifinite faithful operator-valued weight dual to $E\in E(\M,\N)$, see \cite{Kos86}, \cite{ILP98}, \cite{FiIs99}. 

\begin{definition}\label{def:discreteinclusion}\cite{ILP98}.
In the above notation, the inclusion $\N\subset\M$ is called \textbf{discrete} if $E(\M,\N)\neq \emptyset$ (semidiscreteness) and $\hat E_{\restriction \N' \cap \M_1}$ is semifinite for some (hence for all) $E\in E(\M,\N)$. 
\end{definition}

\begin{proposition}\label{prop:discreteiffintert}
Let $\N$ be an infinite factor with separable predual. Then a semidiscrete extension $\N\subset\M$ can be characterized as in Theorem \ref{thm:FiIs} by a generalized Q-system of intertwiners in $\C=\End(\N)$ (Definition \ref{def:genQsysofinterts}) if and only if it is discrete.
\end{proposition}

\begin{proof}
We begin with necessity. Let $(\theta, w, \{m_i\})$ be a generalized Q-system of intertwiners in $\C=\End(\N)$ and consider the dual generalized Q-system of intertwiners $(\gamma, w, \{M_i\})$ as in Definition \ref{def:dualgenQsysofinterts}. By definition $M_i^* e_\N M_i$ are mutually orthogonal projections in $\M_1 = \left\langle \M, e_\N\right\rangle$ which $\sum_i M_i^* e_\N M_i = \oneop$. On one hand, $\M e_\N \M \subset m_{\hat E}$, where $m_{\hat E}$ denotes the domain of $\hat E$, because $\hat E(e_\N) = \oneop$ by \cite[\Lem 3.1]{Kos86}. On the other hand, $M_i^* e_\N M_i \in \N' \cap \M_1$ by the intertwining property of the $M_i$ on $\N$. Hence $M_i^* e_\N M_i$ are mutually orthogonal projections which sum up to $\oneop$ in the domain of 
$\hat E_{\restriction \N' \cap \M_1} \in P(\N' \cap \M_1, \N' \cap \M)$. This is equivalent to semifiniteness of $\hat E_{\restriction \N' \cap \M_1}$ by \cite[\Lem 3.2]{FiIs99}, see also \cite[\Lem 2.2]{HKZ91}, hence to discreteness of $\N\subset\M$. The same is true if $\N\subset\M$ is an arbitrary semidiscrete inclusion of von Neumann algebras with separable predual. 

The converse implication relies on deep results on the structure of $\N'\cap\M_1$ due to \cite{ILP98}. Consider a discrete inclusion $\N\subset\M$ where $\N$ is a factor, $\M$ a von Neumann algebra, and choose $E\in E(\M,\N)$. Then $\M_1 = \left\langle \M, e_\N\right\rangle$ is a factor and $\N\subset\M_1$ a subfactor. By the same argument leading to \cite[\Prop 2.8]{ILP98} we get a decomposition of $\N'\cap\M_1$ as a direct sum of four algebras, where only the first survives by discreteness assumption and because $\Ad_{J_{\M,\Omega}}(\N'\cap\M_1) = \N'\cap\M_1$, \cf comments after \cite[\Def 3.7]{ILP98}. In particular $\N'\cap\M_1$ is a direct sum of type $I$ factors and $P\N \subset P \M_1 P$ has finite index for every finite rank projection $P\in\N' \cap\M_1$ by \cite[\Lem 2.7 $(ii)$]{ILP98}. Now, arguing as in the proof of \cite[\Thm 3.5]{FiIs99} and using \cite[\Lem 3.2]{FiIs99}, see also \cite[\Prop 3.2 $(ii) \Rightarrow (i)$]{ILP98}, by discreteness we can write $\oneop = \sum_i P_i$, $i\in I$, where $P_i\in\N' \cap \M_1$ are non-trivial mutually orthogonal projections such that $P_i\in m_{\hat E}$.
Each $P_i$ gives rise to a subendomorphism of the dual canonical endomorphism $\theta\in\End(\N)$ of $\N\subset\M$. Indeed, $P_i$ and $e_\N$ are infinite projections in $\M_1$ because $\N$ is an infinite factor, \cf \cite[\Lem 3.1]{FiIs99}, hence we can choose partial isometries $W_i \in \M_1$ such that $W_i^*W_i = P_i$, $W_iW_i^* = e_\N$. Then 
$$W_i n W_i^* = \rho_i(n)e_\N,\quad n\in\N$$ 
defines $\rho_i\in\End(\N)$, $\rho_i\prec\theta$, because $e_\N\M_1e_\N\ = \N e_\N$, \cf \cite[\Lem 3.1]{ILP98}. The endomorphism $\rho_i$ has finite index, \ie, finite dimension \cite{LoRo97}, whenever $P_i$ has finite rank in $\N'\cap\M_1$, indeed the inclusion $P_i\N \subset P_i \M_1 P_i$ is isomorphic to $\rho_i(\N)\subset\N$.
Moreover, $\theta = \oplus_i \rho_i$. From $W_i = e_\N W_i P_i$ we get $W_iP_i\in n_{\hat E}$ because $n_{\hat E}$ is a left ideal, and $W_i \in m_{\hat E} = n_{\hat E}^* n_{\hat E}$. By the push-down lemma \cite[\Lem 2.2]{ILP98} generalized to non-factorial inclusions \cite[\Lem 3.3]{FiIs99} we can write $W_i = e_\N \psi_i$, where $\psi_i\in\M$, $\psi_i := \hat E(W_i)$. One can check that $\psi_i$ is a \emph{charged field} for $\rho_i$, indeed
$$\psi_i n = \hat E(W_i P_i n) = \hat E(W_i n P_i) = \hat E(\rho_i(n) e_\N W_i) = \rho_i(n) \psi_i,\quad n\in\N$$
and that $e_\N E(\psi_i\psi_i^*) = e_\N \psi_i \psi_i^* e_\N = W_iW_i^* = e_\N$, hence $E(\psi_i\psi_i^*) = \oneop$ because $n \mapsto n e_\N$ is an isomorphism of $\N$ onto $\N e_\N$. Moreover $P_i = \psi_i^* e_\N \psi_i$, hence $\{\psi_i\}$ is a Pimsner-Popa basis for $\N\subset\M$ with respect to $E$. In particular, $\M = \left\langle \N, \{\psi_i\} \right\rangle$ as in the proof of \cite[\Thm 4.1]{FiIs99}, see also \cite[\Sec 1.1.4]{Pop95}, \cite[\Lem 3.8]{ILP98}.

Now, chosen a canonical endomorphism $\gamma$ for $\N\subset\M$, thanks to \cite[\Prop 5.1]{Lon89} there is an isometry $w\in\N$ such that $w\in\Hom_{\End(\N)}(\id,\theta)$, where $\theta := \gamma_{\restriction \N}$, $E(m) = w^*\gamma(m)w$ for every $m\in\M$, and $e_\N = \gamma_1^{-1}(ww^*)$. Define
$$w_i := \gamma(\psi_i^*)w,\quad M_i := w_i \psi_i,\quad i\in I$$
where $w_i\in\N$ are such that $w_i\in\Hom_{\End(\N)}(\rho_i,\theta)$, \cf \cite[\Sec 5]{LoRe95}, and $M_i\in\M$ have the desired intertwining property with $\theta$, namely $M_i n = \theta(n) M_i$, $n\in\N$, \cf Definition \ref{def:dualgenQsysofinterts}. Observe that $w_i$ are non-trivial isometries $w_i^*w_i = E(\psi_i \psi_i^*) = \oneop$ and that $P_i = \psi_i^* e_\N \psi_i = M_i^* e_\N M_i$ because $e_\N\in\N'$. As a consequence $\{M_i\}$ is another Pimsner-Popa basis for $\N\subset\M$ with respect to $E$ and $E(M_iM_i^*) = w_i E(\psi_i \psi_i^*) w_i^* = w_iw_i^*$.
Setting
$$m_i := \gamma(M_i),\quad i\in I$$
we have that $m_i\in\N_1 := \gamma(\M) = \left\langle \theta(\N), \{m_i\} \right\rangle$ fulfill $m_i\in\Hom_{\End(\N)}(\theta,\theta^2)$, 
$$p_i := \gamma_1(P_i) = m_i^*ww^*m_i$$ 
are mutually orthogonal projections in $\N$ which $\sum_i p_i = \oneop$, and $n w = 0 \Rightarrow n = 0$ for $n\in \N_1$ follows immediately from faithfulness of $E':= \gamma \circ E \circ \gamma^{-1}$.
Hence $(\theta,w,\{m_i\})$ is a generalized Q-system of intertwiners in $\C=\End(\N)$ associated, in the sense of Theorem \ref{thm:FiIs}, to the discrete inclusion $\N\subset\M$.
\end{proof}

\begin{remark}
With these normalizations for $w$ and $\psi_i$, $i\in I$, we have that $\rho_i(n) = E(\psi_i n \psi_i^*)$, $n\in\N$, \ie, $\rho_i$ is implemented by a single charged field $\psi_i$ via the conditional expectation $E$, \cf \cite[\Sec 5]{LoRe95}, \cite[\Sec 4.4]{BKLR16}.
\end{remark}

\begin{proposition}\label{prop:uniquepipoexpansion}
Let $\N\subset \M$ be a discrete inclusion as in Definition \ref{def:discreteinclusion} and $\{\psi_i\}\subset\M$ a Pimsner-Popa bases of charged fields as in the proof of Proposition \ref{prop:discreteiffintert}. Then for every $m\in\M$, the coefficients $E(\psi_i m)\in\N$ in the Pimsner-Popa expansion (Proposition \ref{prop:pipoexpansion})
$$m=\sum_i \psi_i^* E(\psi_i m)$$
are uniquely determined.
\end{proposition}

\begin{proof}
We have already checked in the proof of Proposition \ref{prop:discreteiffintert} that $E(\psi_i\psi_i^*) = \oneop$ for every $i\in I$, hence we can apply Proposition \ref{prop:pipoexpansion}.
\end{proof}

\begin{proposition}\label{prop:discreteminimal}
Let $\N$ be an infinite factor with separable predual and $\N\subset\M$ a discrete extension as in Proposition \ref{prop:discreteiffintert}. Fix a conditional expectation $E\in E(\M,\N)$ and a canonical endomorphism $\gamma$ with dual canonical endomorphism $\theta = \gamma_{\restriction \N}$. Then a generalized Q-system of intertwiners $(\theta,w,\{m_i\})$ can be chosen such that the set of indices $I$ labels the irreducible subsectors $[\rho_i]$ (necessarily with finite dimension) of $[\theta]$, counted with (arbitrary) multiplicity. There is a distinguished label $0\in I$, corresponding to one occurrence of the identity sector $[\id]$, such that 
\begin{equation}\label{eq:mzero} m_0 = \theta(w) \end{equation} 
\ie
$$
\tikzmath{
	\draw[ultra thick]
	(0,-14) node [below] {$\theta$}--(0,-10)..
	controls (0,0) and (10,0)
	..(10,-10)--(10,-14) node [below] {$\theta$};
	\draw[ultra thick]
	(5,7) node [above] {$\theta$}--(5,-2.5) node [] {\textbullet};
	\draw (5,0.5) node [right] {$0$};
	}\,= 
	\hspace{-1mm}
\tikzmath{
	\draw[ultra thick]
	(0,-14) node [below] {$\theta$}--(0,7) node [above] {$\theta$};
	\draw[ultra thick]
	(9,-14) node [below] {$\theta$}--(9,-2.5) node [] {\textbullet};
	}.
$$
\end{proposition}

\begin{proof}
$\N'\cap\M_1$ is a direct sum of type $I$ factors by discreteness assumption. Hence we can refine the family of orthogonal projections $P_i$ encountered in the proof of the previous proposition such that each $P_i$ is minimal in $\N'\cap\M_1$ and again in the domain of $\hat E$, thus each $\rho_i \prec \theta$, $i\in I$, is irreducible (with finite index). Every subsector of $[\theta]$ arises in this way and $\theta = \oplus_i \rho_i$.

The second statement follows by observing that the Jones projection $e_\N$ is minimal in $\N'\cap\M_1$ if and only if $\N$ is a factor, if and only if $\id$ is irreducible as an object (tensor unit) of $\End(\N)$. Now, by \cite[\Lem 2.2, \Prop 2.4]{HKZ91} we can assume that $P_0 = e_\N$ and choose $W_0 = e_\N$, hence $\psi_0 = \hat E(e_\N) = \oneop$ and $w_0 = M_0 = w$, \ie, $m_0 = \theta(w)$.
\end{proof}

\begin{remark}\label{rmk:discreteandirred}
In the assumptions of Proposition \ref{prop:discreteiffintert}, discreteness of the inclusion $\N\subset\M$ implies 
$$[\theta] = \oplus_i [\rho_i]$$ 
where $[\rho_i]$ are irreducible subsectors with \emph{finite dimension} and counted with (arbitrary) multiplicity in the 
set of indices $I$, \cf comments after \cite[\Def 3.7]{ILP98}. 

If in addition $\N\subset\M$ is \emph{irreducible}, \ie, $\N'\cap\M=\CC\oneop$, then the multiplicity of each $[\rho_i]$ in $[\theta]$ is \emph{finite} and bounded above by the square of the dimension of $[\rho_i]$, see \cite[\Thm 3.3, \App]{ILP98}.
\end{remark}

\begin{remark}\label{rmk:discretespecialpipo}
The Pimsner-Popa elements $\{m_i\}\subset\N$, or equivalently $\{M_i\}\subset\M$ (Definition \ref{def:dualgenQsysofinterts}), constructed from discrete inclusions via charged fields as in Proposition \ref{prop:discreteiffintert} have the following additional properties. Compute 
$ w^* m_i = w^* \gamma(w_i \psi_i) = w_i w^* \gamma(\psi_i) = w_i w_i^*$, hence
\begin{equation}\label{eq:pallinoasx} w^* m_i = m_i^*w = w_i w_i^* \end{equation}
\ie
$$
\tikzmath{
	\draw[ultra thick]
	(0,-10) node [] {\textbullet}--(0,-10)..
	controls (0,0) and (10,0)
	..(10,-10)--(10,-17) node [below] {$\theta$};
	\draw[ultra thick]
	(5,7) node [above] {$\theta$}--(5,-2.5) node [] {\textbullet};
	\draw (5,0.5) node [right] {$i$};
	}\,= 
	\hspace{-1mm}
\tikzmath{
	\draw[ultra thick]
	(0,10) node [] {\textbullet}--(0,10)..
	controls (0,0) and (10,0)
	..(10,10)--(10,17) node [above] {$\theta$};
	\draw[ultra thick]
	(5,-7) node [below] {$\theta$}--(5,2.5) node [] {\textbullet};
	\draw (5,-0.5) node [right] {$i$};
	}\,= 
\tikzmath{
	\draw[ultra thick] (0,19) node [above] {$\theta$}--(0,-19) node [below] {$\theta$};
	\fill[color=white] (-4,4) rectangle (4,12);
    	\draw (-4,4) rectangle node {$w_i^*$} (4,12);
	\fill[color=white] (-4,-4) rectangle (4,-12);
    	\draw (-4,-4) rectangle node {$w_i$} (4,-12);
	\draw (3.7,0) node {$\rho_i$};
	}\,,\quad\qquad
\tikzmath{
	\draw[ultra thick]
	(0,-10) node [] {\textbullet}--(0,-10)..
	controls (0,0) and (10,0)
	..(10,-10)--(10,-17) node [below] {$\theta$};
	\draw[ultra thick]
	(5,7) node [above] {$\theta$}--(5,-2.5) node [] {\textbullet};
	\draw (5,0.5) node [right] {$i$};
	}\,=:
\tikzmath{
	\draw[ultra thick] (5,12) node [above] {$\theta$}--(5,-12) node [below] {$\theta$};
	\draw[ultra thick] (-3,0) node [] {\textbullet} -- (5,0) node [] {\textbullet};
	\draw (5,0.5) node [right] {$i$};
	}\,=
\tikzmath{
	\draw[ultra thick] (5,12) node [above] {$\theta$}--(5,-12) node [below] {$\theta$};
	\draw[ultra thick] (-3,3) node [] {\textbullet} -- (5,3) node [] {\textbullet};
	\draw (5,3.5) node [right] {$i$};
	\draw[ultra thick] (-3,-3) node [] {\textbullet} -- (5,-3) node [] {\textbullet};
	\draw (5,-2.5) node [right] {$i$};
	}
$$
for every $i\in I$. Consider the spatial isomorphism $\theta(\N) \subset \N = \gamma_1(\N \subset \M_1)$
such that $\theta(\N)' \cap \N = \gamma_1(\N' \cap \M_1)$, where $\gamma_1$ is the canonical endomorphism for $\M\subset\M_1$ dual to $\gamma$. 
From $\gamma_1(P_i) = w_iw_i^* = E(M_iM_i^*)$ we conclude that 
$$p_i = q_i$$
where $p_i =\gamma_1(P_i)$, and $q_i = E(M_iM_i^*)$ are defined in Lemma \ref{lem:pi-poequiv}. In particular $q_i$, $i\in I$, are mutually orthogonal projections in $\N$ such that $\sum_i q_i = \oneop$ as well. 

If we consider $\{M_i\}$ constructed as in Proposition \ref{prop:discreteminimal} we have in addition $w^* m_0 = w_0w_0^* = ww^*$ and
\begin{equation}\label{eq:pallinoadx} \theta(w^*) m_i = \delta_{i,0} \oneop \end{equation}
\ie
$$
\tikzmath{
	\draw[ultra thick]
	(0,-17) node [below] {$\theta$} --(0,-10)..
	controls (0,0) and (10,0)
	..(10,-10)--(10,-10) node [] {\textbullet};
	\draw[ultra thick]
	(5,7) node [above] {$\theta$}--(5,-2.5) node [] {\textbullet};
	\draw (5,0.5) node [right] {$i$};
	}\,=\;
\delta_{i,0}\hspace{-1mm}
\tikzmath{
	\draw[ultra thick] (0,12) node [above] {$\theta$} -- (0,-12) node [below] {$\theta$};	
}
$$
for every $i\in I$. Moreover
$$E(M_i) = E(M_iM_0^*)M_0 = \delta_{i,0} M_0.$$ 
\end{remark}

\begin{definition}\label{def:unitalqsys}
We say that a generalized Q-system of intertwiners (Definition \ref{def:genQsysofinterts}) is \textbf{unital}, if it satisfies in addition the analogue of equations (\ref{eq:mzero}), (\ref{eq:pallinoasx}), (\ref{eq:pallinoadx}), namely 
$$m_0 = \theta(w), \quad w^*m_i = m_i^* w = (w^* m_i)^2, \quad \theta(w^*)m_i = \delta_{i,0} \oneop$$
for every $i\in I$, and for a distinguished label $0\in I$.
\end{definition}

One can easily check that the braided product (Definition \ref{def:brprodgenQsysofinterts}) of two unital generalized Q-systems of intertwiners is again unital.

%%%
\section{Generalized Q-systems of intertwiners for local nets}\label{sec:gennetQsysofinterts}
%%%

Let $\{\A\} = \{\O\in\K \mapsto \A(\O)\}$ be a \emph{net} of infinite von Neumann factors (typically of type $\III_1$) over a partially ordered by inclusion and \emph{directed} set $\K$ of open bounded regions $\O$ of spacetime (\eg, the set of open proper bounded intervals $\O = I\subset\RR$, or double cones in Minkowski space $\O\subset \RR^{n+1}$, $n\geq 1$). A net is called \emph{isotonous} if $\O\subset\tilde\O$ implies $\A(\O)\subset\A(\tilde\O)$, and \emph{local} if $\A(\O)$ and $\A(\tilde\O)$ commute elementwise whenever $\O \subset \tilde\O '$, where $\tilde\O'$ denotes the space-like complement of $\tilde\O$ in $\RR^{n+1}$, $n\geq 1$, or the interior of the complement of the interval $\tilde\O = \tilde I$ in $\RR$. 

\begin{definition}\label{def:localnet}
A net $\{\A\}$ as above fulfilling isotony and locality is called a \textbf{net of local observables}, also abbreviated as \textbf{local net}. 
\end{definition}

We refer to \cite{HaagBook}, \cite[\Sec 3]{LoRe95}, \cite[\Ch 5]{NCGlec04} for more explanations and for the physical motivations behind this notion.

Now, let $\{\A\}$ be realized on a separable Hilbert space $\Hilb_0$ (vacuum space) and assume the existence of a unit vector $\Omega_0\in\Hilb_0$ (vacuum vector) which is cyclic and separating for each local algebra $\A(\O)$. In this case, we say that $\{\A\}$ is a \emph{standard} net on $\Hilb_0$ with respect to $\Omega_0$ and denote by $\omega_0 := (\Omega_0 | \cdot\Omega_0)$ the vacuum state of the net. We say that \emph{Haag duality} holds for $\{\A\}$ in the vacuum space if
$$\A(\O')' = \A(\O)$$
for every $\O\in\K$, where $\A(\O')$ is the \Cstar-algebra generated by all $\A(\tilde\O)$, $\tilde\O\in\cK$, $\tilde\O\subset\O'$. 

Denote by $\DHR\{\A\}\subset\End(\A)$ the category of \textbf{DHR endomorphisms} of the net, see \cite{DHR71}, \cite{DHR74}, \cite{FRS92}, and by $\A$ the \textbf{quasilocal algebra}, \ie, the \Cstar-algebra generated by $\{\A\}$. In the following we shall be interested in two distinguished subcategories of the DHR category, namely

\begin{definition}\label{def:DHRf-d}
Denote by $\DHR_f\{\A\}$ and $\DHR_d\{A\}$ the full subcategories of $\DHR\{\A\}$ whose objects are, respectively, \emph{finite-dimensional} DHR endomorphisms and (possibly infinite, countable) \emph{direct sums} of those.  
\end{definition}

More precisely the most general object $\rho$ in $\DHR_d\{\A\}$ arises as follows. Let $\rho_i$ be a family of at most countably many irreducible finite-dimensional DHR endomorphisms which can be localized in $\O\in\cK$. Let $\{w_i\}$ be a (possibly infinite) Cuntz family of isometries in $\A(\O)$ such that $w_i w_i^*\in\rho(\A)'\cap\A(\O)$ are mutually orthogonal projections and $\sum_i w_i w_i^* = \oneop$. Then $\rho = \sum_i \Ad_{w_i} \rho_i$, where $\rho_i =: \Ad_{w^*_i} \rho$ and the sum converges elementwise in the strong operator topology because $\rho = \sum_i w_iw_i^* \rho(\cdot) \sum_j w_jw_j^* = \sum_i w_i \rho_i(\cdot) w_i^*.$
Similarly, the most general arrow $t$ between objects $\rho,\sigma$ in $\DHR_d\{\A\}$ can be written as $t = \sum_j v_jv_j^* t \sum_i w_iw_i^* = \sum_{i,j} v_j t_{j,i} w_i^*$ where $\{w_i\}$, $\{v_j\}$ are Cuntz families, respectively, for $\rho$, $\sigma$ and $t_{j,i} := v^*_j t w_i$ are arrows from $\rho_i$ to $\sigma_j$.

\begin{remark}
Observe that $\DHR_f\{\A\}\subset\DHR_d\{\A\}\subset\DHR\{\A\}$ and each inclusion is full, replete and stable under (finite) direct sums and subobjects. The first two categories are \emph{semisimple} in the sense that every object can be written as a (possibly infinite) direct sum of irreducible finite-dimensional objects.
\end{remark}

The following is the net-theoretic version of Definition \ref{def:genQsysofinterts}, and generalizes the notion of Q-system for local nets given in \cite[\Sec 4]{LoRe95}.

\begin{definition}\label{def:netgenQsysofinterts}
Let $\{\A\}$ be a local net fulfilling Haag duality as above. A \textbf{generalized net Q-system of intertwiners} in $\C = \DHR\{\A\}$ 
is a triple $(\theta, w, \{m_i\})$ consisting of a DHR endomorphism $\theta$ in $\DHR\{\A\}$, an isometry $w\in\Hom_{\DHR\{\A\}}(\id,\theta)$, and a family $\{m_i\}\subset\A$ indexed by $i$ in some set $I$, such that
\begin{itemize}
\item[$(i)$] $p_i := m_i^* ww^* m_i$ are mutually orthogonal projections in $\A$, \ie, $p_i p_j^* = \delta_{i,j} p_i$, such that $\sum_i p_i = \oneop$.
\item[$(ii)$] $a w = 0 \Rightarrow a = 0$ if $a\in \A_1(\O) := \left\langle\theta(\A(\O)), \{m_i\}\right\rangle$ for some localization region $\O\in\K$ of $\theta$ and for any other $\tilde\O\in\K$ such that $\O \subset \tilde\O$.
\item[$(iii)$] $m_i \in\Hom_{\DHR\{\A\}}(\theta,\theta^2)$, $i\in I$.
\end{itemize}
\end{definition}

\begin{remark}\label{rmk:netQsysislocalQsys}
By the localization property of $\theta$ and by Haag duality, $(\theta, w, \{m_i\})$ is a generalized Q-system of intertwiners in $\End(\A(\tilde\O))$ (Definition \ref{def:genQsysofinterts}) for every $\tilde\O$ as above. Indeed, $\DHR\{\A\}$ sits into $\End(\A(\tilde\O))$ via the restriction functor as a (full if local intertwiners are global), replete and \emph{braided} tensor subcategory for every such $\tilde\O$, \cf \cite[\Sec 3]{GiRe15}. 
\end{remark}

\begin{remark}
Condition $(iii)$ in Definition \ref{def:netgenQsysofinterts}, in view of Proposition \ref{prop:discreteiffintert}, excludes many interesting infinite index extensions of local nets. Notably the Virasoro net $\{\Vir_c\}$ in one spacetime dimension, which sits in every conformal (diffeomorphism covariant) net, gives often rise to infinite index \emph{semidiscrete} but \emph{non-discrete} extensions if $c > 1$, see \cite{Reh94vir}, \cite{Car04}, \cite{Xu05}.
It is however fulfilled in many examples of chiral conformal embeddings with infinite index, see Section \ref{sec:exampleU(1)}, as in compact orbifold theories in low and higher dimensions, see \cite{Xu00orb} and \cite{DoRo90}, and of course in every finite index extension. 
\end{remark}

\begin{definition}\cite{LoRe95}\label{def:inclnets}.
An \textbf{inclusion of nets} is defined by two isotonous nets of von Neumann algebras $\{\A\}$, $\{\B\}$ over the same partially ordered set of spacetime regions $\K$ and realized on the same separable Hilbert space $\Hilb$ such that $\A(\O) \subset \B(\O)$ for every $\O\in\K$. In this case, we write 
$$\{\A \subset \B\}$$
and call $\{\B\}$ an \textbf{extension} of $\{\A\}$. The inclusion is called \emph{irreducible} if $\A(\O)'\cap\B(\O) = \CC\oneop$ for every $\O\in\K$. The net $\{\B\}$ is \emph{relatively local} with respect to $\{\A\}$ if $\B(\O)\subset\A(\O')'$ for every $\O\in\K$. If $\{\A\}$ is local, $\{\B\}$ will be always implicitly assumed to be relatively local with respect to $\{\A\}$.

The inclusion of nets is called \emph{standard} if there is a vector $\Omega\in\Hilb$ which is standard for $\{\B\}$ on $\Hilb$ and for $\{\A\}$ on a subspace $\Hilb_0\subset\Hilb$. A normal faithful \emph{conditional expectation} $E$ of $\{\B\}$ onto $\{\A\}$ is a family indexed by $\O\in\K$ of normal faithful conditional expectations $E_\O \in E(\B(\O),\A(\O))$ which respect inclusions, namely such that ${E_{\tilde\O}}_{\restriction \B(\O)} = E_\O$ if $\O\subset\tilde\O$. A normal faithful \emph{state} $\omega$ of $\{\B\}$ is a conditional expectation of $\{\B\}$ onto the trivial net $\{\CC\}$ and $E$ as above is called \emph{standard} if it preserves the standard vector state $\omega:=(\Omega|\cdot \Omega)$ of the net, namely $\omega_\O \circ E_\O = \omega_\O$ for every $\O$. We say that the extension $\{\A\subset\B\}$ is \emph{discrete} if $\A(\O)\subset\B(\O)$ is discrete (Definition \ref{def:discreteinclusion}) for every $\O\in\K$.
\end{definition}

The following theorem extends the results of \cite[\Thm 4.9]{LoRe95} to the case of infinite index discrete inclusions of nets of von Neumann algebras.

\begin{theorem}\label{thm:discreteextnets}
Let $\{\A\}$ be a local net fulfilling Haag duality and standardly realized on $\Hilb_0$ as in the beginning of this section. Then a generalized net Q-system of intertwiners $(\theta, w, \{m_i\})$ in $\C=\DHR\{\A\}$ (Definition \ref{def:netgenQsysofinterts}) which is also unital (Definition \ref{def:unitalqsys}) gives an isotonous net of von Neumann algebras $\{\B\}$ such that $\{\A\subset \B\}$ is a discrete standard inclusion of nets with a normal faithful standard conditional expectation. The net $\{\B\}$ is always relatively local with respect to $\{\A\}$, and it is itself local if and only if $\theta(\eps_{\theta,\theta}) m_i m_j = m_j m_i$ for every $i,j\in I$, where $\eps$ denotes the DHR braiding.   
\end{theorem}

\begin{proof}
Let $\O\in\K$ be a localization region of the DHR endomorphism $\theta$, call $\N:=\A(\O)$ and $\theta \equiv \theta_{\restriction\N}\in\End(\N)$ the restriction of $\theta$ to $\N$, and observe that $w\in\N$, $m_i\in\N$ for every $i\in I$ by Haag duality. From Theorem \ref{thm:FiIs} we get $\N_2 \subset \N_1 \subset\N$ with a normal faithful conditional expectation $E':=\theta(w^*\cdot w)\in E(\N_1,\N_2)$ and such that $\theta$ is a canonical endomorphism for $\N_1\subset\N$. Now $\N$ acts standardly on $\Hilb_0$ by assumption hence $\theta = \Ad_\Gamma$ on $\N$, where $\Gamma := J_{\N_1,\Phi}J_{\N,\Phi}$ and $\Phi\in\Hilb_0$ is cyclic and separating for $\N_1$ and $\N$. Let $\M:= \Ad_{\Gamma^*}(\N_1)$ be the corresponding canonical extension of $\N_1\subset\N$ with canonical endomorphism $\gamma := {\Ad_{\Gamma}}_{\restriction \M}$. Lift accordingly the conditional expectation $E:=w^*\gamma(\cdot)w\in E(\M,\N)$ and consider the normal faithful $E$-invariant state $\varphi := \omega_0 \circ E$ of $\M$, where $\omega_0 = (\Omega_0|\cdot\Omega_0)$ is the vacuum state of $\{\A\}$. 
The operators $M_i := \gamma^{-1}(m_i)\in\M$ as in Definition \ref{def:dualgenQsysofinterts} form a Pimsner-Popa basis for $\N\subset\M$ with respect to $E$ and fulfill 
\begin{equation}\label{eq:localintertMi}M_i n = \theta(n) M_i,\quad i\in I,\quad n\in\N.\end{equation}
Now consider the (normal faithful) GNS representation $(\Hilb_\varphi, \pi_\varphi, \Omega_\varphi)$ of $\M$ with respect to $\varphi = \varphi\circ E$. 
The inclusion $\pi_\varphi(\N)\subset\pi_\varphi(\M)$ on $\Hilb_\varphi$ with conditional expectation $E_\varphi$ given by $E_\varphi(\pi_\varphi(m)) := \pi_\varphi(E(m))$, $m\in \M$, is spatially isomorphic to $\N\subset\M$ on $\Hilb_0$ with respect to $E$. Moreover, $(\Omega_\varphi|E_\varphi(\pi_{\varphi}(m)) \Omega_\varphi) = (\Omega_\varphi|\pi_{\varphi}(m) \Omega_\varphi)$, $m\in\M$ and $e_\N := [\pi_\varphi(\N)\Omega_\varphi]$ is the associated Jones projection. 
By spatial isomorphism we have that $\{\pi_\varphi(M_i)\}$ is a Pimsner-Popa basis for $\pi_\varphi(\N)\subset\pi_\varphi(\M)$ with respect to $E_\varphi$, and $\gamma_\varphi$ given by $\gamma_\varphi(\pi_\varphi(n)) := \pi_\varphi(\gamma(n))$, $n\in \N$, is a canonical endomorphism with dual canonical $\theta_\varphi := {\gamma_\varphi}_{\restriction\pi_\varphi(\N)}$.
In particular, we have a direct sum decomposition
$$\Hilb_\varphi = \overline{\sum_i \pi_\varphi(M_i^*) e_\N \Hilb_\varphi}$$
where $\pi_\varphi(M_i^*)e_\N$, $i\in I$, are partial isometries with mutually orthogonal range and domain projections and we let
$$\Hilb_{0,\N,\varphi} := e_\N \Hilb_\varphi.$$
Every $n\in \N\subset\M$ acts by left multiplication on $\Hilb_\varphi$, then
\begin{equation}\label{eq:GNSofN}\pi_\varphi(n)\sum_i \pi_\varphi(M_i^*)\psi_i = \sum_i \pi_\varphi(M_i^*)\pi_\varphi(\theta(n))\psi_i\end{equation}
where $\psi = \sum_i \pi_\varphi(M_i^*)\psi_i$, with $\psi_i :=  e_\N \pi_\varphi(M_i) \psi \in \pi_\varphi(q_i) e_\N\Hilb_\varphi$, is the generic vector of $\Hilb_\varphi$. As in the proof of \cite[\Thm 4.9]{LoRe95}, this representation of $\N=\A(\O)$ extends to the whole net. Indeed, the linear map
$$U_0 : n\Omega_0\mapsto\pi_\varphi(n)\Omega_\varphi, \quad n\in\N$$
extends to a unitary operator from $\Hilb_0$ onto $\Hilb_{0,\N,\varphi}$, 
due to $\varphi \equiv \omega_0 \circ E$ and $E(n)=n$, $n\in\N$, which implements ${\pi_\varphi}_{\restriction\N}$ on the subspace $\Hilb_{0,\N,\varphi}$ via adjoint action. 
For every quasilocal observable $a\in\A$ and $\psi_i$ as above, define
$$\pi_\varphi(a)\sum_i \pi_\varphi(M_i^*)\psi_i := \sum_i \pi_\varphi(M_i^*) U_0 \theta(a) U_0^*\psi_i.$$
One can check that $\pi_\varphi$ is a well-defined bounded and locally normal representation of $\A$ on $\Hilb_\varphi$ which extends the GNS representation restricted to $\N$ due to equation (\ref{eq:GNSofN}). In this representation, the intertwining relation (\ref{eq:localintertMi}) extends to the net, namely
\begin{equation}\label{eq:globalintertMi}\pi_\varphi(M_i)\pi_\varphi(a) = \pi_\varphi(\theta(a)) \pi_\varphi(M_i), \quad i\in I,\quad a\in\A.\end{equation}
To show this, we first check that in the representation on $\Hilb_\varphi$ we have that $e_{\tilde\N}:=[\pi_\varphi(\tilde\N)\Omega_\varphi]$, where $\tilde\N := \A(\tilde\O)$, fulfills
$$e_\N = e_{\tilde\N}$$
for every $\tilde\O\in\K$ (not necessarily $\O\subset\tilde\O$). Indeed, $\Omega_\varphi=\pi_\varphi(M_0^*)\pi_\varphi(w)\Omega_\varphi$ because $M_0^* w = w^*w=\oneop$ by unitality assumption, and the closed linear span in $\Hilb_\varphi$ of vectors of the form
$$\pi_\varphi(a)\Omega_\varphi = \pi_\varphi(a) \pi_\varphi(M_0^*) \pi_\varphi(w)\Omega_\varphi = \pi_\varphi(M_0^*) U_0 \theta(a)wU_0^*\Omega_\varphi$$ 
$$= \pi_\varphi(M_0^*) U_0 w a\Omega_0 = U_0 a \Omega_0$$
does not depend on whether $a\in \N$ or $a\in\tilde\N$ by the intertwining property of $w$ on $\A$ and because $\Omega_0$ is cyclic for every local algebra on $\Hilb_0$ by standardness assumption. Hence $e_{\tilde\N}\Hilb_\varphi = \Hilb_{0,\N,\varphi}\subset\Hilb_\varphi$ for every $\tilde\O\in\K$.

Now let $\psi$ and $\psi_i$ be as in equation (\ref{eq:GNSofN}), and assume that $a\in\A(\tilde\O)$ for some $\O\subset\tilde\O$.
From the \lhs of (\ref{eq:globalintertMi}) we get
$$\pi_\varphi(M_i)\pi_\varphi(a)\psi = \sum_j \pi_\varphi(M_i M_j^*) U_0 \theta(a) U_0^*\psi_j$$ 
because left multiplication is continuous in the GNS representation. 

By Proposition \ref{prop:pipoexpansion} (valid for arbitrary semidiscrete inclusions) we can write
\begin{equation}\label{eq:pipoalgebra}\pi_\varphi(M_i) \pi_\varphi(M_j)^* = \sum_k \pi_\varphi(M_k)^* \pi_\varphi(l_j^{ki}), \quad l_j^{ki} := E(M_kM_iM_j^*)\in\N\end{equation}
where $l_j^{ki} = w^*\gamma(M_kM_iM_j^*)w = w^*m_km_im_j^*w$ and intertwines $\theta$ with $\theta^2$ on the whole net by assumption, \ie 
$$l_j^{ki}\in\Hom_{\DHR\{\A\}}(\theta,\theta^2).$$ 
Recall that the convergence in the \rhs of equation (\ref{eq:pipoalgebra}) is given by the topology induced by the seminorms $\|m\|_\eta^2 = \eta(m^*m)$, $m\in\pi_\varphi(\M)$, with $\eta$ any normal state on $\pi_\varphi(\M)$ such that $\eta = \eta\circ E_\varphi$. Thus
$$\sum_j \pi_\varphi(M_i M_j^*) U_0 \theta(a) U_0^*\psi_j = \sum_{j} \big(\sum_{k}\pi_\varphi(M_k^*) \pi_\varphi(l_j^{ki}) \big) U_0 \theta(a) U_0^*\psi_j$$
and since the vector $U_0 \theta(a) U_0^*\psi_j \in e_{\tilde\N} \Hilb_\varphi = e_{\N} \Hilb_\varphi$ induces a normal $E_\varphi$-invariant state on $\pi_\varphi(\M)$, we get
$$=\sum_{j,k} \pi_\varphi(M_k^*) U_0 \theta^2(a)l_j^{ki} U_0^*\psi_j= \pi_\varphi(\theta(a)) \pi_\varphi(M_i) \psi$$
which is the \rhs of (\ref{eq:globalintertMi}), for every $\psi\in\Hilb_\varphi$, thus the equation is proven. 

We define
$$\B(\O) := \pi_\varphi(\M) \equiv \langle \pi_\varphi(\N), \{\pi_\varphi(M_i)\} \rangle,$$
the crucial point is to extend the construction to bigger regions and define accordingly a coherent family of normal faithful standard conditional expectations with respect to a common cyclic and separating vector. Let $\tilde\O\in\K$ be such that $\O\subset\tilde\O$ and set
$$\B(\tilde\O) := \langle \pi_\varphi(\tilde\N), \{\pi_\varphi(M_i)\} \rangle,\quad \tilde\N := \A(\tilde\O),$$
clearly $\B(\O)\subset\B(\tilde\O)$ holds by isotony of $\{\A\}$. 
 
Now, $\Omega_0$ is separating for every $\tilde \N$, thus $\pi_\varphi(a)\mapsto \pi_\varphi(a) e_\N$ is an isomorphism of $\pi_\varphi(\tilde\N)$ onto $\pi_\varphi(\tilde\N) e_\N$, and because of
$$e_\N \B(\tilde\O) e_\N \subset e_\N \langle\B(\tilde\O),e_\N\rangle e_\N = \e_\N \langle \pi_\varphi(\tilde\N), \pi_\varphi(\M), e_\N\rangle e_\N = \pi_\varphi(\tilde\N) e_\N$$ 
provided $\O\subset\tilde\O$, we can define by
$$\label{eq:condexpfromjonesproj}e_\N b\hspace{0.3mm} e_\N = E_\varphi(b) e_\N, \quad b\in\B(\tilde\O)$$
a conditional expectation of $\B(\tilde\O)$ onto $\pi_\varphi(\tilde\N)$ (for arbitrarily big region $\tilde\O$) which extends the one previously given on $\B(\O) \equiv \pi_\varphi(\M)$. $E_\varphi$ is clearly normal and fulfills $(\Omega_\varphi | E_\varphi (\cdot) \Omega_\varphi) = (\Omega_\varphi | \cdot \Omega_\varphi)$, while faithfulness remains to be checked, together with the separating property of $\Omega_\varphi$ for $\B(\tilde\O)$ if $\O\subset\tilde\O$ and cyclicity for $\B(\tilde\O)$ if $\tilde\O\subset\O$, where $\B(\tilde\O)$ in this second case is defined below. For an arbitrary region $\tilde\O\in\K$, set
\begin{equation}\label{eq:extlocalalgebras}\B(\tilde\O) := \langle \pi_\varphi(\tilde\N), \{\pi_\varphi(u)\pi_\varphi(M_i)\} \rangle,\quad \tilde\N := \A(\tilde\O)\end{equation}
where $u\in\Hom_{\DHR\{\A\}}(\theta,\tilde\theta)$ is a unitary charge transporter (in $\A$) and $\tilde\theta$ is DHR localizable in $\tilde\O$, \footnote{Notice that we assume DHR sectors to be localizable in \emph{every} region in $\cK$.}. In order to show the desired properties of $\Omega_\varphi$ with respect to these new local algebras we need to introduce more GNS representations. Namely, let
$$(\tilde\theta \equiv \tilde\theta_{\restriction\tilde\N}, \tilde w := uw, \{\tilde m_i := u\theta(u)m_i u^*\})$$
be a generalized Q-system of intertwiners in $\End(\tilde\N)$, see Remark \ref{rmk:netQsysislocalQsys}, and perform the same construction as above on the GNS Hilbert space $\Hilb_{\tilde\varphi}$ of some canonical extension $\tilde\N\subset\tilde\M$ with $\tilde E\in E(\tilde\M,\tilde\N)$ and state $\tilde\varphi := \omega_0 \circ \tilde E$ of $\tilde\M$. Consider
$$\Hilb_{\tilde\varphi} = \overline{\sum_i \pi_{\tilde\varphi}(\tilde{M_i}^*) e_{\tilde\N} \Hilb_{\tilde\varphi}}$$
and
$$\tilde U_0 : \tilde n\Omega_0\mapsto\pi_{\tilde\varphi}(\tilde n)\Omega_{\tilde\varphi},\quad \Hilb_0 \rightarrow \Hilb_{0,\tilde\N,\tilde\varphi} \equiv e_{\tilde\N} \Hilb_{\tilde\varphi},\quad \tilde n\in\tilde\N$$
extended as before to an isometric operator into $\Hilb_{\tilde\varphi}$. Then the linear map defined by
$$U \sum_i \pi_{\tilde\varphi}(\tilde{M_i}^*) \tilde U_0 \phi_i := \sum_i \pi_{\varphi}(M_i^*) U_0 u^*\phi_i, \quad \phi_i\in\Hilb_0$$
sends
$$U \Omega_{\tilde\varphi} = U \pi_{\tilde\varphi}(\tilde{M_0}^*) \tilde U_0 \tilde w \Omega_0 = \pi_{\varphi}(M_0^*) U_0 u^*\tilde w \Omega_0 = \Omega_\varphi$$
because $u^*\tilde w = w$, it extends to a unitary operator from $\Hilb_{\tilde\varphi}$ onto $\Hilb_\varphi$, and fulfills
\begin{equation}\label{eq:intertGNSs}U \pi_{\tilde\varphi}(\tilde\M) U^* = \langle \pi_\varphi(\tilde\N), \{\pi_\varphi(u)\pi_\varphi(M_i)\} \rangle \equiv \B(\tilde\O).\end{equation}
Indeed, if $\tilde n \in\tilde\N$ then
$$U \pi_{\tilde\varphi}(\tilde n) U^*\sum_i \pi_{\varphi}(M_i^*) U_0 \phi_i = U \pi_{\tilde\varphi}(\tilde n) \sum_i \pi_{\tilde\varphi}(\tilde{M_i}^*) \tilde U_0 u\phi_i$$
$$= \sum_i \pi_{\varphi}(M_i^*) U_0 u^* \tilde\theta(\tilde n) u\phi_i = \pi_\varphi(\tilde n) \sum_i \pi_{\varphi}(M_i^*) U_0 \phi_i$$
because $u^* \tilde\theta(\tilde n) u = \theta(\tilde n)$, and
$$U \pi_{\tilde\varphi}(\tilde M_i) U^*\sum_j \pi_{\varphi}(M_j^*) U_0 \phi_j = \sum_{j,k} \pi_{\varphi}(M_k^*) U_0 u^* \tilde{l}_{j}^{ki} u \phi_j$$
where the coefficients $\tilde{l}_{j}^{ki} := \tilde E(\tilde M_k \tilde M_i \tilde M_j^*) = \tilde w \tilde m_k \tilde m_i \tilde m_j^* \tilde w\in\tilde\N$ are analogous to those in (\ref{eq:pipoalgebra}). One can compute $\theta(u^*)u^*\tilde{l}_{j}^{ki}u = l_{j}^{ki}$, hence
$$= \sum_{j,k} \pi_{\varphi}(M_k^*) U_0 \theta(u) l_{j}^{ki} \phi_j = \pi_\varphi(u)\pi_\varphi(M_i)\sum_j \pi_{\varphi}(M_j^*) U_0 \phi_j$$
which proves (\ref{eq:intertGNSs}). By considering this unitary intertwiner for every region $\tilde\O$ we obtain that $\Omega_\varphi$ is cyclic and separating for every $\B(\tilde\O)$ on $\Hilb_\varphi$, in particular $E_\varphi$ is faithful over every $\tilde\O$.

The extension $\{\B\}$ does not depend on the specific choice of unitary charge transporter $u$ made in equation (\ref{eq:extlocalalgebras}). Indeed, by Haag duality any two of them $u$, $v$ differ by $uv^*\in\A(\tilde\O)$. Also, it depends on the choice of the initial localization region $\O$ for $\theta$ and of the extended vacuum state $\varphi$ only up to unitary isomorphism. 

Relative locality of $\{\B\}$ with respect to $\{\A\}$ is always guaranteed by the localization properties of $\theta$ while the statement about locality of $\{\B\}$ follows by the very definition of the DHR braiding. Indeed, $uM_i vM_j = vM_j uM_i$ where $u$ and $v$ are unitaries in $\A$ which transport the localization region of $\theta$ to two mutually space-like regions (respectively left and right localized in low dimensions) if and only if $\eps_{\theta,\theta} M_iM_j = \theta(u^*)v^*u\theta(v) M_iM_j = M_j M_i$ for every $i,j\in I$, and the proof is complete.
\end{proof}

\begin{remark}\label{rmk:*algchargedintert}
If the Pimsner-Popa expansion appearing in equation (\ref{eq:pipoalgebra}) comes from an \emph{irreducible} subfactor, or from a \emph{finite index} inclusion, and if the unital generalized Q-system of intertwiners is defined from charged fields $M_i = w_i \psi_i$, $i\in I$, see Proposition \ref{prop:discreteminimal} and Remark \ref{rmk:discreteandirred}, then the sum over $k$ (a priori convergent in the GNS topology) is \emph{finite} by Frobenius reciprocity among finite-dimensional endomorphisms of $\N$, \cite[\Lem 2.1]{LoRo97}. Indeed, in this case one has $l_j^{ki} = w_k \rho_k(w_i) E(\psi_k\psi_i\psi_j^*)w_j^*$ and $E(\psi_k\psi_i\psi_j^*)\in\Hom_{\End(\N)}(\rho_j,\rho_k\rho_i)$.

In other words, we have a unital \emph{*-algebra of charged intertwiners} with possibly infinitely many generators $\{M_i\}$ but finite (\lqq discrete") fusion rules, \cf \cite[\App A]{LoRe04}. 
\end{remark}

In this section we assumed that a generalized net Q-system of intertwiners (Definition \ref{def:netgenQsysofinterts}) was given and we have shown how to associate to it a relatively local net extension. In Section \ref{sec:discretecase}, given a discrete inclusion of von Neumann algebras, we have seen how to construct a generalized Q-system of intertwiners (Definition \ref{def:genQsysofinterts}). But when do generalized Q-systems of intertwiners exist for discrete relatively local net extensions?

In the finite index setting, in \cite{GuLo92} it is proven that the DHR category restricted to finite index endomorphisms of a chiral CFT is a full and replete subcategory of the category of endomorphisms of a local algebra (\lqq local intertwiners are global"). This of course single-handedly carries over the theory of ordinary Q-systems to such nets. In \cite{LoRe95} it is shown that actually less is needed; in the presence of a coherent conditional expectation, with finite index arguments it can be shown that a Q-system in the DHR category does exist, see \cite[\Cor 3.8, \Cor 3.7 and \Lem 4.1]{LoRe95}. On the other hand, when we consider also infinite-dimensional irreducible sectors, in general it is not true that the category of DHR endomorphisms is a \emph{full} subcategory of the category of endomorphisms of a local algebra, see \cite{Wei08} for a counter-example using $\{\Vir_c\}$, $c>2$. In the absence of these features, we can make some additional assumptions. 

\begin{proposition}\label{prop:netqsys}
Let $\{\mathcal{A}\subset\mathcal{B}\}$ be a discrete relatively local inclusion of nets with standard conditional expectation $E$ (Definition \ref{def:inclnets}) and let $\{\mathcal{A}\}$ be a Haag dual net of type $\III$ factors. Let $\gamma$ be a canonical endomorphism for $\{\A\subset\B\}$ as in \emph{\cite[\Cor 3.3]{LoRe95}} and $\theta\in\DHR\{\A\}$ be its dual canonical endomorphism as in \emph{\cite[\Cor 3.8]{LoRe95}} (localized in $\O_0\in\K$). Moreover, let $w\in\A(\O_0)$ as in \emph{\cite[\Cor 3.7]{LoRe95}} such that $E=w^{*}\gamma(\cdot)w$.

Assume that one of the following two conditions is fulfilled

\begin{itemize}
\item[$(i)$] $\{\A\subset\B\}$ is irreducible and $\theta=\oplus_i \rho_i$ in $\DHR\{\A\}$ where each $\rho_i$ is an irreducible DHR subendomorphism (localized in $\O_0$) and $[\rho_i]$ has finite dimension in $\DHR\{\A\}$. 
\item[$(ii)$] $\Hom_{\DHR\{\A\}} (\theta,\theta) = \Hom_{\End(\A(\O))} (\theta,\theta)$ for every $\O\in\K$, $\O_0\subset\O$, or equivalently $\theta=\oplus_i \rho_i$ in $\DHR\{\A\}$ and $\Hom_{\DHR\{\A\}}(\rho_i,\rho_j) = \Hom_{\End(\A(\O))}(\rho_i,\rho_j)$ for every $i,j\in I$.
\end{itemize}

Then there is a Pimsner-Popa basis $\{\psi_{i}\}$ for the inclusion $\A(\mathcal{O}_{0})\stackrel{E}{\subset}\B(\mathcal{O}_{0})$ consisting of global charged fields, \ie
$$\psi_{i} a= \rho_{i}(a) \psi_{i},\quad a\in\A.$$
Setting $M_{i}:=\gamma(\psi_{i}^{*})w\psi_{i}$ we have another Pimsner-Popa basis for $\A(\mathcal{O}_{0})\stackrel{E}{\subset}\B(\mathcal{O}_{0})$ (\cf Proposition \ref{prop:discreteiffintert}), and $(\theta,w,\{m_i\})$, where $m_i := \gamma(M_i)\in\A(\O_0)$, is a unital generalized net Q-system of intertwiners  in $\DHR\{\A\}$ (Definition \ref{def:netgenQsysofinterts} and \ref{def:unitalqsys}). 
\end{proposition}

\begin{proof}
Assume $(i)$ and let $p_{i}\in\Hom_{\DHR\{\A\}}(\theta,\theta)$ be the projections which determine the decomposition $\theta=\oplus_i \rho_i$, together with orthogonal isometries $w_{i}\in\Hom_{\DHR\{\A\}}(\rho_i,\theta)\cap\A(\O_0)$ such that $w_{i}w_{i}^{*}=p_{i}$.

For $\mathcal{O}\in\mathcal{K}$, let $e_{\mathcal{A}}=[\mathcal{A}(\mathcal{O})\Omega]$ be the Jones projection, and $\B_{1}(\mathcal{O})=\left\langle \mathcal{B}(\mathcal{O}),e_{\mathcal{A}}\right\rangle$ the Jones extension for the inclusion $\mathcal{A}(\O)\subset\mathcal{B}(\mathcal{O})$. By standardness assumption the Jones projections agree for every $\mathcal{O}\in\mathcal{K}$. Denote by $\hat{E}_{\mathcal{O}}\in P(\mathcal{B}_{1}(\mathcal{O}),\mathcal{B}(\mathcal{O}))$ the operator-valued weight dual to $E_{\restriction\mathcal{B}(\mathcal{O})}$ and let $m_{\hat{E}_{\mathcal{O}}}$ and $n_{\hat{E}_{\mathcal{O}}}$ be respectively its domain and definition ideal.
By the same arguments leading to \cite[\Thm 3.2, \Cor 3.3]{LoRe95}, the formula
$$\gamma_{1}(x)e_{\mathcal{A}}=v_{1}xv_{1}^{*},\quad x \in \B_1(\O)$$
where $v_1\in \B_1(\O)$ is an isometry such that $v_1v_1^*= e_\A$, $\gamma_1(v_1) = w$, allows to extend the dual canonical endomorphism $\gamma_{1}:\mathcal{B}_{1}(\mathcal{O}_{0})\rightarrow \mathcal{A}(\mathcal{O}_{0})$ to a map between the quasilocal algebras $\mathcal{B}_{1}$ and $\mathcal{A}$, such that $\gamma_{1 \restriction \mathcal{B}_{1}(\mathcal{O})}$ is the dual canonical endomorphism for the inclusion $\mathcal{A}(\mathcal{O})\subset\mathcal{B}(\mathcal{O})$ for every $\mathcal{O}\in\mathcal{K}$, $\O_0\subset\O$ (\ie, $\gamma_{1\restriction \mathcal{B}_{1}(\mathcal{O})}$ is a canonical endomorphism for $\mathcal{B}(\mathcal{O})\subset\mathcal{B}_{1}(\mathcal{O})$, $\gamma_{1\restriction \mathcal{B}(\mathcal{O})}=\gamma_{\restriction \mathcal{B}(\mathcal{O})}$ and $\gamma_{1}(\B_{1}(\mathcal{O}))=\mathcal{A}(\mathcal{O})$).

By the discreteness assumption and \cite[\Prop 2.8]{ILP98} we have that, for every $\mathcal{O}\in\mathcal{K}$, $\mathcal{B}_{1}(\mathcal{O})\cap\mathcal{A}(\mathcal{O})^{\prime}$ is a direct sum of type $I$ factors, which by irreducibility of $\mathcal{A}(\mathcal{O})\subset\mathcal{B}(\mathcal{O})$ and \cite[\Thm 3.3]{ILP98} are finite dimensional. 

$P_{i}:=\gamma_{1}^{-1}(p_{i})$ is a finite sum of minimal projections in $\B_1(\O)\cap\A(\O)'$, $\O_0\subset\O$, since $\rho_{i}{_{\restriction\A(\O)}}$ has finite dimension by hypothesis, and thus $P_{i}\in n_{\hat{E}_{\mathcal{O}}}$. Indeed, let $z_{P_{i}}$ be the central support of $P_{i}$ in $\mathcal{B}_{1}(\mathcal{O})\cap\mathcal{A}(\mathcal{O})^{\prime}$. $\hat{E}_{\mathcal{O}\restriction z_{P_{i}}(\mathcal{B}_{1}(\mathcal{O})\cap\mathcal{A}(\mathcal{O})^{\prime})z_{P_{i}}}$ is semifinite since $z_{P_{i}}$ can be written as a sum of minimal projections in $n_{\hat{E}_{\mathcal{O}}}$, and thus finite since $z_{P_{i}}(\mathcal{B}_{1}(\mathcal{O})\cap\mathcal{A}(\mathcal{O})^{\prime})z_{P_{i}}$ is finite dimensional. Now, let
$$W_{i}:=\gamma_{1}^{-1}(ww_{i}^{*})\in \mathcal{B}_{1}(\mathcal{O}_{0}).$$
Since $W_{i}=e_{\mathcal{A}}W_{i}P_{i}$ and $P_{i}\in n_{\hat{E}_{\mathcal{O}}}$, we have $W_{i}\in m_{\hat{E}_{\mathcal{O}}}$.
Let $\psi_{i,\mathcal{O}}:=\hat{E}_{\mathcal{O}}(W_{i})$. Exactly as in the proof of Proposition \ref{prop:discreteiffintert}, and using $w_i = \gamma(\psi_{i,\O}^*)w$, we get that $\psi_{i,\mathcal{O}}$ is a charged field for $\rho_{i}$ on $\A(\O)$, \ie, $\psi_{i,\mathcal{O}}a=\rho_{i}(a)\psi_{i,\mathcal{O}}$, $a\in\mathcal{A}(\mathcal{O})$, and the collection $\{\psi_{i,\mathcal{O}}\}$ is a Pimsner-Popa basis for the inclusion $\mathcal{A}(\mathcal{O})\subset\mathcal{B}(\mathcal{O})$. For any $\O\in\K$, $\O_0\subset\O$, by the push-down lemma we have
$$e_{\mathcal{A}}\psi_{i,\mathcal{O}}=W_{i}=e_{\mathcal{A}}\psi_{i,\mathcal{O}_{0}}.$$
Thus applying $\hat{E}_{\mathcal{O}}$ to the above formula, we obtain $\psi_{i,\mathcal{O}}=\psi_{i,\mathcal{O}_{0}}$.
The rest of the proof follows exactly as in Proposition \ref{prop:discreteiffintert}.

Assuming $(ii)$, the proof proceeds along similar lines. By discreteness assumption, one can take projections $p_i\in \theta(\A(\O_0))'\cap\A(\O_0)$ which lay in $n_{\hat{E'}_{\O_0}}$, where $E'_{\O_0} = \gamma \circ E_{\restriction\B(\O_0)} \circ \gamma^{-1}$, and which give a local (in this case also global) decomposition of $\theta = \oplus_i \rho_i$ into DHR subendomorphisms. Now, let $w_i$ such that $p_i = w_iw_i^*$ and define $P_i$, $W_i$ and $\psi_{i,\O_0}$ as before such that $W_i = e_\A \psi_{i,\O_0}$. To conclude it is enough to observe that for any $\O\in\K$, $\O_0\subset\O$, we have $P_i\in n_{\hat{E}_{\O}}$, because $P_i = \psi_{i,\O_0}^* e_A \psi_{i,\O_0}$ by construction and $\psi_{i,\O_0}\in\B(\O_0)\subset\B(\O)$, and also $P_i\in \B_1(\O) \cap \A(\O)'$, because $\theta(\A(\O_0))'\cap\A(\O_0) = \theta(\A(\O))'\cap\A(\O)$ for every such $\O$.
\end{proof}

These assumptions are verified, \eg, for compact group orbifolds \cite{DoRo72}, \cite[\Thm 4.3]{Rob74}, and for theories with a good behaviour with respect to the scaling limit \cite[\Cor 6.2]{DMV04} in 3+1D, and of course for strongly additive CFTs in 1D, \ie, Haag dual nets on $\RR$ \cite[\Lem 1.3]{GLW98}.

%%%
\subsection{Construction of extensions: an alternative way}\label{sec:altext}
%%%

In this section we present an alternative proof of Theorem \ref{thm:discreteextnets} which we feel somewhat more intuitive and which lends itself to describing the braided product of nets in a more direct way.
The basic idea is that a generalized net Q-system of intertwiners $(\theta, w, \{m_i\})$ in $\DHR\{\A\}$ (Definition \ref{def:netgenQsysofinterts}), assuming Haag duality of $\{\A\}$, induces a family of Q-systems (one for every local algebra) each one of which characterizes a local extension. The algebraic structure of the extended net, including a distinguished conditional expectation, is then captured with a coherent inductive procedure and the spatial features of the net are completely determined by the vacuum state.

More precisely, let $\O_0\in\K$ be a reference localization region for $\theta$ and choose a unitary charge transporter $u_{\O}\in\Hom_{\DHR\{\A\}}(\theta,\theta_\O)$ for every $\mathcal{O}\in \mathcal{K}$, where $\theta_{\mathcal{O}} := \Ad_{u_{\mathcal{O}}}\theta$ is localized in $\mathcal{O}$. For every $\O\in\K$, we obtain a generalized Q-system (of intertwiners) in $\End(\mathcal{A}(\mathcal{O}))$ (Definition \ref{def:genQsysofinterts}) by setting
$$(\theta_\mathcal{O}, w_\mathcal{O}, \{{m_i}_\mathcal{O}\}) := (\Ad_{u_\mathcal{O}}\theta, u_\mathcal{O} w, \{u_\mathcal{O}\theta(u_\mathcal{O})m_{i} u_\mathcal{O}^*\}).$$
The faithfulness condition appearing in Definition \ref{def:genQsys} is verified since, for every $\tilde{\mathcal{O}}\in\mathcal{K}$ such that $\mathcal{O}_{0} \cup \mathcal{O}\subset\tilde{\mathcal{O}}$, we have
$$\mathcal{A}_{1}(\mathcal{O}) \subset u_{\mathcal{O}}\langle \theta(\mathcal{A}(\tilde{\mathcal{O}})),m_{i}\rangle u_{\mathcal{O}}^{*} = \Ad_{u_{\mathcal{O}}}(\mathcal{A}_{1}(\tilde{\mathcal{O}})),$$
where 
$$\A_1(\mathcal{O}) := \langle \theta_\O(\mathcal{A}(\mathcal{O})),u_{\mathcal{O}}\theta(u_{\mathcal{O}})m_{i}u_{\mathcal{O}}^{*}\rangle.$$

With this data, we now construct an inductive generalized sequence of net extensions $\{\{\mathcal{B}_{\tilde{\mathcal{O}}}\}, \mathcal{\tilde{O}\in\mathcal{K}}\}$ of $\{\A\}$, indexed by $\tilde\O\in\mathcal{K}$, and defined \emph{only} on regions $\mathcal{O}\in\mathcal{K}$, $\mathcal{O}\subset\tilde{\mathcal{O}}$, which we will then patch together. 
For a fixed $\tilde{\mathcal{O}}\in\mathcal{K}$, $\theta_{\tilde{\mathcal{O}}}=\Ad_{u_{\tilde{\mathcal{O}}}}\theta$ is a canonical endomorphism for 
$\mathcal{A}_{1}(\tilde{\mathcal{O}})\subset \mathcal{A}(\tilde{\mathcal{O}})$ by Theorem \ref{thm:FiIs}, and thus can be implemented on $\mathcal{A}(\tilde{\mathcal{O}})$ by a unitary $\Gamma_{\tilde{\mathcal{O}}}$, namely
$$\theta_{\tilde{\mathcal{O}}}(x)=\Gamma_{\tilde{\mathcal{O}}}x\Gamma_{\tilde{\mathcal{O}}}^{*}\,,\quad x\in\mathcal{A}(\tilde{\mathcal{O}}).$$
Similarly, for every $\mathcal{O}\subset \tilde{\mathcal{O}}$, we have
$$\theta_{\mathcal{O}}(x)=\Ad_{u_{\mathcal{O}}}(\theta(x))= \Ad_{u_{\mathcal{O}}u_{\tilde{\mathcal{O}}}^{*}\Gamma_{\tilde{\mathcal{O}}}}(x)\,,
\quad x\in \mathcal{A}(\tilde{\mathcal{O}}).$$
Now, fixed $\tilde{\mathcal{O}}\in\mathcal{K}$, we define an isotonous net $\{\B_{\tilde\O}\} := \{\mathcal{O}\in \mathcal{K}, \tilde\O\supset\O\mapsto\mathcal{B}_{\tilde{\mathcal{O}}}(\mathcal{O})\}$ by setting
$$\mathcal{B}_{\mathcal{\tilde{O}}}(\mathcal{O}):=\Ad_{(u_{\mathcal{O}}u_{\tilde{\mathcal{O}}}^{*}\Gamma_{\tilde{\mathcal{O}}})^{*}}(\mathcal{A}_{1}(\mathcal{O}))=\langle \mathcal{A}(\mathcal{O}), \{u_{\mathcal{O}} \Ad_{\Gamma_{\tilde{\mathcal{O}}}^{*}u_{\tilde{\mathcal{O}}}}(m_{i})\}\rangle.$$

\begin{remark}
The dual canonical endomorphism $\theta$ for an extension of nets $\{\mathcal{A}\subset\mathcal{B}\}$, \cite[\Cor 3.8]{LoRe95}, is not implemented globally by unitaries. This is clear since by \cite[\Prop 3.4]{LoRe95} the embedding homomorphism $\iota$ of $\{\mathcal{A}\}$ into $\{\mathcal{B}\}$ is equivalent to $\theta$ as a representation and thus would imply the inclusion to be trivial. Of course it is possible to find unitaries which implement $\theta$ locally but the choice of these unitaries in non-unique. The above coherent choice guarantees the isotony of the net $\{\mathcal{B}_{\tilde{\mathcal{O}}}\}$.
\end{remark}

The next step is the construction of an inductive family of embeddings $\iota_{\tilde{\mathcal{O}}_{1},\tilde{\mathcal{O}}_{2}}$ of $\{\mathcal{B}_{\tilde{\mathcal{O}}_{1}}\}$ into $\{\mathcal{B}_{\tilde{\mathcal{O}}_{2}}\}$ with $\tilde{\mathcal{O}}_{1}\subset \tilde{\mathcal{O}}_{2}$. This is straightforward.

\begin{proposition}
Let $\tilde{\mathcal{O}}_{1}\subset \tilde{\mathcal{O}}_{2}$, $\tilde{\mathcal{O}}_{1},\tilde{\mathcal{O}}_{2}\in \mathcal{K}$. The map 
$$\iota_{\tilde{\mathcal{O}}_{1},\tilde{\mathcal{O}}_{2}}:=\Ad_{\Gamma_{\tilde{\mathcal{O}}_{2}}^{*}u_{\tilde{\mathcal{O}}_{2}}u_{\tilde{\mathcal{O}}_{1}}^{*}\Gamma_{\tilde{\mathcal{O}}_{1}}}$$
 is an embedding of $\{\mathcal{B}_{\tilde{\mathcal{O}}_{1}}\}$ into $\{\mathcal{B}_{\tilde{\mathcal{O}}_{2}}\}$, \ie, it sends local algebras $\mathcal{B}_{\tilde{\mathcal{O}}_{1}}(\mathcal{O})$ onto local algebras $\mathcal{B}_{\tilde{\mathcal{O}}_{2}}(\mathcal{O})$ and acts as the identity map on $\mathcal{A}(\mathcal{O})$ for every $\O\in\K$, $\mathcal{O}\subset\tilde{\mathcal{O}}_{1}$.
\end{proposition}

\begin{proof}
Follows by easy direct computation.
\end{proof}

The collection of nets $\{\{\mathcal{B}_{\tilde{\O}}\},\tilde{\O}\in\mathcal{K}\}$ and maps $\{\iota_{\tilde{\mathcal{O}}_{1},\tilde{\mathcal{O}}_{2}}, \tilde{\mathcal{O}}_{1}\subset\tilde{\mathcal{O}}_{1},\tilde{\mathcal{O}}_{1},\tilde{\mathcal{O}}_{2}\in\mathcal{K}\}$ forms an \emph{inductive system}. We can thus take the inductive limit of the $C^{*}$-algebras $\mathcal{B}_{\tilde{\mathcal{O}}}(\tilde{\mathcal{O}})$ from which we obtain a $C^{*}$-algebra $\mathcal{B}$. The subalgebras $\mathcal{B}(\mathcal{O}):=\iota_{\mathcal{O}}(\mathcal{B}_{\mathcal{O}}(\mathcal{O}))$, where $\iota_{\mathcal{O}}$ is the embedding of $\mathcal{B}_{\mathcal{O}}(\mathcal{O})$ into $\mathcal{B}$, are $W^{*}$-algebras since it is easy to see that they have a predual. Thus we have obtained an isotonous net of $W^{*}$-algebras, $\{\mathcal{B}\}$. Now we see that from the data of the Q-system we can also define a consistent family of conditional expectations.

\begin{proposition}\label{prop:altexpectation}
There is a normal faithful conditional expectation from $\{\mathcal{B}\}$ to $\{\mathcal{\A}\}$, \ie, $E_{\mathcal{O}}:\mathcal{B}(\mathcal{O})\rightarrow\mathcal{A}(\mathcal{O})$ for every $\O\in\K$, such that ${E_{\mathcal{O}_{2}}}_{\restriction\mathcal{B}(\mathcal{O}_{1})}=E_{\mathcal{O}_{1}}$ if $\mathcal{O}_{1}\subset \mathcal{O}_{2}$.
\end{proposition}

\begin{proof} 
First define a coherent conditional expectation on  $\{\mathcal{B}_{\tilde{\mathcal{O}}}\}$ for $\tilde{\mathcal{O}}\in\mathcal{K}$. By Theorem \ref{thm:FiIs} we have a conditional expectation $E_\O':=\theta_{\mathcal{O}}(w^{*}u_{\mathcal{O}}^{*}\cdot u_{\mathcal{O}}w)$ for the inclusion $\theta_{\mathcal{O}}(\mathcal{A}(\mathcal{O}))\subset\mathcal{A}_{1}(\mathcal{O})$. $E_\O'$ can be lifted to a conditional expectation $E_{\tilde{\mathcal{O}} \mathcal{O}}$ for the inclusion $\mathcal{A}(\mathcal{O})\subset\mathcal{B}_{\tilde{\mathcal{O}}}(\mathcal{O})$, $\O\subset\tilde\O$, since the two inclusions are isomorphic via $\Ad_{u_{\mathcal{O}}u_{\tilde{\mathcal{O}}}^{*}\Gamma_{\tilde{\mathcal{O}}}}$. Computing explicitly, we get
$$E_{\tilde{\mathcal{O}} \mathcal{O}} = w^{*}u_{\tilde{\mathcal{O}}}^{*}\Gamma_{\tilde{\mathcal{O}}}\cdot \Gamma_{\tilde{\mathcal{O}}}^{*}u_{\tilde{\mathcal{O}}}w$$
which shows that we indeed have a consistent family of conditional expectations on $\{\mathcal{B}_{\tilde{\mathcal{O}}}\}$. Now, to show that these expectations lift to the inductive limit net $\{\B\}$, it is enough to check that ${E_{\tilde{\mathcal{O}}_{2} \mathcal{O}}}_{\restriction\iota_{\tilde{\mathcal{O}}_{1},\tilde{\mathcal{O}}_{2}}(\mathcal{B}_{\tilde{\mathcal{O}}_{1}}(\mathcal{O}))}={E_{\tilde{\mathcal{O}}_{1} \mathcal{O}}}$ for $\tilde{\O}_{1}\subset\tilde{\O}_{2}$, but this is a trivial computation.
\end{proof}

If $\omega_{0}$ is the vacuum state of $\{\mathcal{A}\}$, let $\omega:=\omega_{0}\circ E$, where $E$ is the consistent conditional expectation of the inclusion $\{\mathcal{A}\subset\mathcal{B}\}$ defined above, lifted to the quasilocal \Cstar-algebra $\mathcal{B}$. We call $\omega$ the vacuum state of $\{\mathcal{B}\}$ and the GNS representation induced by $\omega$ the vacuum representation. We denote by $\{\A\subset\B^Q\}$ and $\{\B^Q\}$ the extension constructed in this way, in its vacuum representation. 

\begin{remark}
It is not hard to check that the construction of the net $\{\mathcal{B}^Q\}$ and its conditional expectation $E^Q$ onto $\{\A\}$ does not depend on the choice of the family of unitary charge transporters $u_{\mathcal{O}}$, nor on the choice of $\Gamma_{\tilde{\mathcal{O}}}$.
\end{remark}

Up to now, we have seen that we can build (discrete, relatively local) extensions of nets $\{\A\subset\B^Q\}$ associated to generalized net Q-systems of intertwiners in $\DHR\{\mathcal{A}\}$. A natural question to ask is if this procedure insures that, if the Q-system comes from a given extension $\{\A\subset\B\}$, the induced extension will be unitarily equivalent to the starting one. The answer is affirmative when the generalized net Q-system is constructed as in Proposition \ref{prop:netqsys}.

\begin{lemma}\label{lemma:dimoffields}
Let $\{\mathcal{A}\subset \mathcal{B}\}$ be as in Proposition \ref{prop:netqsys}, assuming either $(i)$ or $(ii)$, and let $\theta=\oplus_{i} m_i\rho_{i}$ be a decomposition of $\theta$ into irreducibles in $\DHR\{\A\}$, where $[\rho_i]\neq[\rho_j]$, $m_i$ is the multiplicity of $[\rho_i]$ in $[\theta]$, and $\rho_i$, $\theta$ are localized in $\O\in\K$.
Then
$$H_{\rho_{i}}(\mathcal{O}):=\{\psi\in\mathcal{B}(\mathcal{O}), \psi a=\rho_{i}(a)\psi, a\in\mathcal{A}\}$$
is isomorphic as a Hilbert space to $\Hom_{\DHR\{\A\}}(\rho_i,\theta)$ via the map $\Phi:\psi\mapsto\gamma(\psi^{*})w$.
\end{lemma}

\begin{proof}
Note that $E(\psi_{1}\psi_{2}^{*})\in \Hom_{\DHR\{\A\}}(\rho_{i},\rho_{i})$, $\psi_{1},\psi_{2}\in H_{\rho_{i}}(\mathcal{O})$, is an inner product for $H_{\rho_{i}}(\mathcal{O})$ since $\rho_{i}$ is irreducible in $\DHR\{\A\}$. We have seen in Proposition \ref{prop:netqsys} that there is a collection $\{\psi_{j}\}\subset H_{\rho_{i}}(\mathcal{O})$, $j=1,\ldots, m_i$ (with $m_i$ possibly infinite), which is orthonormal with respect to the above inner product, and which is mapped via $\Phi:\psi\mapsto\gamma(\psi^{*})w$ onto an orthonormal basis of $\Hom_{\DHR\{\A\}}(\rho_{i},\theta)$.
Since $E(\psi_{1}\psi_{2}^{*})=\Phi(\psi_{1})^{*}\Phi(\psi_{2})$, the map $\Phi:\psi\mapsto\gamma(\psi^{*})w$ is an isomorphism of $H_{\rho_{i}}(\mathcal{O})$ onto $\Hom_{\DHR\{\A\}}(\rho_{i},\theta)$.
\end{proof}

\begin{proposition}
Let $\{\mathcal{A}\subset\mathcal{B}\}$ and $(\theta,w,\{m_{i}\})$ be as in Proposition \ref{prop:netqsys}, assuming either $(i)$ or $(ii)$. Then the inclusion $\{\mathcal{A}\subset \mathcal{B}^{Q}\}$ obtained from $(\theta,w,\{m_{i}\})$ is unitarily equivalent to $\{\mathcal{A}\subset \mathcal{B}\}$.
\end{proposition}

\begin{proof}
We first show that 
$$\mathcal{A}_{1}(\mathcal{O}) = \gamma_{\mathcal{O}}(\mathcal{B}(\mathcal{O})), \quad \mathcal{A}_{1}(\mathcal{O}) := \langle \theta_\O(\mathcal{A}(\mathcal{O})),u_{\mathcal{O}}\theta(u_{\mathcal{O}})m_{i}u_{\mathcal{O}}^{*}\rangle$$
where $\gamma_{\mathcal{O}} := \Ad_{u_\O} \gamma$ is a canonical endomorphism for the inclusion $\mathcal{A}(\mathcal{O})\subset\mathcal{B}(\mathcal{O})$. To see this,
let $\tilde{w_{i}}\in\mathcal{A}(\mathcal{O})$ be isometries such that $\tilde{w_{i}}\tilde{w_{i}}^{*}=u_{\mathcal{O}}w_{i}w_{i}^{*}u_{\mathcal{O}}^{*}$.
Then we have that $\tilde{w_{i}}^{*}u_{\mathcal{O}}w_{i}$ is a unitary and $\rho_{i, \mathcal{O}}:=\Ad_{\tilde{w_{i}}^{*}u_{\mathcal{O}}w_{i}}\rho_{i}$ is localized in $\mathcal{O}$. 
Using Lemma \ref{lemma:dimoffields}, we have that
$$H_{\rho_i}(\mathcal{O}_{0})=H_{\rho_i}(\tilde{\mathcal{O}})=\tilde{w_{i}}^{*}u_{\mathcal{O}}w_{i} H_{\rho_{i,\mathcal{O}}}(\tilde{\mathcal{O}})=\tilde{w_{i}}^{*}u_{\mathcal{O}}w_{i} H_{\rho_{i,\mathcal{O}}}(\mathcal{O})$$
for every $\tilde{\O}\in\mathcal{K}$ with $\O_0\cup \mathcal{O}\subset \tilde{\O}$.
Consequently
$$u_{\mathcal{O}}M_{i}=u_{\mathcal{O}}w_{i}\psi_{i}=u_{\mathcal{O}}w_{i}w_{i}^{*}w_{i}\psi_{i}=(u_{\mathcal{O}}w_{i}w_{i}^{*}u_{\mathcal{O}}^{*})(u_{\mathcal{O}}w_{i}\psi_{i})$$
$$=(\tilde{w_{i}}\tilde{w_{i}}^{*})(u_{\mathcal{O}}w_{i}\psi_{i})=\tilde{w_{i}}\tilde{\psi_{i}}\in\mathcal{B}(\mathcal{O})$$
from which the claim easily follows.

Thus it is clear that for a fixed $\tilde{\mathcal{O}}\in\mathcal{K}$, the map $\pi_{\tilde{\mathcal{O}}}:\mathcal{B}(\tilde{\mathcal{O}})\rightarrow \mathcal{B}^{Q}_{\tilde{\mathcal{O}}}(\tilde{\mathcal{O}})$, $\pi_{\tilde{\mathcal{O}}}:=\Ad_{\Gamma_{\tilde\O}^*}\gamma_{\tilde{\mathcal{O}}}$, is an isomorphism of von Neumann algebras, which maps $\mathcal{B}(\mathcal{O})$ onto $\mathcal{B}^Q_{\tilde{\mathcal{O}}}(\mathcal{O})$ and which lifts to a representation $\pi$ of the net $\{\mathcal{B}\}$. To show that this representation is unitarily equivalent to the vacuum representation, it is enough to show by the GNS theorem that
$$E^{Q}\circ \pi = E$$
with $E^{Q}$ and $E$ respectively the conditional expectations of $\{\mathcal{A}\subset\mathcal{B}^{Q}\}$ and $\{\mathcal{A}\subset\mathcal{B}\}$, but this is clear using $E=w^{*}\gamma(\cdot)w$.
\end{proof}

\begin{remark}
Note that in our second construction of the net $\{\mathcal{B}^Q\}$, the \emph{only} instance where the \emph{intertwining property} of the $m_{i}$ was used is to make sure that ${m_i}_\O = u_{\mathcal{O}}\theta(u_{\mathcal{O}}) m_i u_{\mathcal{O}}^{*} \in \mathcal{A}(\mathcal{O})$ (precisely by the intertwining property of the $m_{i}$ and Haag duality of $\{\mathcal{A}\}$, \cf Remark \ref{rmk:netQsysislocalQsys}). If a generalized Q-system does not have the intertwining property then the isotonous, relatively local net extension $\{\mathcal{B}^Q\}$ can still be defined in the same way, although the conditional expectation $E^Q$ cannot be defined on regions $\mathcal{O}\not\supset\mathcal{O}_{0}$ since in general ${m_i}_{_\mathcal{O}} = u_{\mathcal{O}}\theta(u_{\mathcal{O}}) m_i u_{\mathcal{O}}^{*}$ need not be in $\mathcal{A}(\mathcal{O})$, and thus $(\theta_{\mathcal{O}}, w_{\mathcal{O}},\{{m_{i}}_{\mathcal{O}}\})$ is not a priori a generalized Q-system in $\End(\mathcal{A}(\mathcal{O}))$.
\end{remark}

%%%
\section{Covariance of extensions}\label{sec:covarext}
%%%

In this section we show how spacetime covariance (\eg, M\"obius covariance in 1D or Poincar\'e covariance in 3+1D) extends from $\{\A\}$ to $\{\B\}$, where $\{\A\}$ is a local \emph{covariant} net over a directed set of spacetime regions $\K$, and $\{\A\subset\B\}$ is an extension with the properties implied by Theorem \ref{thm:discreteextnets}. This fact is common knowledge among experts, \cf \cite[\Rmk 4.3]{KaLo04} for irreducible extensions of $\{\Vir_c\}$, $c<1$, \cite[\Sec 3C]{BMT88} for time translation covariance in extensions of the chiral $U(1)$-current, and \cite[\Sec 6]{MTW16} for more recent examples of diffeomorphism covariant extensions of the $U(1)$-current in 1+1D. See also \cite[\Sec 6]{DoRo90}, \cite[\Thm 8.4]{DoRo89-1} for the covariance of canonical field extensions in 3+1D. In this section, see Theorem \ref{thm:discretecovariance}, we give a general proof of covariance for extensions of nets, with finite or infinite index (of discrete type as in Theorem \ref{thm:discreteextnets}). The proof essentially relies on \emph{tensoriality} and \emph{naturality} properties of the \emph{action} of the spacetime symmetry group (implemented by covariance cocycles) on the DHR category. Hence we formulate it in a \Cstar tensor categorical language, \cf \cite[\App 5]{Tur10mueger} due to M. M\"uger. But first we need a few definitions. 

Let $\cP$ be a (pathwise) connected and simply connected group of spacetime symmetries (\eg, $\cP = \widetilde{\Mob}$ the universal covering of the M\"obius group acting on $\RR$ (actually on $\RRbar = \RR \cup \{\infty\}$), or $\cP = \tilde{\cP}^\uparrow_+$ the universal covering of the proper orthochronous Poincar\'e group acting on $\RR^{3+1}$). 
Assume that $\cP$ contains a distinguished $(n+1)$-parameter subgroup, $n \geq 0$, of \lqq spacetime translations" (\eg, the rotations inside $\widetilde{\Mob}$, or the four-dimensional spacetime translations inside $\tilde{\cP}^\uparrow_+$). The following definition describes Poincar\'e covariant theories on Minkowski space and M\"obius covariant theories on the real line at the same time, \cf \cite[\Sec 8]{GuLo92}, \cite[\Sec 1]{BGL93}, \cite[\Sec 3]{CKL08}.

\begin{definition}\label{def:covarnet}
An isotonous net $\{\A\}$ of von Neumann algebras realized on $\Hilb_0$ over a directed set of spacetime regions $\K$ is called \textbf{covariant} with respect to $\cP$ if there is a strongly continuous unitary representation $g\mapsto U(g) $ of $\cP$ on $\Hilb_0$ such that
$$U(g)\A(\O)U(g)^* = \A(g\O),\quad \O\in\K,\quad g\in\cU_\O$$
where $\cU_\O\subset\cP$ denotes the (pathwise) connected component of the identity $e$ in $\cP$ of the set $\{g\in\cP : g\O\in\K\}$.
We always assume that $\cU_\O$ is a non-trivial neighborhood of $e$ for every $\O\in\K$ (\ie, $\cK$ is \lqq \emph{locally stable}" under the action of $\cP$), and that if $\O_1, \O_2\in\K$ and $g\in\cU_{\O_1}\cap\cU_{\O_2}$ then there is $\O\in\K$ such that $\O_1\cup\O_2\subset\O$ and $g\in\cU_\O$ (\ie, $\K$ is \lqq \emph{$\cP$-stably directed}").\footnote{These assumptions are the abstraction of the geometric properties which are needed in this section. They are fulfilled, \eg, by all the examples of spacetime symmetries $\cP$ acting on directed sets of bounded regions $\cK$ mentioned above.}

Concerning spectral properties, we assume that the generators of the spacetime translation subgroup (energy-momentum operators) have positive joint spectrum, and that there is a $\cP$-invariant unit vector $\Omega_0\in\Hilb_0$ (vacuum vector) which is cyclic for $\left\langle\A(\O), U(g) : \O\in\K, g\in\cP\right\rangle$.
\end{definition}

\begin{remark}\label{rmk:globalaction}
Assume first that $\cP$ preserves $\K$, \ie, $g\O \in\K$ for every $g\in\cP$, $\O\in\K$ (\eg, if $\K$ is the set of all double cones in Minkowski space and $\cP$ is the universal covering of the Poincar\'e group, or if $\K$ is the set of all open proper bounded intervals of $\RR$ and $\cP$ is the translation-dilation subgroup of the M\"obius group), or equivalently $\cU_\O = \cP$ for every $\O\in\K$. Consider then a local net $\{\A\}$ over $\K$ as in Definition \ref{def:localnet}, fulfilling Haag duality and covariant with respect to $\cP$ as in Definition \ref{def:covarnet}. Denoted by $\alpha_g:=\Ad_{U(g)}$ the adjoint action on $\B(\Hilb_0)$, we have an action of $\cP$ on the net $\{\A\}$ (which extends to an action by normal *-automorphisms of the quasilocal algebra $\A$), and another action of $\cP$ on DHR endomorphisms $\rho$ in $\DHR\{\A\}$ given by $^g \rho := \alpha_g\rho\alpha_g^{-1}$. Observe that $^g \rho$ is again DHR and localizable in $g\O$ if $\rho$ is localizable in $\O$. Moreover, $^g t := \alpha_g(t)\in\Hom_{\DHR\{\A\}}({^g\rho},{^g\sigma})$, $g\in\cP$, if $t\in\Hom_{\DHR\{\A\}}(\rho,\sigma)$. In other words, we have an \textbf{action} of $\cP$ on the category $\DHR\{\A\}$ (as a strict \Cstar braided tensor category) by autoequivalences (actually automorphisms), which is also strict in the terminology of \cite[\App 5]{Tur10mueger}. Indeed, one can easily check that ${^g(\rho\times\sigma)} = {^g\rho} \times {^g\sigma}$, where $\rho\times\sigma = \rho\sigma$ (composition of endomorphisms of $\A$), and ${^g\id = \id}$ for every $g\in\cP$ and $\rho,\sigma$ in $\DHR\{\A\}$. Also, ${^{g}({^h\rho}) = {^{gh}\rho}}$ and ${^e\rho} = \rho$ if $e$ is the identity in $\cP$.
\end{remark}

On the other hand, if not every $g\in\cP$, $\O\in\K$ fulfill $g\O \in\K$ (\eg, if $\K$ is the set of all open proper bounded intervals in $\RR$ and $\cP$ is the universal covering of the M\"obius group) then $\alpha_g$, $g\in\cP$, are \emph{not} always automorphisms of the quasilocal algebra $\A$ and the previous global statements have to be replaced with \emph{local} ones by specifying local algebras and spacetime regions. For instance, ${^g \rho} = \alpha_g\rho\alpha_{g^{-1}}$, for a fixed $g\in\cP$, is well defined on every $\A(\O)$, $\O\in\K$, such that $g^{-1}\O\in\K$, and it is an endomorphisms of $\A(\O)$ if $\rho$ is and endomorphisms of $\A(g^{-1}\O)$ (\eg, if $\rho$ is DHR localizable in $g^{-1}\O$). Similarly, the intertwining relation for ${^g t}$ between ${^g\rho}$ and ${^g\sigma}$, if $t\in\Hom_{\DHR\{\A\}}(\rho,\sigma)$, must be intended locally. In this level of generality we give the following definition, \cf \cite[\Sec 2, \App A]{Lon97}, \cite[\App 5]{Tur10mueger}.

\begin{definition}\label{def:equivaraction}
Let $\{\A\}$ be a local net realized on $\Hilb_0$ as in Definition \ref{def:localnet}, fulfilling Haag duality and covariant with respect to a group of spacetime symmetries $\cP$ as in Definition \ref{def:covarnet}. Let $\alpha_g :=\Ad_{U(g)}$, $g\in\cP$, and let 
$$\C\subset\DHR\{\A\}$$ 
be a full and replete tensor subcategory, closed under finite direct sums and subobjects. 
We say that $\cP$ has an \textbf{equivariant action} on $\C$ (and write ${\C^\cP} = \C$) if there is a map
$$g,\rho \mapsto z(g,\rho)$$
where $g\in\cP$, $\rho$ is an object of $\C$, such that
\begin{itemize}
\item [$(i)$] $z(\cdot,\rho)$ is a strongly continuous unitary valued map in $\B(\Hilb_0)$, $z(g,\id) = \oneop$ for every $g\in\cP$, $z(e,\rho) = \oneop$ for every $\rho$ in $\C$, and
$$z(gh,\rho) = \alpha_g(z(h,\rho)) z(g,\rho)$$ for every $g,h\in\cP$ and $\rho$ in $\C$. (\lqq cocycle identity")
\item [$(ii)$] $$z(g,\rho) \rho(a) z(g,\rho)^* = \alpha_g \rho\, \alpha_{g^{-1}}(a),\quad a\in\A(\tilde\O)$$ if $\rho$ is DHR localizable in $\O_0\in\K$, $g\in\cU_{\O_0}$ is such that $g^{-1}\in\cU_{\O_0}$, and $\tilde\O\in\K$ is such that $\O_0 \cup g\O_0 \subset \tilde\O$, $g^{-1}\in\cU_{\tilde\O}$.
\footnote{The existence of at least one $\tilde\O$ with these properties is guaranteed because $\cK$ is $\cP$-stably directed by assumption (Definition \ref{def:covarnet}).} (\lqq local intertwining property")
\item [$(iii)$] $$z(g,\rho)\in\A(\O)$$ if $\rho$ is DHR localizable in $\O_0\in\K$, $g\in\cU_{\O_0}$, and $\O\in\K$ is such that $\O_0 \cup g\O_0 \subset \O$.
\item [$(iv)$] $$\alpha_g(t) = z(g,\sigma) t z(g,\rho)^*$$ if $\rho$ and $\sigma$ are DHR localizable respectively in $\O_1$ and $\O_2\in\K$, $g\in\cU_{\O_1}\cap\cU_{\O_2}$, and $t\in\Hom_{\DHR\{\A\}}(\rho,\sigma)$. (\lqq naturality of cocycles")
\item [$(v)$] $$z(g,\rho\sigma) = z(g,\rho) \rho(z(g,\sigma))$$ if $\rho, \sigma, g$ are as in $(iv)$. (\lqq tensoriality of cocycles")
\item [$(vi)$] $\Ad_{z(g,\rho)} \rho$ is DHR localizable in $g\O_0\in\K$, if $\rho, g$ are as in $(iii)$. (\lqq global localization property")
\end{itemize}
\end{definition}

\begin{remark}
In the case that $g\O\in\K$ for every $g\in\cP$, $\O\in\K$, we have a (global) action of $\cP$ on $\C \subset \DHR\{\A\}$ (as a strict \Cstar braided tensor category), see \cite[\Def 1.2]{Tur10mueger}. Then the equivariance of the action as in Definition \ref{def:equivaraction}, \cf \cite[\Def 2.1]{Tur10mueger}, says that the map $z$ defines a \emph{tensor natural transformation} (isomorphism) between the trivial action $\iota$ of $\cP$ on $\C$ by autoequivalences  and the action defined by $\alpha$. Naturality is automatic because $\cP$ is considered as a discrete tensor category, \ie, the only morphisms are the identity morphisms, while tensoriality is encoded in the cocycle identity $(i)$. The properties $(iv)$ and $(v)$ above say that $z(g,\cdot)$ is a natural tensor transformation (unitary isomorphism) between tensor functors $\iota_g$ and $\alpha_g$ for every $g\in\cP$.
\end{remark}

\begin{lemma}\label{lem:covarDHRendo}
In the assumptions of Definition \ref{def:equivaraction}, let $z(\cdot,\cdot)$ be a map fulfilling the properties $(i)$ and $(ii)$, then the following holds as well
\begin{itemize}
\item [$(ii)'$] $$z(g,\rho) \rho(\alpha_g(a)) z(g,\rho)^* = \alpha_g (\rho(a)),\quad a\in\A(\O),\quad g\in\cU_{\O}$$ if $\rho$ is DHR localizable in $\O\in\K$.
The same is true if $a\in\A(\tilde\O)$ and $g\in\cU_{\tilde\O}$ for any $\tilde\O\in\K$.
\end{itemize}
In other words, the unitaries $U_\rho(g) := z(g,\rho)^* U(g)$, $g\in\cP$, implement the covariance of $\rho$ with respect to $\cP$ (\cf \emph{\cite[\Sec 4.2]{CKL08}}) and give a strongly continuous unitary representation of $\cP$ on $\Hilb_0$. 
\end{lemma}

\begin{proof}
Let $a\in\A(\O)$, $g\in\cU_\O$ and assume that $\O\in\cK$ is a localization region of $\rho$. Also, let $\cV\subset\cU_\O$ be a symmetric neighborhood of $e$, \eg, $\cV := \cU_\O\cap\cU_\O^{-1}$. Consider the set of all elements $\cV(e) \subset \cU_\O$ that can be joined to $e$ by a $\cV$-chain in $\cU_\O$, namely those $g\in\cU_\O$ such that there are $x_1,\ldots,x_n\in\cU_\O$, $n\geq 1$, with $x_1=e$, $x_n=g$, and $x_{j+1}x_j^{-1}\in\cV$ for every $j=1,\ldots,n-1$, \cf \cite[\Def 19, 144]{BePl01}. By a standard argument, $\cV(e)$ is open and closed in $\cU_\O$, hence $\cV(e) = \cU_\O$ by connectedness. Then every $g\in\cU_\O$ can be written as $g=g_1g_2\cdots g_n$ where $g_j\in\cV$, and in addition $g_j\cdots g_n\in\cU_\O$ for every $j=1,\ldots, n$, $n\geq 1$. Just set $g_j := x_{j+1}x_j^{-1}$, $j=1,\ldots,n-1$, and $g_n := e$. 

Now, by the cocycle identity $(i)$ we have
\begin{equation}\label{eq:longcocycleid}z(g_1g_2\cdots g_n,\theta) = \alpha_{g_1\cdots g_{n-1}}(z(g_n,\theta))\cdots \alpha_{g_1}(z(g_2,\theta)) z(g_1,\theta)\end{equation}
and we want to compute its adjoint action on $\rho(\alpha_g(a))$, $a\in\A(\O)$. 

Thus, $z(g_1,\theta) \rho(\alpha_g(a)) z(g_1,\theta)^* = \alpha_{g_1}(\rho(\alpha_{g_2\cdots g_n}(a)))$ because $g_1, {g_1}^{-1}\in\cV\subset\cU_\O$, hence equality $(ii)$ holds on every $\A(\tilde\O)$, $\tilde\O\in\K$, such that $\O\cup g_1\O\subset\tilde\O$ and $g_1^{-1}\in\cU_{\tilde\O}$. Moreover, $g_1^{-1}\in\cU_{g\O}=\cU_\O g^{-1}$ because $g_1^{-1}g = g_2\cdots g_n\in\cU_\O$, hence we can assume that $g\O\subset\tilde\O$, by enlarging $\tilde\O$ if necessary, because $\cK$ is $\cP$-stably directed by assumption (Definition \ref{def:covarnet}). Continuing, $z(g_2,\theta) \rho(\alpha_{g_2\cdots g_n}(a)) z(g_2,\theta)^* = \alpha_{g_2}(\rho(\alpha_{g_3\cdots g_n}(a)))$ because $g_2, {g_2}^{-1}\in\cV\subset\cU_\O$, hence we can choose $\tilde\O\in\K$ as in $(ii)$ such that $g_2^{-1}\in\cU_{\tilde\O}$ and again further assume that $g_2\cdots g_n\O\subset\tilde\O$. Indeed, $g_2^{-1}\in\cU_{g_2\cdots g_n\O}=\cU_\O g_n^{-1}\cdots g_2^{-1}$ because $g_3\cdots g_n\in\cU_\O$. By finite iteration we get the first claim.

By the cocycle identity $(i)$, the unitaries $U_\rho(g) := z(g,\rho)^* U(g)$, $g\in\cP$, form a representation of $\cP$. Indeed, $U_\rho(g) U_\rho(h) = z(g,\rho)^* \alpha_g(z(h,\rho))^* U(gh) = U_\rho(gh)$ and $z(g^{-1},\rho) = \alpha_{g^{-1}}(z(g,\rho)^*)$, hence also $U_\rho(g^{-1}) = U_\rho(g)^*$, follow from $z(e,\rho) = \oneop$. We want to show that it implements the covariance of $\rho$. 

Let $a\in\A(\tilde\O)$, $g\in\cU_{\tilde\O}$ for an arbitrary $\tilde\O\in\cK$. Define $\tilde\cV := \cU_{\tilde\O}\cap\cU_{\tilde\O}^{-1}$ and consider $\cW := \cV \cap \tilde\cV$, or any other symmetric neighborhood $\cW$ of $e$ such that $\cW\subset\cU_{\O} \cap\cU_{\tilde\O}$. By the same argument as above, we have $g\in\cU(\tilde\O)=\cW(e)$, \ie, we can write $g=g_1\cdots g_n$, where $g_j\in\cW$ and $g_j\cdots g_n\in\cU_{\tilde\O}$ for every $j=1,\ldots,n$, $n\geq 1$, and
$$U_\rho(g) \rho(a) U_\rho(g)^* = U_\rho(g_1)\cdots U_\rho(g_n) \rho(a) U_\rho(g_n)^*\cdots U_\rho(g_1)^*.$$
Now, $g_n\in\cW\subset\cU_{\O} \cap\cU_{\tilde\O}$ hence there is $\O_1\in\K$ such that $\O\cup\tilde\O\subset \O_1$, $g_n\in\cU_{\O_1}$, moreover $\rho$ is localized in $\O_1$, $a\in\A(\O_1)$, thus by the first claim we get $U_\rho(g_n) \rho(a) U_\rho(g_n)^* = \rho(\alpha_{g_n}(a))$. Continuing, $g_{n-1}\in\cU_\O\cap\cU_{g_n\tilde\O}$ and we can repeat the previous argument on $\O_2\in\K$ such that $\O\cup g_n\tilde\O\subset\O_2$, $g_{n-1}\in\cU_{\O_2}$, to get $U_\rho(g_{n-1})\rho(\alpha_{g_n}(a))U_\rho(g_{n-1}^*) = \rho(\alpha_{g_{n-1}g_n}(a))$. By finite iteration we get
$$U_\rho(g)\rho(a)U_\rho(g)^* = \rho(\alpha_g(a))$$
where $a\in\A(\tilde\O)$, $g\in\cU_{\tilde\O}$ for an arbitrary $\tilde\O\in\cK$, completing the proof.
\end{proof}

With similar arguments one can extend naturality and tensoriality of cocycles to (almost all) $g\in\cP$, namely

\begin{lemma}\label{lem:tensornat}
In the assumptions of Definition \ref{def:equivaraction}, let $z(\cdot,\cdot)$ be a map fulfilling the properties $(i)$ and $(iv)$, then the following holds as well
\begin{itemize}
\item [$(iv)'$] $$\alpha_g(t) = z(g,\sigma) t z(g,\rho)^*,\quad g\in\cP.$$
\end{itemize}
If $z(\cdot,\cdot)$ fulfills $(i)$, $(ii)$, $(iii)$ and $(v)$, then it fulfills also
\begin{itemize}
\item [$(v)'$] $$z(g,\rho\sigma) = z(g,\rho) \rho(z(g,\sigma)),\quad g\in\cU_{\O_2},$$
where $\O_2\in\K$ is a DHR localization region of $\sigma$.
\end{itemize}
\end{lemma}

\begin{proof}
Let $\O_1,\O_2\in\K$ be respectively localization regions of $\rho, \sigma$. To prove the first statement, write $g\in\cP$ as $g=g_1\cdots g_n$, $n\geq 1$, where $g_j\in\cU_{\O_1}\cap\cU_{\O_2}$, $j=1,\ldots, n$. Then make use of equation (\ref{eq:longcocycleid}) and apply $(iv)$ at each step.

To prove the second statement, write $g\in\cU_{\O_2}$ as before, and assume in addition that $g_j\cdots g_n\in\cU_{\O_2}$, $j=1,\ldots, n$, \cf proof of Lemma \ref{lem:covarDHRendo}. Then make again use of equation (\ref{eq:longcocycleid}) for $z(g,\rho\sigma)$ and apply $(v)$ for each $g_j\in\cU_{\O_1}\cap\cU_{\O_2}$. Repeated use of Lemma \ref{lem:covarDHRendo} gives the desired conclusion. Notice that $z(g,\sigma)$ belongs to the quasilocal algebra $\A$ because of assumption $(iii)$, hence one can safely apply the endomorphisms $\rho$. 
\end{proof}

Now we show that the properties (mainly tensoriality and naturality) of covariance cocycles expressed by the \emph{equivariance} of the action of spacetime symmetries on the DHR category ensure \emph{covariance} of the extended nets constructed as in Theorem \ref{thm:discreteextnets}.

\begin{theorem}\label{thm:discretecovariance}
Let $\{\A\}$ be a local net fulfilling Haag duality, standardly realized on $\Hilb_0$, and covariant with respect to a group of spacetime symmetries $\cP$ (Definition \ref{def:covarnet}). Assume in addition that either $\cP$ acts transitively on $\K$ (\ie, for every $\O_1, \O_2\in\K$ there is $g\in\cU_{\O_1}$ such that $g\O_1=\O_2$), 
or $\cP$ preserves $\cK$ (\ie, $g\O\in\K$ for every $g\in\cP$, $\O\in\K$), \footnote{This assumption is needed to obtain covariance of $\{\B\}$ on all the regions in $\K$, \cf footnote after equation (\ref{eq:extlocalalgebras}). Examples are $\widetilde \Mob$ acting transitively on bounded intervals in $\RR$, or $\tilde\cP^{\uparrow}_+$ preserving double cones in $\RR^{3+1}$.}.\\
Then an extension $\{\B\}$ of $\{\A\}$ constructed as in Theorem \ref{thm:discreteextnets} from a unital generalized net Q-system of intertwiners $(\theta, w, \{m_i\})$ is automatically covariant, provided that $\cP$ has an equivariant action on a tensor subcategory $\C\subset \DHR\{\A\}$ (\ie, $\C^{\cP} = \C$) which contains $\theta$ (Definition \ref{def:equivaraction}).
\end{theorem}

\begin{proof}
Let $(\theta, w, \{m_i\})$, $i\in I$ be a generalized net Q-system of intertwiners in $\DHR\{\A\}$ and construct the extension $\{\A\subset\B\}$ as in Theorem \ref{thm:discreteextnets}. Here $\O\in\K$ is a fixed localization region for $\theta$ and $\N = \A(\O)$. In the following we denote by $\Hilb := \Hilb_\varphi$ the Hilbert space of $\{\B\}$, we identify $\Hilb_0 = e_\N\Hilb$ and $a\in\A$, $M_i\in\B(\O)$ with their images under $\pi_\varphi$ in $\B(\Hilb)$. Thus 
$$\Hilb = \overline{\sum_i M_i^* \Hilb_0}$$
where every $\psi\in\Hilb$ can be written as $\psi = \sum_i M_i^* \psi_i$, with $\psi_i \in q_i \Hilb_0$.
Moreover, we have $M_i a = \theta(a) M_i$ for every $a\in\A$, $i\in I$. Having full control of the Hilbert space thanks to the Pimsner-Popa condition, we can set
$$\hat U(g)\psi := \sum_i M_i^* z(g,\theta)^* U(g) \psi_i $$
for every $g\in\cP$, $\psi\in\Hilb$, where $U$ implements the covariance of $\{\A\}$ on $\Hilb_0$ and $z(\cdot,\theta)$ is the covariance cocycle of $\theta$ given by equivariance. By definition of $\alpha_g$ and by the cocycle identity we have $z(g,\theta)^* U(g)z(h,\theta)^* U(h) = z(g,\theta)^* \alpha_g(z(h,\theta))^* U(g)U(h) = z(gh,\theta)^*U(gh)$, hence $\hat U$ is a representation of $\cP$ on $\Hilb$, which is strongly continuous and unitary as one can easily check.

In order to show that $\hat U$ implements covariance of $\{\B\}$ with respect to $\cP$, take first $a\in\A(\O)$, where $\O$ is as above, take $g\in\cU_\O$, see Definition \ref{def:covarnet}, $\psi\in\Hilb$, and compute
$$\hat U(g) a \hat U(g)^*\psi = \sum_i M_i^* z(g,\theta)^* U(g) \theta(a) U(g)^* z(g,\theta)\psi_i$$
$$= \sum_i M_i^* z(g,\theta)^* \alpha_g(\theta(a)) z(g,\theta)\psi_i = \sum_i M_i^* \theta(\alpha_g(a)) \psi_i = \alpha_g(a) \psi,$$
where the third equality follows from Lemma \ref{lem:covarDHRendo}. Take now $M_i\in\B(\O)$, $i\in I$, then
$$\hat U(g) M_i \hat U(g)^* \psi = \sum_{j,k} M_k^* z(g,\theta)^* U(g) l_j^{ki} U(g)^* z(g,\theta) \psi_j$$
where the coefficients $l_j^{ki}\in\N = \A(\O)$ are those given in equation (\ref{eq:pipoalgebra}). By naturality of cocycles, see the property $(iv)$ in Definition \ref{def:equivaraction}, and because $l_j^{ki} \in \Hom_{\DHR\{\A\}}(\theta,\theta^2)$, we have $\alpha_g(l_j^{ki}) \equiv U(g) l_j^{ki} U(g)^*= z(g,\theta^2) l_j^{ki} z(g,\theta)^*$. Moreover, by tensoriality of cocycles, see the property $(v)$ in Definition \ref{def:equivaraction}, we have that $z(g,\theta^2) = z(g,\theta) \theta(z(g,\theta))$, hence
$$= \sum_{j,k} M_k^* \theta(z(g,\theta)) l_j^{ki} \psi_j = z(g,\theta) M_i \psi.$$
The global localization property $(vi)$ in Definition \ref{def:equivaraction} implies that $z(g,\theta)$ is a unitary charge transporter for $\theta$ from $\O$ to $g\O\in\K$, hence $z(g,\theta) M_i \in\B(g\O)$ by definition ($\ref{eq:extlocalalgebras}$) of the local algebras, and we conclude
$$\hat U(g) \B(\O) \hat U(g)^* = \B(g\O).$$
Now, covariance for arbitrary regions $\tilde\O\in\K$ and $g\in\cU_{\tilde\O}$ follows either by transitivity of $\cP$ on $\cK$ (trivially), or because $\cP$ preserves $\cK$, in which case $\cU_{\tilde\O} = \cP$ and we can meaningfully write $\Ad_{U_\theta(g)} \theta(u) = \theta(\alpha_g(u))$, $u\in\Hom_{\DHR\{\A\}}(\theta,\tilde\theta)$ and $\theta(z(g,\theta))$ in $\A$, \cf Lemma \ref{lem:covarDHRendo}, \ref{lem:tensornat}.

Positivity of the energy-momentum spectrum holds because $U_\theta$ has positive spectrum, indeed $\theta$ is a (possibly infinite) direct sum of covariant endomorphisms fulfilling the spectrum condition, see \cite[\Thm 5.2]{DHR74}. 
The $\cP$-invariance of the vacuum vector $\Omega := \Omega_\varphi$ follows from $\Omega = M_0^*w\Omega_0$, indeed $\hat U(g) \Omega = M_0^* z(g,\theta)^* U(g) w \Omega_0 = M_0^* z(g,\theta)^* \alpha_g(w) \Omega_0 = \Omega$ by naturality $(iv)$ of the action of $g$ on $w\in\Hom_{\DHR\{\A\}}(\id,\theta)$ and because $z(g,\id) = \oneop$. Thus the extended net $\{\B\}$ is covariant as in Definition \ref{def:covarnet}.
\end{proof}

\begin{remark}
If the quasilocal algebra $\A$ together with the elements of the Pimsner-Popa basis $M_i$, $i\in I$, form a *-algebra of charged intertwiners in the sense of Remark \ref{rmk:*algchargedintert}, one can try to define covariance of the extension $\B$ (at the *-algebra level) by postulating $\hat\alpha_g := \alpha_g$ on $\A$ and $\hat\alpha_g(M_i) := z(g,\theta) M_i$. In this case as well, naturality and tensoriality of the cocycle $z$ guarantee that $\hat\alpha_g$ is *-multiplicative.
\end{remark}

Next, we show how \emph{equivariance} holds, in the sense of Definition \ref{def:equivaraction}, for the action of some typical spacetime symmetry groups on the DHR category in different dimensions. More precisely, we consider here the subcategory $\C=\DHR_{d}\{\A\}$ of $\DHR\{\A\}$ (Definition \ref{def:DHRf-d}) which is relevant for finite index or infinite index discrete extensions treated in Theorem \ref{thm:discreteextnets}.

\begin{example}\label{ex:moebcovar} (\textbf{M\"obius covariant nets in 1D}). 
Let $\cP = \widetilde\Mob$ the universal covering of the M\"obius group and $\cK =$ \{open proper bounded intervals $I\subset\RR$\}. Consider a local $\cP$-covariant net $\{\A\}$ over $\cK$ as in Definition \ref{def:localnet}, \ref{def:covarnet}, fulfilling Haag duality on $\RR$, namely $\A(I')' = \A(I)$, $I\in\cK$, $I' = \RR\smallsetminus \bar I$. By locality and $\cP$-covariance we have $U(\Rot_{2\pi}) = \oneop$ \cite[\Thm 1.1]{GuLo96}, hence $\{\A\}$ is automatically $\Mob$-covariant and we can extend it to a net $\{\tilde\A\}$ over the open proper intervals of $\Sone$, see \cite[\Prop 16, \Cor 17]{CKL08}. 
The extension coincides with the one given by $\tilde\A(I) := \A(\Sone \smallsetminus \bar I)'$ if $I\subset \Sone$ contains the point at infinity in its closure and $\tilde\A(I) := \A(I)$ otherwise, see \cite[\Lem 49]{KLM01}. 
Moreover, every endomorphism $\rho$ in $\DHR\{\A\}$ extends to a representation $\{\pi_I, I\subset\Sone\}$ of $\{\tilde\A\}$ on $\Hilb_0$ such that 
$$\pi_I = \rho_{\restriction\A(I)}$$ 
if $I$ is identified to a bounded interval of $\RR$ via the Cayley map, see \cite[\Prop 50]{KLM01}. The Bisognano-Wichmann property \cite[\Prop 1.1]{GuLo96} and strong additivity \cite[\Lem 1.3]{GLW98} ensure that finite-dimensional DHR endomorphisms are covariant (with positive energy) with respect to $\cP$, see \cite[\Thm 5.2]{GuLo92}.
For every $\rho$ in $\DHR_{f}\{\A\}$, following \cite[\Prop 8.2]{GuLo92}, we can define by equation (\ref{eq:longcocycleid}) the cocycle $z(\cdot,\rho)$. The definition is well posed (in $\B(\Hilb_0)$) by \cite[\Eq (8.5)]{GuLo92} because any two chains in $\cP$ are homotopic by simple connectedness of $\cP$, see \cite[\Def 45, \Lem 46]{BePl01} for more details. Thus the properties $(i)$, $(ii)$, $(iii)$ of Definition \ref{def:equivaraction} hold, see also $(ii)'$ of Lemma \ref{lem:covarDHRendo}. The property $(vi)$ holds by additivity \cite[\Sec 3]{FrJoe96} while $(iv)$ and $(v)$ can be derived from the results of \cite{Lon97}. Indeed, let $\rho$ and $\sigma$ in $\DHR\{\A\}$ and choose a common localization interval $I\in\cK$. Let $I_1\in\cK$ be such that $\bar I \subset I_1$ and $I_2, I_3\in\cK$ such that $\{I_i, i=1,2,3\}$ is a partition of $\Sone$ obtained by removing three distinct points and counterclockwise ordered. Let $\cV\subset\cU_I\cap\cU_I^{-1}$ be an arbitrarily small symmetric neighborhood of $e$ in $\cP$, whose elements $g$ can be written as products of dilations $\Lambda_{I_i}$ associated to $I_i$, $i=1,2,3$ such that in addition $\Lambda^{j}_{I_{i_j}}$ and $\Lambda^{j}_{I_{i_j}}\cdots\Lambda^{n}_{I_{i_n}}$ map $I$ inside $I_1$ for every $j=1,\ldots,n$ and $g = \Lambda^{1}_{I_{i_1}}\cdots\Lambda^{n}_{I_{i_n}}$. Thus at each step we can consider $I$ as a subinterval of either $I_1$, or $\Sone \smallsetminus \bar I_2$, or $\Sone \smallsetminus \bar I_3$. Observe that the dilations with respect to any such partition of $\Sone$ into three intervals generate $\cP$, see \cite[\Lem 1.1]{GLW98}. Now, with this choice of $\cV$, \cf \cite[\Lem 2.2]{Lon97}, for every $g\in\cV$ we have 
$$\alpha_g(t) = z(g,\sigma) t z(g,\rho)^*$$
if $t\in\Hom_{\DHR\{\A\}}(\rho,\sigma)$, and  
$$z(g,\rho\sigma) = z(g,\rho) \rho(z(g,\sigma)).$$
Building suitable $\cV$-chains in $\cP$ and reasoning as in the proof of Lemma \ref{lem:tensornat}, one can show the properties $(iv)$ and $(v)$\footnote{Tensoriality also follows by observing that $\cP=\widetilde\Mob$ is perfect, see, \eg, \cite[\App A]{Lon97}, hence the unitary representation $U_{\rho\sigma}$ which implements covariance of $\rho\sigma$ in $\DHR_f\{\A\}$ is unique.} in their global formulation $(iv)'$ and $(v)'$ of Lemma \ref{lem:tensornat}.

Now, if $\rho$ is in $\DHR_d\{\A\}$ let $\{w_i\}$ be a (possibly infinite) Cuntz family of isometries in $\A(I)$, for $I\in\cK$ big enough, such that $w_i w_i^*\in\rho(\A)'\cap\A(I)$ are mutually orthogonal projections, $\sum_i w_i w_i^* = \oneop$ and $\Ad_{w_i^*} \rho =: \rho_i$ are irreducible DHR endomorphisms of finite-dimension. For every $g\in\cP$  
$$z(g,\rho) := \sum_i \alpha_g(w_i) z(g,\rho_i)w_i^*$$
converges in the strong operator topology and extends the definition given in $\DHR_f\{\A\}$ by \cite[\Prop 1.3, \Eq (1.13)]{Lon97}. Let $\C:=\DHR_d\{\A\}$, the unitaries $z(g,\rho)$, $g\in\cP$, $\rho$ in $\C$ form again a cocycle map, as one can check directly on each direct summand of $\rho = \oplus_i \rho_i$, hence the action of $\cP$ on $\C$ is equivariant in the sense of Definition \ref{def:equivaraction}.
\end{example}

\begin{example} (\textbf{Poincar\'e covariant nets in 3+1D}). 
Let $\cP = \tilde\cP^{\uparrow}_+$ the universal covering of the Poincar\'e group and $\cK =$ \{double cones $\O\subset \RR^{3+1}$\}. Consider a local $\cP$-covariant net $\{\A\}$ over $\cK$ as in Definition \ref{def:localnet}, \ref{def:covarnet}, fulfilling Haag duality on $\RR^{3+1}$. Assume that $\{\A\}$ fulfills in addition the Bisognano-Wichmann property on wedges, see \cite{BiWi75}, and that local intertwiners between finite-dimensional DHR endomorphisms are global intertwiners, \cf \cite[\Thm 4.3]{Rob74}, \cite[\Cor 6.2]{DMV04}. Due to the fact that Lorentz boosts with respect to different wedges generate $\cP = \tilde\cP^{\uparrow}_+$, we can make again use of the results of \cite{GuLo92}, \cite{Lon97}, in a different geometrical situation, to draw analogous conclusions. Namely, the action of $\cP$ of $\C:=\DHR_d\{\A\}$, which in this case is globally defined, see Remark \ref{rmk:globalaction}, is again equivariant in the sense of Definition \ref{def:equivaraction}.
\end{example}

%%%
\section{Braided product of nets}\label{sec:brproductnets}
%%%

In this section we apply the braided product construction to nets of von Neumann algebras and show that it enjoys some remarkable properties, in analogy to the finite index case, which allows one to extract boundary quantum field theories as in \cite{BKLR16}. Such field theories with transparent boundaries will be discussed in the next section.

Denote by $\{\A\stackrel{E^{L}}{\subset}\B^{L}\}$, $\{\A\stackrel{E^{R}}{\subset}\B^{R}\}$ two discrete relatively local extensions of the same local net $\{\A\}$ (Definition \ref{def:inclnets}) constructed from unital generalized net Q-systems of intertwiners $(\theta^{L},w^{L},\{m^{L}_{i}\})$, $(\theta^{R},w^{R},\{m^{R}_{j}\})$ in $\DHR\{\A\}$ as in Theorem \ref{thm:discreteextnets}. By Proposition \ref{prop:brprodisQsys} and again Theorem \ref{thm:discreteextnets} we know that there is a \emph{braided product extension} $\{\A\subset\B^L\times^\pm_\eps\B^R\}$ such that
$$\begin{array}{ccccc}
&&\B^L&&\\
&\rotatebox[origin=c]{35}{$\subset$}&&\rotatebox[origin=c]{-35}{$\subset$}&\\ 
\A&&&&\hspace{-2mm}\B^L\times^\pm_\eps\B^R\\
&\rotatebox[origin=c]{-35}{$\subset$}&&\rotatebox[origin=c]{35}{$\subset$}&\\
&&\B^R&&
\end{array}$$
where $\eps$ is the DHR braiding. Of course we have the analogous of Proposition \ref{prop:brprodembeddings}, which we rewrite below to establish notation.

We prefer to think of the net $\{\B^L\times^\pm_\eps\B^R\}$ as the one constructed in Section \ref{sec:altext}. Let $\{(\mathcal{B}^{L}\times^{\pm}_{\eps}\mathcal{B}^{R})_{\tilde{\mathcal{O}}}\}$ be the inductive family of nets indexed by $\tilde\O\in\K$ (see Section \ref{sec:altext}), and let $\Gamma_{\tilde{\mathcal{O}}}$ be a unitary that implements $(\theta^L\theta^R)_{\tilde\O} := \theta_{\tilde{\mathcal{O}}}^{L}\theta_{\tilde{\mathcal{O}}}^{R}$ on $\A(\tilde{\mathcal{O}})$, namely
$$\theta_{\tilde{\mathcal{O}}}^{L}\theta_{\tilde{\mathcal{O}}}^{R}(x) = \Ad_{u^{L}_{\tilde{\mathcal{O}}}\theta^{L}(u^{R}_{\tilde{\mathcal{O}}})} \theta^{L}\theta^{R}(x) = \Ad_{\Gamma_{\tilde{\mathcal{O}}}}(x),\quad x\in\A(\tilde{\mathcal{O}}).$$

\begin{proposition}\label{prop:netsbrprodembeddings}
The maps
$$\jmath^{L,\tilde{\mathcal{O}}}:\mathcal{B}^{L}(\tilde{\mathcal{O}})\rightarrow (\mathcal{B}^{L}\times^{\pm}_{\eps}\mathcal{B}^{R})_{\tilde{\mathcal{O}}}(\tilde{\mathcal{O}}) \,, \quad\jmath^{L,\tilde{\mathcal{O}}} := \Ad_{\Gamma_{\tilde{\mathcal{O}}}^{*}} \circ \Ad_{(\eps^{\pm}_{{\theta^{L}_{\tilde{\mathcal{O}}},\theta^{R}_{\tilde{\mathcal{O}}}}})^*}\circ\,\theta^{R}_{\tilde{\mathcal{O}}}\circ\gamma^{L}_{\tilde{\mathcal{O}}}$$
$$\jmath^{R,\tilde{\mathcal{O}}}:\mathcal{B}^{R}(\tilde{\mathcal{O}})\rightarrow (\mathcal{B}^{L}\times^{\pm}_{\eps}\mathcal{B}^{R})_{\tilde{\mathcal{O}}}(\tilde{\mathcal{O}}) \,,\quad \jmath^{R,\tilde{\mathcal{O}}} := \Ad_{\Gamma_{\tilde{\mathcal{O}}}^{*}}\circ\,\theta^{L}_{\tilde{\mathcal{O}}}\circ\gamma^{R}_{\tilde{\mathcal{O}}}$$
lift to embeddings \footnote{The same is true in the representation employed in the first proof of Theorem \ref{thm:discreteextnets}, but it is more lengthy to check.} of $\{\mathcal{B}^{L}\}$ and $\{\mathcal{B}^{R}\}$ into the braided product net $\{\mathcal{B}^{L}\times^\pm_{\eps}\mathcal{B}^{R}\}$
$$\jmath^{L}:\mathcal{B}^{L}\rightarrow \mathcal{B}^{L}\times^{\pm}_{\eps}\mathcal{B}^{R}$$
$$\jmath^{R}:\mathcal{B}^{R}\rightarrow \mathcal{B}^{L}\times^{\pm}_{\eps}\mathcal{B}^{R}.$$
\end{proposition}

\begin{proof}
It is enough to show that $\jmath^{L/ R,\tilde{\mathcal{O}}}(\mathcal{B}^{L/ R}(\O))\subset (\mathcal{B}^{L}\times^{\pm}_{\eps}\mathcal{B}^{R})_{\tilde{\mathcal{O}}}(\O)$ for $\O\subset\tilde{\mathcal{O}}$ and $\iota^{\tilde{\mathcal{O}}_{1},\tilde{\mathcal{O}}_{2}}\circ\jmath^{L/ R, \tilde{\mathcal{O}}_{1}} = (\jmath^{L/ R, \tilde{\mathcal{O}}_{2}})_{\restriction \mathcal{B}^{L/ R}(\tilde{\mathcal{O}}_{1})}$
for $\tilde{\mathcal{O}}_{1}\subset \tilde{\mathcal{O}}_{2}$, but these follow from elementary calculations.
\end{proof}

\begin{proposition}\label{prop:brexpectation}
Let $E^{LR}$ denote the distinguished conditional expectation from $\{\B^L\times_\eps^\pm \B^R\}$ to $\{\A\}$ obtained as in Proposition \ref{prop:altexpectation}, then
$$E^{LR}(\jmath^{L}(b_{L})\jmath^{R}(b_{R}))=E^{L}(b_{L})E^{R}(b_{R}), \quad b_{L}\in \B^{L},\quad b_{R}\in \B^{R}.$$
\end{proposition}

\begin{proof}
Using $\theta^{L}({w^{R}}^{*})\eps^{\pm *}_{\theta^{L},\theta^{R}}={w^{R}}^{*}$ and ${w^{R}}^{*}\eps^{\pm}_{\theta^{L},\theta^{R}}=\theta^{R}({w^{R}}^{*})$ we get
$$\theta^{L}\theta^{R}({w^{L}}^{*}\theta^{R}({w^{R}}^{*})\eps^{\pm *}_{\theta^{L},\theta^{R}}\gamma^{R}\gamma^{L}(b_{L})\eps^{\pm}_{\theta^{L},\theta^{R}}\gamma^{L}\gamma^{R}(b_{R})\theta^{L}(w^{R})w^{L})$$
$$=\theta^{L}\theta^{R}({w^{L}}^{*}\gamma^{R}\gamma^{L}(b_{L})\theta^{L}({w^{R}}^{*}\gamma^{R}(b_{R})w^{R})w^{R})$$
$$=\theta^{L}\theta^{R}(E^{L}(b_{L}E^{R}(b_{R})))=\theta^{L}\theta^{R}(E^{L}(b_{L})E^{R}(b_{R}))$$
from which the proposition follows.
\end{proof}

For the rest of the section, we assume that $\{\mathcal{A}\subset\mathcal{B}^{L}\}$, $\{\mathcal{A}\subset\mathcal{B}^{R}\}$ are as in Proposition \ref{prop:netqsys}, so that the generalized Q-systems $(\theta^{L},w^{L},\{m^{L}_{i}\})$, $(\theta^{R},w^{R},\{m^{R}_{j}\})$ are induced by Pimsner-Popa bases of global charged fields $\{\psi^{L}_{\rho_i}\}\subset \mathcal{B}^{L}(\mathcal{O}), \{\psi^{R}_{\sigma_j}\}\subset \mathcal{B}^{R}(\mathcal{O})$, where $\rho_i\prec\theta^L$, $\sigma_j\prec\theta^R$ are respectively irreducible DHR subendomorphisms with finite dimension, localized in $\O\in\K$.

\begin{proposition}\label{prop:uniquenessbrexp}
$$\{\jmath^{L}(\{\psi^{L}_{\rho_i}\})\jmath^{R}(\{\psi^{R}_{\sigma_j}\})\}$$
is a Pimsner-Popa basis for $\iota(\mathcal{A}(\mathcal{O}))\stackrel{E^{LR}}{\subset} (\mathcal{B}^{L}\times^{\pm}_{\eps}\mathcal{B}^{R})(\mathcal{O})$ and it gives a unique Pimsner-Popa expansion (Proposition \ref{prop:pipoexpansion}).
\end{proposition}

\begin{proof}
The first statement is immediate. To prove the second statement, it is enough to show that 
$E^{LR}(\jmath^{L}(\psi_{\rho_i}^{L})\jmath^{R}(\psi^{R}_{\sigma_j})\jmath^{R}(\psi^{R *}_{\sigma^{\prime}_h})\jmath^{L}(\psi_{\rho^{\prime}_k}^{L *}))=\delta_{\rho_i,\rho'_k}\delta_{\sigma_j,\sigma'_h}$
which follows directly from a calculation analogous to the proof of Proposition \ref{prop:brexpectation}.
\end{proof}

In the following, with abuse of notation, we shall often suppress the above embeddings $\jmath^{L/ R}$, and $\iota:\A\rightarrow\B^L\times_\eps^\pm\B^R$ as well.

\begin{remark}\label{rmk:finindexbrprod}
In \cite{BKLR16} it is shown that the extension associated to the braided product of two \lqq ordinary" Q-systems is characterized algebraically in the following way. Let $\N\subset \M^{A}$ and $\N\subset \M^{B}$ be two finite index inclusions and let $(\theta^{A},w^{A},x^{A})$, $(\theta^{B},w^{B},x^{B})$ be the associated Q-systems. Denote by $\iota^{A/B}$ the respective inclusion maps and by $\theta^{A} = \oplus _{i}\rho_{i}$, $\theta^{B} = \oplus _{j}\sigma_{j}$ the irreducible decompositions of the dual canonical endomorphisms. Then it is known \cite[\Thm 3.11]{BKLR15} that $\M^{A}$ (\resp $\M^{B}$) is finitely generated by $\N$ and $\{\psi_{\rho_{i}}^{A}\}$ (\resp $\{\psi_{\sigma_{j}}^{B}\}$), where $\{\psi_{\rho_{i}}^{A}\}$, $\{\psi_{\sigma_{j}}^{B}\}$ are charged fields.

In this case, the braided product $\M^A\times^\pm_\eps\M^B$ can be completely characterized as the *-algebra freely generated by $\M^{A}$ and $\M^{B}$, modulo the relations
$$\iota^{A}(n)=\iota^{B}(n), \quad n\in \N,$$
$$\psi_{\sigma_{j}}^{B}\psi_{\rho_{i}}^{A}=\eps^{\pm}_{\rho_{i},\sigma_{j}}\psi_{\rho_{i}}^{A}\psi_{\sigma_{j}}^{B}.$$
In the discrete (infinite index) case this is no longer true since the extensions are \emph{not} finitely generated by $\N$ and the charged fields. We have to settle for a weaker form of this result, valid for pairs of \emph{irreducible} extensions, that will nevertheless prove to be useful in Section \ref{sec:exampleU(1)}. 
\end{remark}

Let $\mathfrak{B}^{L}\subset \mathcal{B}^{L}$ be the *-algebra generated by $\iota^{L}(\mathcal{A})$ and the charged fields $\{\psi_{\rho_i}^{L}\}\subset\mathcal{B}^{L}$. Similarly, let $\mathfrak{B}^{R}\subset \mathcal{B}^{R}$ be the *-algebra generated by $\iota^{R}(\mathcal{A})$ and the charged fields $\{\psi_{\sigma_j}^{R}\}\subset\mathcal{B}^{R}$. Let $\mathfrak{B}^{L\times R}\subset\mathcal{B}^{L}\times^\pm_{\eps}\mathcal{B}^{R}$ be the *-algebra generated by $\jmath^{L}(\iota^{L}(\mathcal{N}))=\jmath^{R}(\iota^{R}(\mathcal{N}))$, $\{\jmath^{L}(\psi_{\rho_i}^{L})\}\subset \mathcal{B}^{L}\times^\pm_{\eps}\mathcal{B}^{R}$ and $\{\jmath^{R}(\psi_{\sigma_j}^{R})\}\subset \mathcal{B}^{L}\times^\pm_{\eps}\mathcal{B}^{R}$. Then we have the following

\begin{lemma}\label{lem:freealg}
Suppose in addition that $\{\mathcal{A}\subset \mathcal{B}^{L}\}$ and $\{\mathcal{A}\subset \mathcal{B}^{R}\}$ are irreducible extensions. Then $\mathfrak{B}^{L\times R}$ is isomorphic to the *-algebra freely generated by $\mathfrak{B}^{L}$ and $\mathfrak{B}^{R}$, modulo the relations
$$\iota^{L}(a)=\iota^{R}(a),\quad a\in\mathcal{A},$$
$$\psi^{R}_{\sigma_j}\psi^{L}_{\rho_i}=\eps^{\pm}_{\rho_i,\sigma_j}\psi^{L}_{\rho_i}\psi^{R}_{\sigma_j}.$$
\end{lemma}

\begin{proof}
An arbitrary element $x$ of the free *-algebra generated by $\mathfrak{B}^{L}$ and $\mathfrak{B}^{R}$ modulo these relations can be written as a finite sum $x=\sum n_{\rho_i,\sigma_j}\psi^{L}_{\rho_i}\psi^{R}_{\sigma_j}$
in a unique way. The same is true for any element in $\mathfrak{B}^{L\times R}$ by Proposition \ref{prop:uniquenessbrexp} and Remark \ref{rmk:*algchargedintert}, thus the expansion yields an isomorphism.
\end{proof}

 In \cite{BKLR16} it was shown that the center of the braided product extension is an object of great interest since it contains all the information on transparent boundary conditions between the two starting quantum field theories. We here show that in the discrete case some relevant structural features are retained, in particular that the center of the braided product extension agrees with the relative commutant, which will be useful in the next section for the construction of irreducible phase boundaries from the central decomposition of the braided product. 

The expansion in terms of the Pimsner-Popa basis of charged fields (Proposition \ref{prop:uniquenessbrexp}) can be used to characterize the relative commutant.

\begin{lemma}\label{lem:exprelcomm}
For every $x\in(\mathcal{B}^{L}\times^\pm_{\eps}\mathcal{B}^{R})(\mathcal{O})$ we have
$$x\in(\mathcal{B}^{L}\times^\pm_{\eps}\mathcal{B}^{R})(\mathcal{O})\cap \mathcal{A}(\mathcal{O})^{\prime} \quad \Leftrightarrow \quad x=\sum_{\rho_i,\sigma_j}{\psi_{\sigma_j}^{R}}^{*}{\psi_{\rho_i}^{L}}^{*}r_{\rho_i,\sigma_j}$$ 
with $r_{\rho_i,\sigma_j} := E^{LR}(\psi^{L}_{\rho_i}\psi^{R}_{\sigma_j} x) \in\Hom_{\End_0(\mathcal{A}(\mathcal{O}))}(\id,\rho_i\sigma_j)$.
\end{lemma}

\begin{proof}
It is enough to use the uniqueness of the expansion in Proposition \ref{prop:uniquenessbrexp}
$$nx=\sum_{i,j}{\psi^{R}_{\sigma_j}}^*{\psi^{L}_{\rho_i}}^*E^{LR}(\psi^{L}_{\rho_i}\psi^{R}_{\sigma_j}nx)=\sum_{i,j}{\psi^{R}_{\rho_i}}^*{\psi^{L}_{\sigma_j}}^*\rho_i\sigma_j(n)E^{LR}(\psi^{L}_{\rho_i}\psi^{R}_{\sigma_j}x)$$
$$=\sum_{i,j}{\psi^{R}_{\sigma_j}}^{*}{\psi^{L}_{\rho_i}}^{*}E^{LR}(\psi^{L}_{\rho_i}\psi^{R}_{\sigma_j}xn)=\sum_{i,j}{\psi^{R}_{\sigma_j}}^{*}{\psi^{L}_{\rho_i}}^{*}E^{LR}(\psi^{L}_{\rho_i}\psi^{R}_{\sigma_j}x)n = xn$$
for every $n\in\A(\O)$, thus $r_{\rho_i,\sigma_j}\, n =\rho_i(\sigma_j(n))\,r_{\rho_i,\sigma_j}$.
\end{proof}

As in the finite index case, \cf \cite[\Prop 4.19]{BKLR16}, the center of the braided product of two \emph{local} extensions coincides with the relative commutant of $\{\A\}$ in the braided product.

\begin{proposition}\label{prop:relcomm=center}
Suppose in addition that $\{\mathcal{A}\subset \mathcal{B}^{L}\}$ and $\{\mathcal{A}\subset \mathcal{B}^{R}\}$ are local extensions, then
$$(\mathcal{B}^{L}\times^\pm_{\eps}\mathcal{B}^{R})(\mathcal{O})\cap \mathcal{A}(\mathcal{O})^{\prime}=(\mathcal{B}^{L}\times^\pm_{\eps}\mathcal{B}^{R})(\mathcal{O})\cap (\mathcal{B}^{L}\times^\pm_{\eps}\mathcal{B}^{R})(\mathcal{O})^{\prime}.$$

\begin{proof}
Let us first verify that the von Neumann algebra $\Z$ generated by $r_{\rho_i,\sigma_j}^{*}\psi^{L}_{\rho_i}\psi^{R}_{\sigma_j}$, with $r_{\rho_i,\sigma_j}\in\Hom_{\End_0(\mathcal{A}(\mathcal{O}))}(\id,\rho_i\sigma_j)$, is contained in the center. For every $n\in\A(\mathcal{O})$ we have
$$r_{\rho_i,\sigma_j}^{*}\psi_{\rho_i}^{L}\psi_{\sigma_j}^{R}  n \psi_{\rho^{\prime}_k}^{L}\psi_{\sigma^{\prime}_t}^{R} = n r_{\rho_i,\sigma_j}^{*}\rho_i(\eps^{\pm}_{\rho'_k,\sigma_j})\eps^{\pm}_{\rho'_k,\rho_i}\rho'_k\rho_i(\eps^{\pm}_{\sigma'_t,\sigma_j})\rho'_k(\eps^{\pm}_{\sigma'_t,\rho_i})\psi_{\rho'_k}^{L}\psi_{\sigma'_t}^{R}\psi_{\rho_i}^{L}\psi_{\sigma_j}^{R}$$
by direct computation and using locality of $\{\B^L\}$ and $\{\B^R\}$, \ie 
$$\psi_{\rho'_k}^{L}\psi_{\rho_i}^{L}=\eps^{\pm}_{\rho_i,\rho'_k}\psi_{\rho_i}^{L}\psi_{\rho'_k}^{L},$$
$$\psi_{\sigma'_t}^{R}\psi_{\sigma_j}^{R}=\eps^{\pm}_{\sigma_j,\sigma'_t}\psi_{\sigma_j}^{R}\psi_{\sigma'_t}^{R},$$
\cf Theorem \ref{thm:discreteextnets}. Now, it is easy to see that
$$r_{\rho_i,\sigma_j}^{*} \rho_i(\eps^{\pm}_{\rho'_k,\sigma_j})\eps^{\pm}_{\rho'_k,\rho_i}\rho'_k\rho_i(\eps^{\pm}_{\sigma'_t,\sigma_j})\rho'_k(\eps^{\pm}_{\sigma'_t,\rho_i}) = \rho'_k \sigma'_t(r_{\rho_i,\sigma_j}^{*})$$
from which $r_{\rho_i,\sigma_j}^{*}\psi^{L}_{\rho_i}\psi^{R}_{\sigma_j}$ is contained in the center of the braided product.

For brevity, in the following we denote $B\cap A^{\prime}=(\mathcal{B}^{L}\times^\pm_{\eps}\mathcal{B}^{R})(\mathcal{O})\cap \mathcal{A}(\mathcal{O})^{\prime}$ and $B\cap B^{\prime}=(\mathcal{B}^{L}\times^\pm_{\eps}\mathcal{B}^{R})(\mathcal{O})\cap (\mathcal{B}^{L}\times^\pm_{\eps}\mathcal{B}^{R})(\mathcal{O})^{\prime}$, and consider the inclusions 
$$\Z\subset B\cap B^{\prime}\subset B\cap A^{\prime}.$$
If we take the GNS representation of $B\cap A^{\prime}$ with respect to the vacuum $\Omega$, we get a cyclic and separating vector for $B\cap A^{\prime}$ and for $B\cap B^{\prime}$ as well, by Lemma \ref{lem:exprelcomm}. Now we check that the canonical conjugations of $B\cap A^{\prime}$ and $B\cap B^{\prime}$ with respect to $\Omega$ agree. This holds because the Tomita operator $S$ of $(B\cap A^{\prime},\Omega)$, \ie, the closure of the operator $S_{0}:x\Omega\rightarrow x^{*}\Omega$, $x\in B\cap A^{\prime}$, is an extension of the Tomita operator of $(B\cap B^{\prime}, \Omega)$. Since the latter is continuous and defined on all the GNS Hilbert space (because $B\cap B^{\prime}$ is abelian), the two operators agree and coincide with the respective canonical conjugations. Thus
$$J(B\cap A^{\prime})J=(B\cap A^{\prime})^{\prime}\subset (B\cap B^{\prime})^{\prime}=J(B\cap B^{\prime})J\subset J(B\cap A^{\prime})J$$
from which the result follows.
\end{proof}
\end{proposition}

Lastly, as an application of Theorem \ref{thm:discretecovariance}, we show covariance of the braided product net.

\begin{proposition}\label{prop:covarbrprod}
Let $\{\A\}$ be a local net, covariant with respect to $\cP$ as in the assumptions of Theorem \ref{thm:discretecovariance}. Let $\{\B^L\}$ and $\{\B^R\}$ be two extensions of $\{\A\}$ constructed as in Theorem \ref{thm:discreteextnets} from unital generalized net Q-systems of intertwiners $(\theta^L, w^L, \{m_i^L\})$ and $(\theta^R, w^R, \{m_j^R\})$. Assume that $\cP$ acts equivariantly on two tensor subcategories $\C^L$ and $\C^R$ of $\DHR\{\A\}$ which contain respectively $\theta^L$ and $\theta^R$. Then the braided product net $\{\B^L\times_\eps^{\pm}\B^R\}$ is also covariant with respect to $\cP$. Moreover, the embeddings $\jmath^{L}$ and $\jmath^{R}$ given in Proposition \ref{prop:netsbrprodembeddings} are covariant as representations, namely $\jmath^{L} \circ \hat\alpha_g^{L} (b^L)= \hat\alpha_g^{LR} \circ \jmath^L (b^L)$ and $\jmath^{R} \circ \hat\alpha_g^{R} (b^R) = \hat\alpha_g^{LR} \circ \jmath^R (b^R)$, where $b^{L/R}\in\B^{L/R}(\O)$, for every $g\in\cU_{\O}$ and $\O\in\K$.
\end{proposition}

\begin{proof}
$\cP$ acts equivariantly on $\theta^L\theta^R$ and on the (full, replete) tensor subcategory $\cD$ generated in $\DHR\{\A\}$ by $\theta^L$ and $\theta^R$. Indeed, the cocycle given by $z(g,\theta^L\theta^R) = z(g,\theta^L) \theta^L(z(g,\theta^R))$, $g\in\cP$, is manifestly natural and tensor in $\cD$, hence we can apply Theorem \ref{thm:discretecovariance}. 

The second statement follows from $\hat\alpha_g^{LR}(\jmath^L(M_i^L)) = \jmath^L (\hat\alpha_g^L(M_i^L))$ and $\hat\alpha_g^{LR}(\jmath^R(M_j^R)) = \jmath^R (\hat\alpha_g^R(M_j^R))$ by direct computation using naturality of cocycles.
\end{proof}

%%%
\section{Applications to phase boundaries in QFT}\label{sec:apptodefectsinQFT}
%%%

The main application in QFT for the braided product of ordinary Q-systems in \cite{BKLR16} is the construction and classification of \emph{phase boundary QFTs}.

A boundary is simply a time-like hypersurface of codimension 1 in Minkowski spacetime $\RR^{n+1}$, $n\geq 1$, or a point in $\RR$. 
Perhaps the simplest type of boundary QFT is a system in a one-sided box. Namely, on one side of the boundary (the side of the box) there is a physical system described by bulk fields, while on the other side there is no physical content. This situation is usually referred to as a \emph{hard boundary}, or \emph{reflective boundary}. 

In the following we will be concerned with \emph{phase boundaries}, also called \emph{transmissive boundaries}, which describe QFTs sharing some distinguished chiral fields across the boundary (for example the stress-energy tensor) but the field content may in general be different on the two opposite sides. If the common fields which are not affected by the presence of the boundary include the stress-energy tensor, then the bulk fields may be defined by covariance on all Minkowski spacetime. Of course they do not represent physically meaningful quantities when they are transported to the opposite side of the boundary. In any case, this observation is crucial for the meaningfulness of the following definition.

Let $\{\A\}$ be a local net and let $\{\mathcal{A}\subset\mathcal{B}^{L}\}$, $\{\mathcal{A}\subset\mathcal{B}^{R}\}$ be two \emph{local} extensions (see Definition \ref{def:inclnets}). Let $\iota^{L}$ and $\iota^{R}$ be the corresponding embeddings. $\mathbb{M}^{L}$ and $\mathbb{M}^{R}$ denote the two portions of Minkowski spacetime determined by the boundary.

\begin{definition}\label{def:phaseboundary}
A \textbf{phase boundary condition} (for short \textbf{phase boundary}) between two local extensions $\{\mathcal{A}\subset\mathcal{B}^{L}\}$, $\{\mathcal{A}\subset\mathcal{B}^{R}\}$ is a pair of locally normal representations $\pi^{L}$ and $\pi^{R}$ of the nets $\{\mathcal{B}^{L}\}$ and $\{\mathcal{B}^{R}\}$, respectively, on a common Hilbert space $\mathcal{H}$, with the following properties.
They agree when restricted to the common subnet $\mathcal{A}$, namely 
$$\pi^{L}\circ \iota^{L}=\pi^{R}\circ \iota^{R}$$
and, for $\mathcal{O}_{1}\subset\mathbb{M}^{L}$, $\mathcal{O}_{2}\subset\mathbb{M}^{R}$, and $\mathcal{O}_{1}, \mathcal{O}_{2}$ in relative space-like position, $\pi^{L}(\mathcal{B}^{L}(\mathcal{O}_{1}))$ and $\pi^{R}(\mathcal{B}^{R}(\mathcal{O}_{2}))$ commute, \ie, they respect \emph{locality} across the boundary.

A phase boundary is called \textbf{irreducible} if the inclusions 
$$\pi^{L/R}\circ\iota^{L/R}(\mathcal{A}(\mathcal{O}))\subset \pi^{L}(\mathcal{B}^{L}(\mathcal{O}))\vee \pi^{L}(\mathcal{B}^{R}(\mathcal{O}))$$
are irreducible for every $\mathcal{O}\in\mathcal{K}$.
\end{definition}

In the present setting, we show that the braided product can be decomposed over its center (in general as a direct integral) and its components give rise to irreducible phase boundaries, in analogy to the finite index case.

\begin{remark}
A prominent feature of the finite index case is that the phase boundaries found within the braided product net by central decomposition do exhaust the set of all possible irreducible phase boundaries modulo unitary equivalence. The proof of the latter heavily relies on the finiteness of the index since this insures that the braided product construction can be completely determined algebraically as the free $^{*}$-algebra generated by the starting nets $\{\mathcal{B}^{L}\}$ and $\{\mathcal{B}^{R}\}$ modulo relations as in Remark \ref{rmk:finindexbrprod}. This makes the braided product a universal object in the sense that every irreducible phase boundary condition arises a representation of the former \cite[\Prop 5.1]{BKLR16}. In the infinite index setting this is no longer the case as we will see in Section \ref{sec:exampleU(1)}.
\end{remark}

For ease of exposition, we state the results for \emph{chiral} CFTs (and thus phase boundaries in 1D) although the analysis can be extended to greater generality without difficulty.
Moreover, in order to avoid inconvenient technicalities with disintegration theory, we assume that the starting local extensions have the \emph{split property}, \cite{DoLo84}. This assumption is not too restrictive since most interesting models in QFT have this property, in particular all chiral diffeomorphism covariant models \cite{MTW16}.

Let $\{\mathcal{A}\}$ be a local conformal net (M\"obius covariant, see Definition \ref{def:covarnet}) on $\mathbb{R}$ over a separable Hilbert space and satisfying Haag duality on $\RR$. Exactly as in the notation of \cite[\Prop 55]{KLM01}, for $I, \tilde{I}\in\K$ (here $\K$ is the set of open proper bounded intervals of $\mathbb{R}$), $I\subset\subset\tilde{I}$ means that $\bar I\subset\tilde{I}$. If $\{\A\}$ has the split property, then, for each pair of intervals $I\subset\subset\tilde{I}$, there is an intermediate type $I$ factor $\mathcal{A}(I)\subset N(I,\tilde{I})\subset\mathcal{A}(\tilde{I})$ and we denote by $K(I,\tilde{I})$ the compact operators of $N(I,\tilde{I})$. $\mathcal{I}_{\mathbb{Q}}$ is the set of intervals with rational endpoints and $\mathfrak{A}$ is the \emph{separable} $C^{*}$-subalgebra of $\mathcal{A}$ generated by all $K(I,\tilde{I})$ with $I\subset\subset\tilde{I}$, $I,\tilde{I}\in\mathcal{I}_{\mathbb{Q}}$.

\begin{proposition}\label{prop:KLM}\emph{\cite{KLM01}}.
Let $\pi$ be a locally normal representation of $\mathcal{A}$. Then $\pi_{\restriction \mathfrak A}$ is a representation of $\mathfrak A$ and $\pi_{\restriction K(I,\tilde{I})}$ is non-degenerate for every pair of intervals $I\subset\subset\tilde{I}$. 

Conversely, if $\sigma$ is a representation of $\mathfrak A$ such that $\sigma_{\restriction K(I,\tilde{I})}$ is non-degenerate for all intervals $I$, $\tilde{I}\in\mathcal{I}_{\mathbb{Q}}$, $I\subset\subset\tilde{I}$, there exists a unique locally normal representation $\tilde{\sigma}$ of $\mathcal{A}$ that extends $\sigma$. Moreover, equivalent representations of $\mathcal{A}$ correspond to equivalent representations of $\mathfrak A$.
\end{proposition}

Now, let $\{\mathcal{B}^{L}\}$ and $\{\mathcal{B}^{R}\}$ be local conformal nets extending $\{\A\}$ as in Definition \ref{def:inclnets}. Assume that $\{\B^{L/R}\}$ have the split property and that $\{\mathcal{A}\subset\mathcal{B}^{L/R}\}$ are discrete irreducible extensions with corresponding unital generalized net Q-systems of intertwiners $(\theta^L, w^L, \{m_i^L\})$, $(\theta^R, w^R, \{m_j^R\})$ given by global charged fields as in Proposition \ref{prop:netqsys}.

Define $K_{L}(I,\tilde{I})$, $K_{R}(I,\tilde{I})$, and the separable \Cstar-algebras $\mathfrak{B}_{L}$, $\mathfrak{B}_{R}$ as above. Using the last proposition, we want to show that the embedding homomorphisms $\jmath^{L}$ and $\jmath^{R}$ into the braided product (Proposition \ref{prop:netsbrprodembeddings}) can be decomposed as representations with respect to the center of the braided product. Note that by Proposition \ref{prop:relcomm=center} and \ref{prop:covarbrprod} the centers of the local algebras of the braided product agree, namely we have
$$\Z((\mathcal{B}^{L}\times^\pm_{\eps}\mathcal{B}^{R})(I))=\Z((\mathcal{B}^{L}\times^\pm_{\eps}\mathcal{B}^{R})(J))=\Z(\mathcal{B}^{L}\times^\pm_{\eps}\mathcal{B}^{R})$$
for every $I,J\in\mathcal{K}$.

\begin{proposition}\label{prop:intdecomp}
Let
$${\jmath^{L}}_{\restriction \mathfrak{B}_{L}}\cong \int_{X}^{\oplus}\jmath^{L}_{\lambda}d\mu(\lambda)$$
$${\jmath^{R}}_{\restriction \mathfrak{B}_{R}}\cong \int_{X}^{\oplus}\jmath^{R}_{\lambda}d\mu(\lambda)$$

be the disintegration of the restrictions of the embeddings $\jmath^{L}$, $\jmath^{R}$ to the separable \Cstar-subalgebras $\mathfrak{B}_{L}$ and $\mathfrak{B}_{R}$ with respect to the center of the braided product $\Z(\mathcal{B}^{L}\times^\pm_{\eps}\mathcal{B}^{R})\cong L^{\infty}(X,d\mu)$. Then, for $d\mu$-almost every $\lambda\in X$, the $\jmath^{L}_{\lambda}$ and $\jmath^{R}_{\lambda}$ lift to locally normal representations of the quasilocal \Cstar-algebras $\mathcal{B}^{L}$ and $\mathcal{B}^{R}$ respectively.
\begin{proof}
To prove the assertion, by the above proposition, it is enough to show that there is a $d\mu$-null set $E$ such that ${\jmath^{L}_{\lambda}}_{\restriction K_{L}(I,\tilde{I})}$ (\resp ${\jmath^{R}_{\lambda}}_{\restriction K_{R}(I,\tilde{I})}$) is non-degenerate for every $\lambda\notin E$ and $I,\tilde{I}\in\mathcal{I}_{\mathbb{Q}}$, $I\subset\subset\tilde{I}$. This is easily checked, because for fixed $I,\tilde{I}$, $\jmath^{L}_{\restriction K_{L}(I,\tilde{I})}$(\resp $\jmath^{R}_{\restriction K_{R}(I,\tilde{I})}$) is non-degenerate by Proposition \ref{prop:KLM} and consequently ${\jmath^{L}_{\lambda}}_{\restriction K_{L}(I,\tilde{I})}$ (\resp ${\jmath^{L}_{\lambda}}_{\restriction K_{L}(I,\tilde{I})}$) is also non-degenerate for $d\mu$-almost every $\lambda\in X$. Since $I,\tilde{I}\in\mathcal{I}_{\mathbb{Q}}$, $I\subset\subset\tilde{I}$ are countable, the statement follows.
\end{proof}
\end{proposition}

\begin{proposition}\label{prop:defects}
Let $\lambda\in X\smallsetminus E$ as above. Then
\begin{itemize}
\item[$(1)$] $\jmath^{R}_{\lambda}\circ\iota^{R}=\jmath^{L}_{\lambda}\circ\iota^{L}$.
\item[$(2)$] $\jmath^{R}_{\lambda}(m_j^R)\jmath^{L}_{\lambda}(m_i^L)=\jmath^{L}_{\lambda}(\eps^{\pm}_{\theta^L,\theta^R})\jmath^{L}_{\lambda}(m_i^L)\jmath^{R}_{\lambda}(m_j^R)$.
\item[$(3)$] If $\hat U(g)=\int_{X}\hat U_{\lambda}(g)d\mu$ is the disintegration of the representation of the universal covering of the M\"obius group (given by Proposition \ref{prop:covarbrprod}) with respect to the center of the braided product, then
$$\hat U_{\lambda}(g)\jmath_{\lambda}^{L}(\mathcal{B}^{L}(J)) \hat U_{\lambda}(g)^{*}=\jmath_{\lambda}^{L}(\mathcal{B}^{L}(gJ))$$
$$\hat U_{\lambda}(g)\jmath_{\lambda}^{R}(\mathcal{B}^{R}(J))\hat U_{\lambda}(g)^{*}=\jmath_{\lambda}^{R}(\mathcal{B}^{R}(gJ)).$$
\item[$(4)$] Let $\Omega=\int_{X}\Omega_{\lambda}d\mu$, then $\hat U_{\lambda}(g)\Omega_{\lambda}=\Omega_{\lambda}$ and $\Omega_{\lambda}$ is cyclic for 
$$\bigvee_{J}\jmath^{L}_{\lambda}(\mathcal{B}^{L}(J))\vee \jmath^{R}_{\lambda}(\mathcal{B}^{R}(J)).$$
\item[$(5)$] $\jmath^{L}_{\lambda}(\mathcal{B}^{L}(I))\vee \jmath^{R}_{\lambda}(\mathcal{B}^{R}(I))$ is a factor.
\item[$(6)$] The inclusion $\jmath^{L}_{\lambda}(\iota^{L}(\mathcal{A}(I)))\subset \jmath^{L}_{\lambda}(\mathcal{B}^{L}(I))\vee \jmath^{R}_{\lambda}(\mathcal{B}^{R}(I))$ is irreducible.
\end{itemize}

\begin{proof}
Most of these assertions are trivial and follow from standard techniques in disintegration theory. Covariance, \ie, point $(3)$, follows by Proposition \ref{prop:covarbrprod}, Example \ref{ex:moebcovar}  and by the fact that $\hat U(g)\in \Z(\mathcal{B}^{L}\times^\pm_{\eps}\mathcal{B}^{R})^{\prime}$ by the expansion in Lemma \ref{lem:exprelcomm}.
\end{proof}
\end{proposition}
\begin{remark}
Proposition \ref{prop:defects} shows that the braided product construction of two net extensions $\{\mathcal{A}\subset\mathcal{B}^{L}\}$, $\{\mathcal{A}\subset\mathcal{B}^{R}\}$ with the required properties induces, via central decomposition, a family of \emph{irreducible} phase boundaries $(\jmath_\lambda^L,\jmath_\lambda^R)$ indexed by the spectrum $X$ (up to a measure zero set) of the center of $\mathcal{B}^{L}\times^\pm_{\eps}\mathcal{B}^{R}$ and living on the Hilbert space $\mathcal{H}_{\lambda}$, $\lambda\in X$.

Of course, depending on whether $\{\mathcal{B}^{L}\}$ and $\{\mathcal{B}^{R}\}$ are interpreted to be theories respectively on the left and on the right of the boundary, or vice versa, one has to take the braided product with the correct sign, namely with $\eps^+$ or $\eps^-$.
\end{remark}

%%%
\section{An example with the $U(1)$-current}\label{sec:exampleU(1)}
%%%

In this section we work out concretely the braided product between local extensions of the \emph{$U(1)$-current net}. We will see examples where the center of the braided product net is a continuous algebra and therefore the direct integral representation as in Proposition \ref{prop:intdecomp} does not reduce to a direct sum. This shows in particular that the braided product is not a universal object in the sense of \cite[\Prop 5.1]{BKLR16}. This behaviour is expected, since, as in the finite index case, phase boundary conditions for orbifold theories should be determined by their gauge group, see \cite[\Sec 6.2]{BKLR16}. We will show the manifestation of this fact in at least one example. 

For the definition of the $U(1)$-current $\{\A_{U(1)}\}$ we refer to \cite{BMT88}, \cite{GLW98}, \cite{Lon2}, and to \cite[\Ch 12]{DelPhd} for more detailed calculations.
Let $I$ be a proper interval of $\Sone\smallsetminus \{1\}$ and let $f\in C^{\infty}(\Sone,\mathbb{R})$ with support contained in $I$. Define the net representation $\{\rho_{f, J}\}_{J}$ first on Weyl operators $W(g)$ in the following way
$$\rho_{f, J}(W(g)):=e^{i\int{f(\theta)g(\theta)\frac{d\theta}{2\pi}}}W(g)$$
for $g\in C^{\infty}(\Sone,\mathbb{R})$ with support in a proper interval $J$ of $\Sone\smallsetminus\{1\}$.
These above defined maps are locally unitarily implemented: let $I_{0}$ be a proper interval of $\Sone\smallsetminus\{1\}$ disjoint from $I$ and $J$, and let $f_{0}\in C^{\infty}(\Sone,\mathbb{R})$ with support in $I_{0}$ and such that $\int_{\Sone}{f}=\int_{\Sone}{f_{0}}$. Define $L_{I\rightarrow I_{0}}$ as a primitive of $f_{0}-f$, namely $L_{I\rightarrow I_{0}}^{\prime}=f_{0}-f$. It is an easy calculation to show that
$$W(L_{I\rightarrow I_{0}})W(g)W(L_{I\rightarrow I_{0}})^{*}=\rho_{f,J}(W(g)).$$
Thus the maps $\{\rho_{f, J}\}_{J}$ can be extended in a unique way to the local von Neumann algebras and they determine a locally normal representation of $\{\A_{U(1)}\}$, which is clearly DHR. Moreover these representations are classified up to unitary equivalence by the value $\int_{\Sone}{f}$ which is usually referred to as the \emph{charge}, thus yielding a continuous family of irreducible DHR sectors.

We now compute explicitly the braiding operator for the irreducible DHR representations described above. Let $\rho_{f}$ be localized in the interval $I$. If $\hat{I}$ is an interval disjoint from $I$ and $I<\hat{I}$, take $\hat{f}\in C^{\infty}({\Sone,\mathbb{R}})$ with support in $\hat{I}$ and with same charge as $f$, \ie, $\int_{\Sone}{f}=\int_{\Sone}{\hat{f}}$. If we denote by $u_{\hat{I}}:=W(L_{I\rightarrow \hat{I}})\in \Hom_{\DHR\{\A\}}(\rho_{f}, \rho_{\hat{f}})$ the charge transporter between $\rho_{f}$ and $\rho_{\hat{f}}$, by definition the braiding operator $\eps^{+}_{\rho_{f},\rho_{f}}$ is obtained by
$$\eps^{+}_{\rho_{f},\rho_{f}}=u_{\hat{I}}^{*}\rho_{f}(u_{\hat{I}}).$$
Performing the computation we get 
$$\eps^{+}_{\rho_{f},\rho_{f}}=e^{i\int{fL_{I\rightarrow \hat{I}}}}u_{\hat{I}}^{*}(u_{\hat{I}})=e^{-i\pi Q^{2}}$$
where Q is the charge of the DHR sector $\rho_{f}$. In particular $\eps^{+}_{\rho_{f},\rho_{f}}=\oneop$ if and only if $Q=\sqrt{2\pi N}$ with $N\in\mathbb{N}$.

\subsection{Buchholz-Mack-Todorov extensions}

We here quickly review the \emph{local} extensions of the $U(1)$-current net constructed in \cite{BMT88}. Let $\rho_{f}$ be a DHR automorphism of the $U(1)$-current net localized in the interval $I$ as above, such that $\eps^{+}_{\rho_{f},\rho_{f}}=\oneop$. To shorten notation denote $\rho=\rho_{f}$.
Any such automorphism gives a local extension of the net by a crossed product with the group $\mathbb{Z}$ which acts on the net as powers of $\rho$. Let 
$$\hat{\mathcal{H}}:=\bigoplus_{k\in\mathbb{Z}}\mathcal{H}_{k}$$
with $\mathcal{H}_{k}=\mathcal{H}$ (= vacuum Hilbert space of the $U(1)$-current net) and let $\pi$ be a representation of the quasilocal \Cstar-algebra $\mathcal{A}_{U(1)}$ of the net restricted to $\RR \cong \Sone\smallsetminus\{1\}$, defined as

$$\pi:\mathcal{A}_{U(1)}\rightarrow B(\hat{\mathcal{H}})$$
$$\pi(a):=\bigoplus_{k\in\mathbb{Z}}\rho^{k}(a)$$

Denote by $U$ the shift operator on $\hat{\mathcal{H}}$, \ie, $U\{\xi_{k}\}_{k\in\mathbb{Z}}=\{\xi_{k+1}\}_{k\in\mathbb{Z}}$ for $\xi\in\hat{\mathcal{H}}$.
It is clear that the shift operator $U$ implements the localized automorphism $\rho$ in this representation
$$U\pi(a)U^{*}=\pi(\rho(a)).$$
In other words $U$ is a charged field for $\rho$.

\begin{definition}
The BMT (Buchholz-Mack-Todorov) extension $\{\mathcal{B}_{\rho}\}=\{\mathcal{A}_{U(1)}\rtimes_{\rho}\mathbb{Z}\}$ is the net given by
$$\mathcal{B}_{\rho}(I):=\left\langle \pi(\mathcal{A}(I)), U\right\rangle $$
$$\mathcal{B}_{\rho}(J):=\left\langle \pi(\mathcal{A}(J)), \pi(u_{J})U\right\rangle .$$
\end{definition}

It is an easy matter to check that this definition is well posed and the net is isotonous (it follows directly from Haag duality of the $U(1)$-current net on $\RR$, \ie, strong additivity). Locality of BMT extensions $\{\B_\rho\}$ follows from $\eps^{+}_{\rho,\rho}=\eps^{-}_{\rho,\rho}=\oneop$, \cf Theorem \ref{thm:discreteextnets}. The inclusion $\{\A\subset\B_\rho\}$ is clearly discrete and irreducible.

The DHR automorphisms of the $U(1)$-current extend to representations of the net $\{\mathcal{B}_{\rho}\}$, and the DHR sectors of BMT extensions were already classified in \cite{BMT88}. We recall these facts to establish the notation.

\begin{proposition}\emph{\cite{BMT88}}.
For every DHR automorphism $\sigma$ of $\mathcal{A}_{U(1)}$ there are two locally normal representations $\tilde{\sigma}^{\pm}$ of $\mathcal{B}_\rho$ such that $\tilde{\sigma}^{\pm}(\pi(a)) = \pi(\sigma(a))$, $a\in\mathcal{A}_{U(1)}$. Moreover, $\tilde{\sigma}^{+}=\tilde{\sigma}^{-}$ if and only if $\tilde{\sigma}^{+}$ (or equivalently $\tilde{\sigma}^{-}$) is a DHR representation of the net $\{\mathcal{B}_\rho\}$, if and only if $\eps^{+}_{\rho,\sigma}=\eps^{-}_{\rho,\sigma}$. Otherwise $\tilde{\sigma}^{\pm}$ have solitonic localization (they are localizable in half-lines). In particular, there are $2N$ inequivalent DHR automorphisms of the net $\{\mathcal{B}_\rho\}$, where $Q=\sqrt{2\pi N}$ is the charge of $\rho$.
\end{proposition}

\begin{proof}
The automorphisms $\tilde{\sigma}^{\pm}$ can be defined by $\alpha$-induction of $\sigma$ for the extension $\{\mathcal{A}_{U(1)}\subset\mathcal{B}_\rho\}$, \cite[\Prop 3.9]{LoRe95}, but we here describe them explicitly since we will need them in the following. We first define the action of $\tilde{\sigma}^{\pm}$ on the *-algebra ${\mathfrak{B}}$ generated by $\pi(\mathcal{A}_{U(1)})$ and the shift $U$. Define
$$\tilde{\sigma}^{\pm}(\pi(a)):=\pi(\sigma(a))$$
$$\tilde{\sigma}^{\pm}(U^{n}):=\pi(\eps^{\pm}_{\rho^{n},\sigma})U^{n}$$
To check that this is a well defined endomorphism of the *-algebra it is enough to check that 
$$\tilde{\sigma}^\pm(U^{*})=\tilde{\sigma}^\pm(U)^{*},\quad\tilde{\sigma}^\pm(U)\tilde{\sigma}^\pm(\pi(a))\tilde{\sigma}^\pm(U)^{*}=\tilde{\sigma}^\pm(\pi(\rho(a))).$$
The first relation is an immediate consequence of naturality of the braiding, for the second we have
$$\tilde{\sigma}^\pm(U)\tilde{\sigma}^\pm(\pi(a))\tilde{\sigma}^\pm(U)^{*}=
\pi(\eps^{\pm}_{\rho,\sigma})U\pi(\sigma(a))U^{*}\pi(\eps^{\pm}_{\rho,\sigma})^{*}$$
$$=\pi(\eps^{\pm}_{\rho,\sigma})\pi(\rho(\sigma(a)))\pi(\eps^{\pm}_{\rho,\sigma})^{*}
=\pi(\sigma(\rho(a))).$$

Now, observe that for a fixed proper bounded interval $J$ of $\mathbb{R}$, the endomorphism $(\tilde\sigma)^{\pm}$ restricted to $\mathcal{B}_\rho(J)\cap {\mathfrak{B}}$ is locally implemented by the unitary $\pi(u_{\hat I}) := \pi(W(L_{I\rightarrow \hat{I}}))$ where $\hat{I}$ is a proper bounded interval where $\hat{I}<J$ if we consider $\tilde{\sigma}^{+}$ and $J<\hat{I}$ if we consider $\tilde{\sigma}^{-}$, \ie, $\tilde{\sigma}^\pm(b)=\Ad_{\pi(u_{\hat{I}})}(b)$ for every $b\in\mathcal{B}_\rho(J)\cap {\mathfrak{B}}$. Since $\mathcal{B}_\rho(J)\cap {\mathfrak{B}}$ is ultraweakly dense in $\mathcal{B}_\rho(J)$, the endomorphism can be extended in a unique way consistently on every local algebra.

Regarding the localization of $\tilde{\sigma}^\pm$, if $J<I$
$$\tilde{\sigma}^{+}(\pi(u_{J})U)=\pi(\sigma(u_{J}))\pi(\eps^{+}_{\rho,\sigma})U$$
$$=\pi(\sigma(u_{J}))\pi(\sigma(u_{J}))^{*}\pi(u_{J})U=\pi(u_{J})U$$
Similarly for $I<J$ we have the same result for $\tilde{\sigma}^{-}$. Thus they are localizable in half-lines, a priori, and also DHR if and only if $\eps^{+}_{\rho,\sigma}=\eps^{-}_{\rho,\sigma}$.
\end{proof}

\subsection{Braided product of BMT extensions}

Let $\{\mathcal{B}_{\rho_{L}}\}, \{\mathcal{B}_{\rho_{R}}\}$ be two local BMT extensions of the $U(1)$-current net given by two DHR automorphisms $\rho_{L}$ and $\rho_{R}$ as in the previous section. We would like to construct the braided product of two such nets in a concrete fashion. Let
$$\bigoplus_{l\in\mathbb{Z}}\hat{\mathcal{H}}=\bigoplus_{(l,h)\in\mathbb{Z}^{2}}\mathcal{H}$$
where $\mathcal{H}$ is the vacuum Hilbert space of the $U(1)$-current net, and $\Omega_{\A_{U(1)}}$ is the vacuum vector. We denote $\hat\Omega=\{\hat\Omega_{l,h}\}_{(l,h)\in\mathbb{Z}^{2}}$, with  $\hat\Omega_{l,h}:=\delta_{l,0}\delta_{h,0}\Omega_{\A_{U(1)}}$.

Let $\iota_{L}$ be the solitonic representation of $\mathcal{B}_{\rho_{L}}$ defined on the above Hilbert space as follows
$$\iota_{L}: \mathcal{B}_{\rho_{L}}\rightarrow \bigoplus_{l\in\mathbb{Z}}\B(\hat{\mathcal{H}})\subset \B(\bigoplus_{(l,k)\in\mathbb{Z}^{2}}\mathcal{H})$$
$$\iota_{L}:=\bigoplus_{l\in\mathbb{Z}}\tilde{\rho}_{R}$$
and similarly for $\hat{\iota}_{R}$
$$\hat{\iota}_{R}: \mathcal{B}_{\rho_{R}}\rightarrow \bigoplus_{h\in\mathbb{Z}}\B(\hat{\mathcal{H}})\subset \B(\bigoplus_{(l,k)\in\mathbb{Z}^{2}}\mathcal{H})$$
$$\hat{\iota}_{R}:=\bigoplus_{h\in\mathbb{Z}}\tilde{\rho}_{L}$$
Define 
$$\hat{\eps}:\bigoplus_{(l,h)\in\mathbb{Z}^{2}}\mathcal{H}\rightarrow \bigoplus_{(l,h)\in\mathbb{Z}^{2}}\mathcal{H}$$
$$\hat{\eps}\{\xi_{l,h}\}_{(l,h)\in\mathbb{Z}^{2}}:=\{\eps^{\pm}_{\rho_{L}^{l},\rho_{R}^{h}}\xi_{l,h}\}_{(l,h)\in\mathbb{Z}^{2}}$$
and twist the representation $\hat{\iota}_{R}$ by $\hat{\eps}$
$$\iota_{R}:=\Ad_{\hat{\eps}}(\hat{\iota}_{R}(\cdot)).$$
Observe that
$$\iota_{L}\circ\pi_{L}=\iota_{R}\circ \pi_{R}$$
where $\pi_{L}$ and $\pi_{R}$ are the inclusion maps of $\A_{U(1)}$ into $\mathcal{B}_{\rho_L}$ and $\mathcal{B}_{\rho_R}$ respectively, explicitly
$$(\iota_{L}\circ\pi_{L})(a)=\bigoplus_{k,l\in\mathbb{Z}}\rho_{R}^{h}(\rho_{L}^{l}(a))=\Ad_{\hat{\eps}}(\bigoplus_{k,l\in\mathbb{Z}}\rho_{L}^{l}(\rho_{R}^{h}(a)))=(\iota_{R}\circ\pi_{R})(a)$$
for every $a\in\A_{U(1)}$. Let $U\in \mathcal{B}_{\rho_{L}}(I)$ and $V\in\mathcal{B}_{\rho_{R}}(I)$ be the charged fields for the DHR automorphisms $\rho_L$ and $\rho_R$ respectively. Then

\begin{proposition}\label{prop:u1rel}
$$\iota_{R}(V)\iota_{L}(U)=\iota_{L}(\pi_{L}(\eps^{\pm}_{\rho_{L},\rho_{R}}))\,\iota_{L}(U)\iota_{R}(V)$$
\end{proposition}

\begin{proof}
By direct computation.
\end{proof}

\begin{proposition}
Let $\{\mathcal{B}_{\rho_{L}}\}$ and $\{\mathcal{B}_{\rho_{R}}\}$ two local BMT extensions as above. The net of von Neumann algebras defined by
$$\hat{\mathcal{B}}^{\pm}(I):=\left\langle \iota_{L}\circ\pi_{L}(\mathcal{A}_{U(1)}(I)),\iota_{L}(U),\iota_{R}(V)\right\rangle ,$$
$$\hat{\mathcal{B}}^{\pm}(J):=\left\langle \iota_{L}\circ\pi_{L}(\mathcal{A}_{U(1)}(J)),\iota_{L}\circ\pi_{L}(u_{J})\iota_{L}(U),\iota_{R}\circ\pi_{R}(v_{J})\iota_{R}(V)\right\rangle,$$
where $u_{J}$ and $v_{J}$ are unitary charge transporters respectively for $\rho_{L}$ and $\rho_{R}$ between intervals $I$ and $J$ (\ie the endomorphisms $\Ad_{u_{J}}\rho_{L}$ and $\Ad_{v_{J}}\rho_{R}$ are localized in $J$),
is unitarily equivalent to the braided product net, \ie 
$$\{\hat{\mathcal{B}}^{\pm}\}\cong\{\mathcal{B}_{\rho_{L}}\times^{\pm}_{\eps}\mathcal{B}_{\rho_{R}}\}.$$
\end{proposition} 

\begin{proof}
By Lemma \ref{lem:freealg}, Proposition \ref{prop:u1rel} and the relation $\iota_{L}\circ\pi_{L}=\iota_{R}\circ\pi_{R}$, we know that there exists a surjective homomorphism of *-algebras 
$$\phi:\mathfrak{B}^{L\times R}\rightarrow \tilde{\mathfrak{B}}$$
where $\mathfrak{B}^{L\times R}\subset \mathcal{B}_{\rho_{L}}\times^{\pm}_{\eps}\mathcal{B}_{\rho_{R}}$ is defined as in Lemma \ref{lem:freealg} and $\tilde{\mathfrak{B}}\subset\hat{\B}^{\pm}$ is the *-algebra generated by $\iota_{L}\circ\pi_{L}(\mathcal{A}_{U(1)})$ and $\iota_{L}(U)$, $\iota_{R}(V)$. By the GNS theorem for *-algebras, see \eg \cite[\Sec 1.3]{KhMo15}, in order to show that $\phi$ is implemented by a unitary it is enough to check that $\omega_0\circ E^{LR}=(\hat{\Omega}|\phi(\cdot)\hat{\Omega})$, where $\hat{\Omega}$ is the vacuum vector of $\{\hat{\B}^{\pm}\}$.
This is clear since, for $x=\sum_{i,j}\jmath^{L/R}(x_{i,j})\jmath^{L}(U^{i})\jmath^{R}(V^{j})\in\mathfrak{B}^{L\times R}$, we have $\omega_0\circ E^{LR}(x)=(\Omega_{\A_{U(1)}},x_{0,0}\Omega_{\A_{U(1)}})=(\hat{\Omega},\phi(\cdot)\hat{\Omega})$.
\end{proof}

By considering the braided product of a local BMT extension with itself (as concretely constructed in the previous proposition by taking $\rho_L = \rho_R = \rho$) we give examples where the center of the braided product is a continuous algebra, more specifically $L^{\infty}(\Sone,d\mu)$.

\begin{proposition}
Let $\{\B_{\rho}\}$ be the BMT extension obtained from a DHR automorphism $\rho$ and let $\{\B_{\rho}\times^{\pm}_{\eps}\B_{\rho}\}$ be the braided product extension with itself. Then $\Z(\B_{\rho}\times^{\pm}_{\eps}\B_{\rho})\cong L^{\infty}(\Sone, d\mu)$ with $d\mu$ the Lebesgue measure on the circle.
\end{proposition}

\begin{proof}
Recall that the center of the braided product is the same as the relative commutant 
$$\Z(\B_{\rho}\times^{\pm}_{\eps}\B_{\rho})=\Z((B_{\rho}\times^{\pm}_{\eps}\B_{\rho})(J))=(B_{\rho}\times^{\pm}_{\eps}B_{\rho})(J)^{\prime}\cap\iota(\mathcal{A}_{U(1)}(J))$$
for any proper bounded interval $J$ of $\RR$. Thus Lemma \ref{lem:exprelcomm} provides an expansion for elements $x\in \Z(\B_{\rho}\times^{\pm}_{\eps}\B_{\rho})$ 
$$x=\sum_{i\in\ZZ}V^{-i}U^{i}x_{i}$$
with $x_{i}\in \Hom_{\DHR\{\A_{U(1)}\}}(\id,\rho^{i} \rho^{-i} = \id)\cong \mathbb{C}$.

It is easy to see that there is an isomorphism between the *-algebra generated by the $\{U^{i}V^{-i}\}_{i}$ and the *-algebra generated by the characters of the circle. This same map is also an isomorphisms of pre-Hilbert spaces with inner product on one side induced by the vacuum state $\omega=\omega_{0}\circ E^{LR}$, where $E^{LR}$ is the standard expectation of the braided product net (Proposition \ref{prop:brexpectation}) and $\omega_{0}$ the vacuum state for $\{\mathcal{A}_{U(1)}\}$, and on the other side the usual $L^{2}(\Sone,d\mu)$ inner product.

Thus let $\mathfrak{B}$ denote the *-algebra generated by the $\{U^{i}V^{-i}\}_{i}$, $\bar{\mathfrak{B}}^{\left\|\cdot\right\|_{\omega}}$ its Hilbert completion and let $Char(\Sone)$ be the *-algebra generated by characters of the circle. $\bar{\mathfrak{B}}^{\left\|\cdot\right\|_{\omega}}\cong L^{2}(\Sone,d\mu)$ as Hilbert spaces and let $W$ be the unitary which implements the isomorphism. If $\pi_{\omega}$ is the GNS representation of $\mathfrak{B}$ induced by the state $\omega=\omega_{0}\circ E$ (on the Hilbert space $\bar{\mathfrak{B}}^{\left\|\cdot\right\|_{\omega}}$), and if $\pi_{d\mu}$ is the GNS representation of $Char(\Sone)$, we have $\Ad_W \pi_{\omega}=\pi_{d\mu}$. Hence the isomorphism extends to the ultraweak closure, and
$$\Z(B_{\rho}\times^{\pm}_{\eps}B_{\rho})(I)\cong \pi_{\omega}(\Z(B_{\rho}\times^{\pm}_{\eps}B_{\rho}))=\pi_{\omega}(\mathfrak{B})''\cong \pi_{d\mu}(Char(\Sone))''\cong L^{\infty}(\Sone,d\mu)$$
concluding the proof.
\end{proof}

We thus have an example of an uncountable family of (one-dimensional) irreducible phase boundaries, parametrized by $\Sone$, obtained from the braided product construction. This is obviously in contrast with the finite index case, where the relative commutant is necessarily finite-dimensional. But the difference from the finite index case is actually greater than this: we have an example where the relative commutant is not a discrete algebra. This means that the disintegration in Proposition \ref{prop:defects} that yields irreducible phase boundaries is not a direct sum. Moreover it is not true, in contrast with the finite index case, that every irreducible phase boundary condition comes from a representation of the braided product extension, see \cite[\Prop 5.1, \Cor 5.3]{BKLR16}, due to the absence of non-trivial minimal central projections.

Similarly, one can construct examples where the braided product is itself an irreducible extension and thus it yields a unique irreducible phase boundary. It is not hard to see that this is the case for the braided product of two local BMT extensions of the $U(1)$-current whose generating DHR automorphisms $\rho_{f_1}$, $\rho_{f_2}$ have charges $\int_{\Sone}f_1$ and $\int_{\Sone}f_2$ with irrational quotient. The claim simply follows from the expansion given in Lemma \ref{lem:exprelcomm} and observing that, in this case, the dual canonical endomorphisms of the BMT extensions $\theta_{1}$ and $\theta_{2}$ have only one irreducible subendomorphism in common: the identity.

%%%
\section{Conclusions}
%%%

Index theory provides an elegant and effective machinery to classify and construct extensions of von Neumann algebras and local nets. When this framework is not fully applicable (infinite index case), we have seen that under some physically meaningful structural hypotheses (semidiscreteness, discreteness) some of these results can be suitably generalized. The price to pay is abandoning the purely categorical setting of finite index Q-systems by the emergence of analytical conditions. At the same time, these analytical conditions (convergence of projections, faithfulness of expectations) provide a way to control infinite objects (gauge groups, representation categories, sets of generating fields) exploiting techniques of Operator Algebras in their application to QFT.  

In particular, we have introduced the notion of generalized Q-system of intertwiners (in the category of localizable superselection sectors $\DHR\{\A\}$) for a local net $\{\mathcal{A}\}$, and we have shown that from this data a net extension of $\{\mathcal{A}\}$, in the spirit of \cite{LoRe95}, can be constructed. At the level of properly infinite inclusions, we have seen that the existence of generalized Q-systems of intertwiners is equivalent to the inclusion to be of discrete type. When passing from subfactors to inclusions of local nets as in \cite{LoRe95} this matter is more subtle, and we provided sufficient conditions to guarantee the existence of generalized Q-systems of intertwiners for nets, which cover most interesting examples in low and higher spacetime dimensions. We leave open the question on whether these conditions are always verified by discrete QFT extensions. 

The notion of generalized Q-system of intertwiners lends itself to generalize the definition of braided product between ordinary Q-systems. After proving that the analytic properties of generalized Q-systems of intertwiners turn out to be compatible with the purely algebraic definition of the braided product, we explore some properties of the resulting net extension, showing that it retains some features of its finite index counterpart. In particular, in the case of chiral CFTs, we have seen that its central decomposition can yield uncountable families of irreducible phase boundaries with infinite index. An important issue left open is the classification of all phase boundary conditions among two CFTs. In particular, one would like to understand if, in analogy with \cite{BKLR16}, all the boundary conditions arise in the disintegration of the center of the braided product.

Although the discrete case covers many physical examples, \eg, every orbifold construction by a compact group, the setting of greatest generality for irreducible inclusions of local CFTs (at least assuming the existence of a vacuum vector) is semidiscreteness. Generalized Q-systems do always exist for semidiscrete extensions of properly infinite von Neumann algebras \cite{FiIs99}. An issue that would be worth analyzing further is if methods similar to those explored in this paper can be generalized to treat extensions of local nets which are semidiscrete but not discrete \cite{Car04}, \cite{Xu05}. It would also be interesting to extend the analysis of discrete inclusions to the case of non-separable Hilbert spaces, given that good candidates for such extensions in QFT already appear in \cite{Cio09}, \cite{MTW16}. Lastly, we mention that one can easily construct discrete non-finite local extensions which are not compact group orbifolds by taking tensor products of local nets, \footnote{We thank Y. Tanimoto for pointing out this interesting fact.}. It would also be worth investigating which kind of extensions can arise from braided products of compact group orbifolds, given that, by the arguments of our last section, one can construct extensions whose generating fields have the commutation relations of non-commutative tori.

\bigskip

{\bf Acknowledgements.}
Supported by the European Research Council (ERC) through the Advanced Grant QUEST \lq\lq Quantum Algebraic Structures and Models", and by PRIN-MIUR. We are indebted to R. Longo for proposing us the problem investigated in this work, and to M. Bischoff and K.-H. Rehren for many discussions and suggestions, and for their motivating interest. We also thank I. Khavkine and Y. Tanimoto for useful comments and criticism.

L.G. wishes to thank K.-H. Rehren for an invitation to G\"ottingen (Institut f\"ur Theoretische Physik, Georg-August-Universit\"at) and W. Yuan for an invitation to Beijing (Academy of Mathematics and Systems Science, Chinese Academy of Sciences), where this work has been presented, and thanks them for hospitality and for useful conversations in both these occasions.

%\newpage
\small
%%%NOWALLINTEX%%%\bibliography{mybib}%%%Links to mybib.bib same folder

\newcommand{\etalchar}[1]{$^{#1}$}
\def\cprime{$'$}

\end{document}